\documentclass[10pt,a4paper,reqno]{amsart} 

\usepackage{amsmath}
\usepackage{bbm}
\usepackage{mathtools}
\usepackage{amssymb}
\usepackage{graphicx}
\usepackage{color}
\usepackage{latexsym}
\usepackage[utf8]{inputenc}
\usepackage[T1]{fontenc}
\usepackage{comment}
\usepackage{stmaryrd}
\usepackage{amssymb} 
\usepackage[mathcal]{eucal}
\usepackage{fancybox}
\usepackage{xcolor}

\DeclareMathOperator{\Arg}{Arg}
\DeclareMathOperator{\Ima}{Im}
\DeclareMathOperator{\Rea}{Re}
\DeclareMathOperator{\inte}{int}
\DeclareMathOperator{\ext}{ext}
\newcommand\mps[1]{\marginpar{\small\sf#1}}

\newcommand{\GK}{\mathbb{K}}

\newcommand{\cE}{\mathcal E}
\newcommand{\cF}{\mathcal F}
\newcommand{\R}{\mathbb R}
\newcommand{\Q}{\mathbb Q}
\newcommand{\C}{\mathbb C}
\newcommand{\N}{\mathbb N}

\newcommand{\Rc}{\mathcal{R}}

\newcommand{\pd}[2]{\frac{\partial#1}{\partial#2}}

\newtheorem{prop}{Proposition}[section]

\newtheorem{lem}[prop]{Lemma}

\newtheorem{defi}[prop]{Definition}

\newtheorem{theo}[prop]{Theorem}

\newtheorem{cor}[prop]{Corollary}

\def\emm#1,{{\em #1}}
\newcommand{\beq}{\begin{equation}}
\newcommand{\eeq}{\end{equation}}

\newcommand{\gf}{generating function}
\newcommand{\gfs}{generating functions}

\newcommand{\E}{\mathbb{E}}

\newcommand{\PP}{\mathbb P}

\newcommand{\vareps} {\varepsilon}

\newcommand{\qs}{{\mathbb Q}}  
\newcommand{\cs}{{\mathbb C}} 

\DeclareMathOperator{\cc}{c}
\DeclareMathOperator{\vv}{v}
\DeclareMathOperator{\ff}{f}
\DeclareMathOperator{\ee}{e}

\DeclareMathOperator{\Tpol}{T}
\DeclareMathOperator{\Ppol}{P}

\newcommand{\Dc}{\delta}
\newcommand{\bu}{\bar u}
\newcommand{\bz}{\bar z}
\newcommand{\rt}{\tilde \rho}
\newcommand{\rh}{\hat \rho}
\catcode`\@=11
\def\section{\@startsection{section}{1}%
 \z@{.7\linespacing\@plus\linespacing}{.5\linespacing}%
 {\normalfont\bfseries\scshape\centering}}

\def\subsection{\@startsection{subsection}{2}%
  \z@{.5\linespacing\@plus\linespacing}{.5\linespacing}%
  {\normalfont\bfseries\scshape}}

\def\subsubsection{\@startsection{subsubsection}{3}%
 \z@{.5\linespacing\@plus\linespacing}{-.5em}
  {\normalfont\bfseries\itshape}}
\catcode`\@=12

\title{Spanning forests in regular planar maps} 
\author[M. Bousquet-M\'elou]{Mireille Bousquet-M\'elou}
\author[J. Courtiel]{Julien Courtiel}

\thanks{Both authors were supported by  the French ``Agence Nationale
de la Recherche'', project A3 ANR-08-BLAN-0190. MBM also acknowledges
the hospitality of the Institute of Computer Science, 
Universit\"at Leipzig, where part of this work was carried out.}

\address{MBM \& JC:  CNRS, LaBRI, UMR 5800, Universit\'e de Bordeaux, 
351 cours de la Lib\'eration, 33405 Talence Cedex, France}
\email{bousquet@labri.fr, jcourtie@labri.fr}

\begin{document}
\begin{abstract}
We address the enumeration of $p$-valent planar maps equipped with a
  spanning forest, with a weight $z$ per face and a weight $u$ per connected
  component of the forest. Equivalently, we count $p$-valent maps
  equipped with a spanning \emm tree,, with a weight $z$ per face and a
  weight $\mu:=u+1$ per \emm internally active, edge, in the sense of
  Tutte; or  the (dual) $p$-angulations equipped with a \emm recurrent 
  sandpile configuration,, with a weight  $z$ per vertex and a
  variable $\mu:=u+1$ that keeps track of the \emm level, of the configuration.
 This enumeration problem also corresponds to
  the limit $q\rightarrow 0$ of the $q$-state Potts model on
   $p$-angulations.

Our approach is purely combinatorial. The associated \gf, denoted $F(z,u)$, is expressed in terms of
a pair of series defined  implicitly by a system involving doubly
hypergeometric series.  We derive from this system that $F(z,u)$ is
\emm differentially algebraic, in $z$, that is, satisfies a differential
equation in $z$ with polynomial coefficients in $z$ and $u$. This has
recently been proved to hold for the more general Potts model on 3-valent
maps, but via a much more involved and less combinatorial proof.

For $u\ge -1$, we study the singularities of $F(z,u)$ and the
corresponding asymptotic behaviour of its $n$th coefficient. For
$u>0$, we find the standard asymptotic behaviour of planar maps, with a
subexponential term in $n^{-5/2}$. At $u=0$ we witness a phase
transition with a term $n^{-3}$. When $u\in[-1,0)$, we obtain an
  extremely unusual behaviour in $n^{-3}(\ln n)^{-2}$. To our
  knowledge, this is a new ``universality class'' for planar maps.
\end{abstract}
\maketitle

\section{Introduction}
A \emm planar map, is a proper embedding of a connected graph in the
sphere. The enumeration of planar maps has received a continuous
attention in the past 60 years, first in combinatorics with the
pionneering work of Tutte~\cite{tutte-census-maps}, then  in
theoretical physics~\cite{BIPZ}, where maps are considered as random surfaces
modelling the effect of \emm quantum gravity,, and more recently in
probability theory~\cite{le-gall-miermont,marckert-miermont}.
General planar maps have been studied, as well as sub-families
obtained by imposing  constraints of higher connectivity, or
prescribing the degrees of vertices or faces
({\it e.g.}, triangulations). Precise definitions are given below.

Several robust enumeration methods have been designed, from Tutte's recursive
approach (e.g.~\cite{tutte-triangulations}), which leads to functional equations for the \gfs\ of maps, to
the beautiful bijections initiated by Schaeffer~\cite{Sch97},
and further developed by physicists and combinatorics alike 
~\cite{bernardi-fusy,BDG-blocked}, via the
powerful approach based on matrix integrals~\cite{DFGZJ}. See for
instance~\cite{mbm-survey} for a more complete (though non-exhaustive) bibliography.

Beyond the enumerative and asymptotic properties of  planar maps,
which are now 
well understood,  the attention has also focussed on two more
general questions: maps on higher genus
surfaces~\cite{bender-surface-I,chapuy-marcus-schaeffer}, 
and maps equipped
with an additional structure. The latter question is particularly
relevant in physics, where a surface on which nothing happens (``pure
gravity'') is of little interest. For instance, one has studied maps
equipped with a polymer~\cite{DK88}, with an Ising
model~\cite{BDG-blocked,Ka86,BK87,mbm-schaeffer-ising} or more generally a Potts model, with a proper
colouring~\cite{lambda12,tutte-differential}, with loops
models~\cite{borot1,borot2}, with a spanning tree~\cite{mullin-boisees}, or percolation on planar maps~\cite{angel-perco,bernardi-perco}.

In particular, several papers have been devoted in the
past 20 years to the
study of the Potts model on families of planar maps~\cite{baxter-dichromatic,borot3,daul,eynard-bonnet-potts,guionnet-jones,zinn-justin-dilute-potts}. In combinatorial
terms, this means counting maps equipped with a colouring in $q$
colours, according to the size (\emm e.g.,, the number of edges) and the
number of \emm monochromatic edges, (edges whose endpoints have the
same colour). Up to a change of variables, this also means counting
maps weighted by their \emm Tutte polynomial,, a bivariate
combinatorial invariant which has numerous interesting
specializations. By generalizing  Tutte's formidable solution of
properly coloured triangulations (1973-1982), it has recently been
proved that the Potts \gf\ is \emm differentially algebraic,, that is,
satisfies a (non-linear) differential equation\footnote{with respect
  to the size variable} with polynomial
coefficients~\cite{bernardi-mbm,bernardi-mbm-de,mbm-survey}. This
holds at least for general planar maps and 
for triangulations (or dualy, for cubic maps).

The method that yields these differential equations is extremely
involved, and does not shed much light on the structure of $q$-coloured
maps. Moreover, one has not been able, so far, to derive from these
equations the asymptotic behaviour of the number of coloured maps, nor
the location of phase transitions.

The aim of this paper is to remedy these problems --- so far for a
one-variable specialization of the Tutte polynomial. This
specialization is obtained by setting to $1$ one of the variables, or by taking (in an adequate way) the limit 
$q\rightarrow 0$ in the Potts model. Combinatorially, we are simply
counting maps (in this paper, $p$-valent maps) equipped with a
spanning forest. We call them \emm forested maps,. This problem has
already been studied in~\cite{sportiello} via a random matrix
approach, but with no explicit solution. The \gf\ $F(z,u)$ that we obtain
keeps track of the \emm size, of the map (the number of faces;
variable $z$) and of the number of trees in the
forest (minus one; variable $u$). The specialization $u=0$ thus
counts maps equipped with a spanning \emm tree, and was determined  a
long time ago by Mullin~\cite{mullin-boisees}.  

\medskip
Here is  an outline of the paper. We begin in Section~\ref{sec:prelim} with general definitions on maps,
and on the Tutte polynomial. We recall some of its combinatorial
descriptions, and underline in particular that the series $F(z,
\mu-1)$, once expanded in powers of $z$ and
$\mu$, has non-negative coefficients and admits several
combinatorial interpretations. This important observation implies that the natural domain
of the parameter $u$ is $[-1, +\infty)$ rather than $[0, +\infty)$. 
  In Section~\ref{sec:eq}, we obtain in a purely combinatorial
manner an expression of $F(z,u)$ in terms of the solution of a  system
of two functional equations. 
In Section~\ref{sec:de} we  derive from this system that $F(z,u)$ is
differentially algebraic in $z$, and give  explicit differential equations
for cubic  ($p=3$) and 4-valent ($p=4$) maps. Section~\ref{sec:u+1} is a
combinatorial interlude explaining why all series occurring in our
 equations, like $F(z,u)$ itself,  still have non-negative
coefficients when $u \in [-1,0]$. 

The rest of the paper is devoted to asymptotic results, still for
$p=3$ and $p=4$:
when $u>0$, 
forested maps follow the standard asymptotic  
behaviour of planar maps ($ \mu^n n^{-5/2}$) but then there is a phase
transition at $u=0$ (where one counts maps equipped with a spanning
tree), and a very unusual asymptotic behaviour in $ \mu^n n^{-3}(\ln
n)^{-2}$ holds when $u\in[-1,0)$. To our knowledge, this is
  the first time a class of planar maps exhibits this asymptotic
  behaviour. This proves in particular
that $F(z,u)$ is not \emm D-finite,, that is, does not satisfy any
\emm linear, differential equation in 
$z$ for these values of $u$ (nor for a generic value of $u$). This is
in contrast with the case $u=0$, for which the \gf\ of maps equipped
with a spanning forest is known to be D-finite.

Our key tool is the \emm singularity analysis,
of~\cite{flajolet-sedgewick}: its basic principle is to derive the
asymptotic behaviour of 
the coefficients of a series $F(z)$ from the  singular behaviour of $F$
 near its \emm dominant singularities, (\emm i.e.,, singularities of
minimal modulus). The first case we study (4-valent maps with $u>0$)
is simple: first, one of the two series involved in our system vanishes; the remaining one, denoted $R$, satisfies  an inversion equation $\Omega(R(z))=z$
for which the (unique) dominant singularity $\rho$ of $R$
is such that $R(\rho)$ lies in the domain
of analyticity of $\Omega$. 
One obtains for $R$ a ``standard'' square root singularity. This is
well understood and almost routine. Two ingredients make the other
cases significantly harder: 
\begin{itemize}
\item  when $u<0$,  $R(\rho)$  is a singularity of $\Omega$,
\item when $p=3$ (cubic maps) we have to deal with a system of two
  equations;  the analysis of systems  is 
  delicate, even in the so-called \emm positive case,, which corresponds in
  our context to $u>0$  (see~\cite{drmota-systems,burris}).
\end{itemize}
These difficulties, which culminate when $p=3$ and $u<0$, are addressed in  Sections~\ref{sec:general}
and~\ref{sec:inversion}. Section~\ref{sec:general} establishes general results on
 implicitly defined series. Section~\ref{sec:inversion} focusses on
 the inversion equation $\Omega(R(z))=z$ in the case where (up to
 translation) $\Omega$ has a  $z \ln z$ singularity  at $0$. One
 then  applies these results to the asymptotic analysis of forested maps in
Sections~\ref{sec:asympt-4} (4-valent maps) and~\ref{sec:asympt-3} (cubic maps). Section~
\ref{sec:random} exploits the results of Section~\ref{sec:asympt-4} to
study some properties of large \emm random,  maps equipped with a spanning
forest or a spanning tree. 

We conclude in Section~\ref{sec:final} with a few comments.

\section{Preliminaries}
\label{sec:prelim}
\subsection{Planar maps}
A \textit{planar map} is a proper embedding of a connected graph (possibly
with loops and multiple edges) in the oriented sphere, considered up
to continuous deformation. All maps in this paper are planar, and we
often omit the term ``planar''. A \textit{face} is a (topological)
connected component of the complement of the embedded graph. Each edge
consists of two  \textit{half-edges}, each incident to
an endpoint of the edge. A \emph{corner} is an ordered pair  $(e_1,e_2)$
of
 half-edges incident to the same vertex, such that $e_2$ immediately
 follows $e_1$ in counterclockwise order. The
\emph{degree} of a vertex or a face is the number of corners incident
to it. A vertex of degree $p$ is called
\emm $p$-valent,. One-valent vertices are also called \emm leaves,. A
map is \emph{$p$-valent} if all vertices are 
$p$-valent. A \emph{rooted map} is a map with a marked corner
$(e_1,e_2)$, called the 
\emm root, and indicated by an arrow in our figures. The \emph{root vertex} is the vertex incident to the root. The
\emph{root half-edge}  is $e_2$ and the \emm root edge, is the edge
supporting $e_2$.  This way of rooting maps is equivalent to the more
standard
 way where one  marks the root edge and orients it from
$e_2$ to its other half-edge. All maps of the paper are rooted, and we
often omit the term ``rooted''.
The \emph{dual} of a
map $M$, denoted $M^*$, is the map obtained by placing a 
vertex of $M^*$ in each face of $M$ and an edge of $M^*$ across each
edge of $M$; see Figure~\ref{fig:dual}(a). The dual of a
$p$-valent map  is a map with all faces of degree $p$, also called
\emm $p$-angulation,.

\begin{figure}[thb]
(a)\ \begin{tabular}{c}{\includegraphics[scale=0.8]{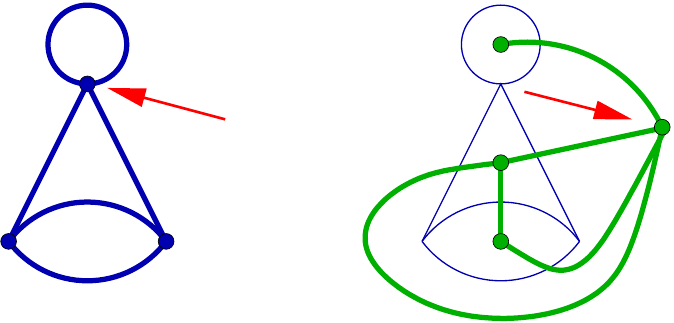}}\end{tabular}
\hskip 15mm \ \ \
(b)\ \begin{tabular}{c}{\includegraphics[scale=1.2]{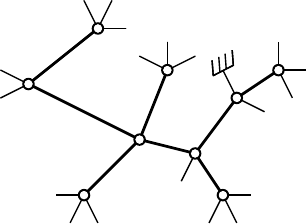}}\end{tabular}
\caption{ (a) A rooted planar map and its dual (rooted at
  the dual corner). (b) A 4-valent leaf-rooted tree.}
\label{fig:dual}
\end{figure}

 A \emph{(plane) tree} is a planar map with a unique face.
A tree is $p$-valent if all non-leaf vertices have degree $p$.
We  consider  the edges leading to the leaves as half-edges, as
suggested by Figure~\ref{fig:dual}(b). 
A \emph{leaf-rooted} tree (resp. \emph{corner-rooted}) is a tree
with a marked leaf (resp. corner). The number of $p$-valent
leaf-rooted (resp. corner-rooted) trees with $k$ leaves is denoted by
$t_k$ (resp. $t^c_k$) (the notation should be $t_{k,p}$ and
$t_{k,p}^c$, but we consider $p$ as a fixed integer, $p\ge 3$). These numbers are
well-known~\cite[Thm.~5.3.10]{stanley-vol-2}: they are $0$ unless  $k =
(p-2)\ell+2$ with $\ell\ge 1$, and in this case,
\beq
\label{deftk}
 t_k =  \frac {((p-1)\ell)!}
     {\ell!((p-2)\ell+1)!}
\quad\quad \hbox{and}\quad\quad
t^c_k =  p \frac {((p-1)\ell)!}{(\ell-1)!((p-2)\ell+2)!}.
\eeq

Let $M$ be a rooted planar map with vertex set $V$. A \emph{spanning forest}
of $M$ is a graph $F=(V,E)$
where $E$ is a subset of edges of $M$ forming no cycle. Each connected
component of $F$ is a tree, and  the \textit{root component} is
the tree containing the root vertex.
We say that the pair $(M,F)$ is a  \emph{forested map}.
We denote by $F(z,u)$ the generating function of
{$p$-valent forested maps}, counted by  faces
(variable $z$) and  non-root components (variable  $u$):
\beq\label{F-def}
F(z,u)=\sum_{M \ p-{\rm{\small valent}} \atop F\ {\rm{\small{spanning
        \ forest}}}}
z^{\ff(M)} u^{\cc(F)-1},
\eeq
where  $\ff(.)$ denotes the number of faces and $\cc(.)$  the number of components.
When $p=3$,
\beq\label{F-3}
F(z,u) = \left( 6 + 4\,u \right) {z}^{3}+ \left( 140 + 234\,u +
  144\,{u}^{2}+32\,{u}^{3} \right) {z}^{4} +  O(z^5).
\eeq
The coefficient $(6+4u)$ means that there are 10 trivalent (or \emm cubic,)
forested maps with 3 faces: 6 in which the forest is a tree, and 4 in
which it has two components (Figure~\ref{fig:tenmaps}).

\begin{figure}[thb]
{\includegraphics[scale=1.2]{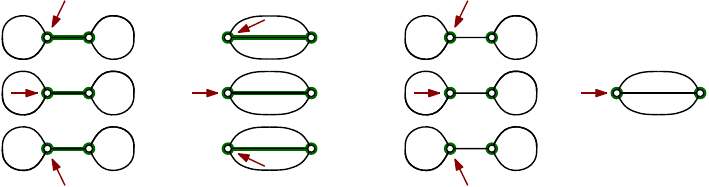}}
\caption{The 10 forested cubic maps with 3 faces.}
\label{fig:tenmaps}
\end{figure}

\subsection{Forest counting,  the Tutte polynomial, and related models}
\label{sec:tutte}
Let $G=(V,E)$ be a graph with vertex set $V$ and edge set $E$.
The \emm Tutte polynomial,
of $G$ is the following polynomial in two indeterminates
(see \emm e.g., \cite{Bollobas:Tutte-poly}):
\beq\label{Tutte-def}
\Tpol_G(\mu,\nu):=\sum_{S\subseteq
  E}(\mu-1)^{\cc(S)-\cc(G)}(\nu-1)^{\ee(S)+\cc(S)-\vv(G)},
\eeq
where the sum is over all spanning subgraphs of $G$ (equivalently,
over all subsets $S$ of edges) and  $\vv(.)$, $\ee(.)$ and $\cc(.)$ denote
respectively the number of vertices, edges and connected
components. 
The quantity $\ee(S)+\cc(S)-\vv(G)$ is the \emm cyclomatic
number of $S$,, that is, the minimal number of edges one has to delete
from $S$ to obtain a forest.

When $\nu=1$,  the only subgraphs that contribute to~\eqref{Tutte-def}
are the forests.
Hence the  \gf\ of forested maps defined by~\eqref{F-def} can
be written as
\beq\label{F-Tutte}
F(z,u)=\sum_{M \ p-\hbox{\small valent}} z^{\ff(M)} \Tpol_M(u+1,1).
\eeq
Note that we write
$\Tpol_M$ although the value of the Tutte polynomial only depends on
the underlying graph of $M$, not on the embedding.

Even though this is not clear from~\eqref{Tutte-def}, the 
polynomial $\Tpol_G(\mu,\nu)$ has non-negative coefficients in $\mu$ and
$\nu$. This was proved combinatorially by
Tutte~\cite{tutte-dichromate}, who showed that
$\Tpol_G(\mu,\nu)$ counts spanning \emm trees, of $G$ according to two
parameters, called \emm internal, and \emm external, activities. Other
combinatorial descriptions of $\Tpol_G(\mu,\nu)$, in terms of other
notions of activity, were given later. Let
us present the one due to Bernardi, which is
nicely related to maps~\cite{bernardi-tutte}. Following Mullin~\cite{mullin-boisees}, we call \emm
tree-rooted map, a map  $M$ equipped with a spanning tree $T$. Walking around
$T$ in counter-clockwise order, starting from the root, defines a
total order on the edges: the first edge that is met
is the smallest one, and so on (Figure~\ref{fig:treetour}). An edge $e$ is \emm internally active, if
it belongs to $T$ and is minimal in its cocycle; that is, all the
edges $e'\not = e$ such that $\left(T\setminus\{e\}\right) \cup \{e'\}$ is a tree
are larger than $e$. It is \emm externally active, if
it does not belong to $T$ and is minimal in the cycle created by
adding $e$ to $T$. Denoting by $\inte(M,T)$ and $\ext(M,T)$ the numbers
of internally and externally edges, one has:
$$
\Tpol_M(\mu,\nu)=\sum_{T\ \small{\rm{spanning\  tree}}} \mu^{\inte(M,T)}  \nu^{\ext(M,T)}.
$$
A non-obvious property of this description is that it
only depends on the underlying graph of~$M$. 

\begin{figure}[thb]
{\includegraphics{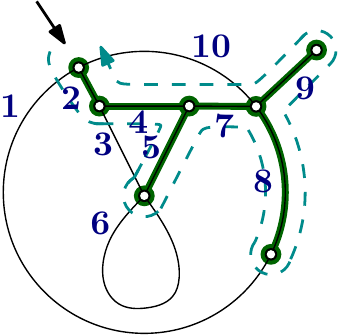}}
\caption{The edges of a tree-rooted map are naturally order by walking
  around the tree. The active edges are those labelled 1,  3, 6 and 9.}
\label{fig:treetour}
\end{figure}

Returning to~\eqref{F-Tutte}, we thus obtain a second description of $F(z,u)$:
\beq\label{F-act}
F(z,u)=\sum_{M \ p-{\rm{\small valent}} \atop T\ \small{\rm{spanning\  tree}}} z^{\ff(M)} 
(u+1)^{\inte(M,T)}.
\eeq
In particular, it makes
sense combinatorially to write $u=\mu-1$ and take $u\in[-1,\infty)$.

 We now give four more descriptions of $F(z,u)$ in terms of the dual
 $p$-angulations.
 For any planar map $M$, it is known that
$$
\Tpol_{M^*}(\mu,\nu)=\Tpol_M(\nu,\mu).
$$
Since 
$$
\Tpol_M(1,\nu)= \sum_{S \subset E, \, S \, \rm{connected}}
(\nu-1)^{\ee(S)+\cc(S)-\vv(M)},
$$
 we first derive from~\eqref{F-Tutte} that
 \begin{eqnarray}
   F(z,u)&= &\sum_{M \ p-\rm{\small{angulation}}} z^{\vv(M)}
   \Tpol_M(1,u+1)
\label{F-dual-Tutte}
\\
&=& \sum_{M \ p-\rm{\small{angulation}} \atop
  S\ \small{\rm{connected \ subgraph}}} z^{\vv(M)} 
u^{\ee(S)+\cc(S)-\vv(M)}
\nonumber 
 \end{eqnarray}
counts $p$-angulations $M$ equipped with a connected (spanning)
subgraph $S$,
by the vertex number of $M$ and the cyclomatic number of $S$.
Also, the ``dual'' expression of~\eqref{F-act} reads
\beq\label{F-dual}
F(z,u)=\sum_{M \ p-{\rm{\small angulation}} \atop T\ \small{\rm{spanning\  tree}}} z^{\vv(M)} 
(u+1)^{\ext(M,T)}.
\eeq

Our next interpretation of $F(z,u)$, which we will not entirely
detail, relies on the connection between $\Tpol_M(1,\nu)$ and the
\emm recurrent, (or: \emm critical,) configurations of the sandpile
model on $M$. It is known~\cite{merino,cori-borgne} that 
$$
\Tpol_M(1,\nu) = \sum_{C \ \rm{\small recurrent}} \nu^{\ell(C)},
$$
where the sum runs over all recurrent configurations
$C$, and  $\ell (C)$ is the \emm level, of $C$. Hence
\beq\label{F-sandpile}
 F(z,u)= \sum_{M \ p-{\rm{\small{angulation}}} \atop C \ \rm{recurrent} } z^{\vv(M)}
  (u+1)^{\ell(C)}
\eeq
also counts $p$-angulations $M$ equipped with a recurrent
configuration $C$ of the sandpile model, by the vertex number of $M$
and the level of $C$.

Our final interpretation is in terms of  the \emm Potts, model.
Take $q\in \N$. 
A \emm $q$-colouring, of the vertices of $G=(V,E)$ is a map $c : V
\rightarrow \{1, \ldots, q\}$. An edge of $G$ is \emm monochromatic,
if its endpoints share the same colour. Every loop is thus
monochromatic. The number of monochromatic edges is denoted by $m(c)$.
The \emm partition function of the  Potts model,
on $G$ counts colourings by the number of monochromatic edges:
$$
\Ppol_G(q, \nu)= \sum_{c  : V\rightarrow \{1, \ldots, q\}}
\nu^{m(c)}.
$$
The Potts model is a classical magnetism model in statistical physics, which
includes (for $q=2$) the famous Ising model (with no magnetic
field)~\cite{welsh-merino}. Of course, $\Ppol_G(q,0)$ is the chromatic
polynomial of $G$. More generally,  it is not hard to see that 
  $\Ppol_{G}(q,\nu)$ is
always a \emm polynomial, in $q$ and $\nu$, and a multiple of $q$. Let
us define the \emm reduced Potts polynomial,
$\tilde \Ppol_G(q,\nu)$ by
$$
 \Ppol_G(q,\nu)= q\, \tilde \Ppol_G(q,\nu).
$$
Fortuin and Kasteleyn established the 
  equivalence of $\tilde \Ppol_G$ 
with the Tutte polynomial~\cite{Fortuin:Tutte=Potts}: 
\begin{eqnarray*}
\tilde \Ppol_G(q,\nu)~=~\sum_{S\subseteq E(G)}q^{\cc(S)-1}(\nu-1)^{\ee(S)}~=~(\mu-1)^{\cc(G)-1}(\nu-1)^{\vv(G)-1}\,\Tpol_G(\mu,\nu),
\end{eqnarray*}
for $q=(\mu-1)(\nu-1)$. Setting $\mu=1$, we obtain, for a \emm connected,
graph $G$
$$
\tilde \Ppol_G(0,\nu)~=(\nu-1)^{\vv(G)-1}\,\Tpol_G(1,\nu).
$$
Returning to~\eqref{F-dual-Tutte} finally gives
\beq\label{F-Potts}
F(z,u)= u\sum_{M \ p-\rm{\small{angulation}} }(z/u)^{\vv(M)}\, \tilde
\Ppol_M(0,u+1).
\eeq

\subsection{Formal power series}

Let $A=A(z) \in \mathbb K[[z]]$ be a power series in  one variable with
coefficients in a field $\GK$. We say that $A$ is \emm D-finite, if it satisfies
a (non-trivial) linear differential equation with coefficients in
$\GK[z]$ (the ring of polynomials in $z$). More generally, it is \emph{D-algebraic}
if there exist a positive
integer $k$ and a non-trivial polynomial $P \in \mathbb
K[z,x_0,\dots,x_k]$ such that $P\left(z,A,\frac {\partial A} {\partial
    z},\dots,\frac {\partial^k A} {\partial z^k}\right) =
0$. 

A $k$-variate power series $A=A(z_1,\ldots, z_k)$ with coefficients in
$\GK$ is \emph{D-finite} if its partial derivatives (of all orders) span a
finite dimensional vector space over 
$\mathbb{K}(z_1,\ldots, z_k)$.


\section{Generating functions for forested maps}
\label{sec:eq}
Fix $p\ge 3$. 
We give here a system of equations that defines the \gf\
$F(z,u)$ of $p$-valent forested maps, or, more precisely, the series $zF'_z(z,u)$ that counts forested maps with a
marked face. We also give 
 simpler systems for two variants of
$F(z,u)$,  involving no derivative.

\subsection{$p$-Valent  maps}

\begin{theo} \label{thm:equations} 
Let $\theta$, $\Phi_1$ and $\Phi_2$
  be the following doubly hypergeometric series:
$$
\theta(x,y)  = \sum_{i \geq 0} \sum_{j \geq 0} t^c_{2i+j} {2i + j
  \choose i,i,j} x^i y^j,  
$$
\beq\label{phi}
\Phi_1(x,y) = \sum_{i\geq 1} \sum_{j \geq 0} t_{2i+j} { 2i + j - 1 \choose i - 1,i,j} x^i y^j, \ \ \ \ 
\Phi_2(x,y) = \sum_{i\geq 0} \sum_{j \geq 0} t_{2i+j+1} { 2i + j  \choose i,i,j} x^i y^j ,
\eeq
where $t_k$ and  $t^c_k$ 
are given by~\eqref{deftk} and ${a+b+c} \choose {a,b,c}$ denotes the trinomial
coefficient $(a+b+c)!/(a!b!c!)$.

There exists a unique pair $(R,S)$ 
of power series in $z$ with constant term $0$ and coefficients in $\qs[u]$ that
 satisfy
\begin{eqnarray}
 R &=& z + u\,\Phi_1(R,S),\label{defR}  \\ 
S &=& u\,\Phi_2(R,S).\label{defS} 
\end{eqnarray}

The generating function $F(z,u)$ of
  $p$-valent forested maps is characterized by 
  $F(0,u)=0$ and 
\beq\label{frs}
F'_z(z,u) = \theta(R,S).
\eeq
\end{theo}
\noindent{\bf Remarks}\\
1. These equations allow us to compute the first terms in the
expansion of $F(z,u)$, for any fixed $p\ge 3$. This is how we
obtained~\eqref{F-3}.\\ 
2. When $p$ is even, then $t_{2i+1}=0$ for all $i$. In particular, all
terms occurring in the definition~\eqref{phi} of $\Phi_2$ are multiples
of $y$, so that  $S=0$.  The simplified system reads:
\beq\label{system-simple}
F'_z(z,u) = \theta(R) \quad \quad \hbox{and}\quad \quad  R = z + u\,\Phi(R),
\eeq
with 
$$
\theta(x)  = \sum_{i \geq 0}  t^c_{2i} {2i 
  \choose i} x^i \quad \quad  \hbox{and}\quad \quad
\Phi(x) = \sum_{i\geq 1}  t_{2i} { 2i  - 1 \choose i} x^i .
$$
\\
3. When $u=0$, an even more drastic simplification follows from~(\ref{defR}-\ref{defS}): not only
$S=0$, but also $R=z$, so that~\eqref{frs} becomes an explicit
expression of $F'_z$:
$$
F'_z(z,0) =\sum_{i \geq 0}  t^c_{2i} {2i   \choose i}z^i,
$$
or equivalently,
\beq\label{Fz0}
F(z,0) =\sum_{i \geq 0}  t^c_{2i} {2i  \choose i}\frac{z^{i+1}}{i+1}
=\sum_{\ell \ge 1}\frac{p((p-1)\ell)!}{(\ell-1)! (1+(p-2)\ell/2)!
  (2+(p-2)\ell/2)!} z^{2+(p-2)\ell/2},
\eeq
where we require $\ell $ to be even if $p$ is odd.
This series counts $p$-valent maps equipped with a spanning tree, and
this expression was already proved by Mullin~\cite{mullin-boisees}.
\\
4. The series  $\theta$ and  $\Phi_i$  are explicited when $p=4$ and
$p=3$ in Sections~\ref{sec:de4} and~\ref{sec:cubic-de}, respectively.

\medskip
In order to prove Theorem~\ref{thm:equations}, we first relate $F(z,u)$ to the \gf\ of planar
maps counted by the distribution of their vertex degrees. More precisely, let
$\bar M\equiv \bar M(z,u; g_1, g_2, \ldots ; h_1,h_2,\ldots)$ be the generating
function of rooted planar maps,
with a weight $z$ per face, $u g_k$  per \emm non-root, vertex
of degree $k$ and $h_k$ if the root vertex has degree~$k$. 
\begin{lem}  \label{lem:cont} 
The series $F(z,u)$ is related to $M$ through:  
$$
 F(z,u)=  \bar M(z,u;t_1, t_2, \ldots; t^c_1, t^c_2,
 \ldots).
$$
\end{lem}

\begin{figure}[thb]
\begin{center}
(a)\  \begin{tabular}{c}{\includegraphics{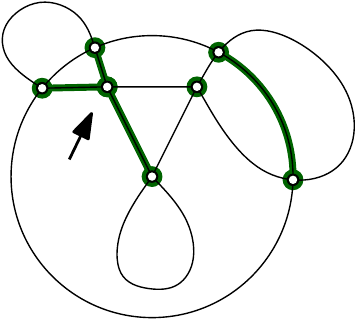}}
\end{tabular}\ \ \
(b)\    \begin{tabular}{c}{\includegraphics{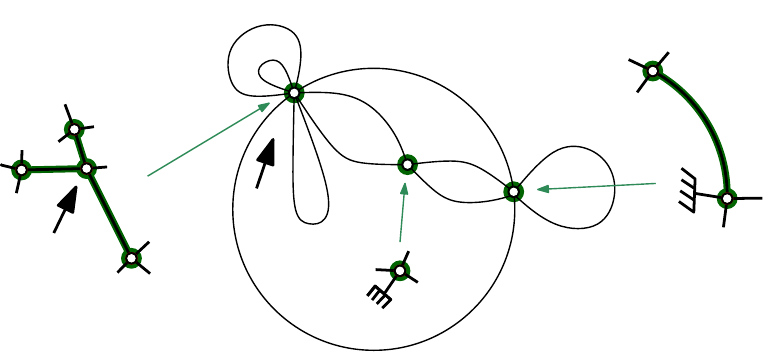}}
\end{tabular}\ \ \
\\
(c)\    \begin{tabular}{c}{\includegraphics{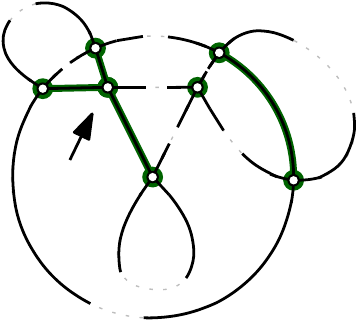}}
\end{tabular}
\end{center}
\caption{(a) A 4-valent forested map with 9 faces and 2
  non-root components. 
(b) The same map,
  after contraction of the forest. (c) Assembling the 3 trees gives the
  original forested map.}
\label{fig:forcou}
\end{figure}

\begin{proof} 
The idea is to contract each tree of a spanning forest, incident to $k$ half-edges, into a
  $k$-valent vertex. It is adapted from~\cite[Appendix
  A]{BDG-blocked}, where the authors study  4-valent forested maps for which
  the root edge is not in the forest. It can also be
  seen as an extension of Mullin's construction for maps equipped with
  a spanning tree~\cite{mullin-boisees}. Finally, it
  also appears in~\cite{sportiello}.

Let us now get into the details. First, let us recall that rooted maps
have no symmetries: all vertices, edges and half-edges are
distinguishable. In particular, one can fix, for every rooted planar map $M'$
(with arbitrary valences) a total order on its half-edges. This order
may have a combinatorial significance --- a good choice is
the order in which half-edges are visited when applying the
construction of~\cite{bdg2002} --- but can also be arbitrary.

We now describe a bijection $\Phi$, illustrated in Figure~\ref{fig:forcou},  between forested $p$-valent maps $(M,F)$
and pairs formed of a  map $M'$ and a collection $(T_v, v \in V(M'))$
of $p$-valent trees associated with the vertices of $M'$, such that the tree
associated with the root vertex of $M'$ is corner-rooted,  the
others are leaf-rooted, and the number of leaves of $T_v$ is the
degree of $v$ in $M'$. 

The map $M'$ is obtained by contracting all edges of the forest
$F$ (Figure~\ref{fig:forcou}(b)). The arrow that marks the root corner remains at the same place. Now split
into two half-edges each edge of $M$ that is not in $F$: this gives a
collection of $p$-valent trees, each of them being naturally associated with
a vertex $v$ of $M'$. The half-edges of these trees form together the
edges of $M'$ (Figure~\ref{fig:forcou}(c)). If $v$ is the
root vertex of $M'$, then $T_v$ inherits the corner-rooting of
$M$. Otherwise, we root $T_v$ at the smallest of its half-edges, for
the total order on half-edges of $M'$.

The following properties are readily checked:
\begin{itemize}
\item $T_v$ has $k$ leaves if $v$ has degree $k$ in $M'$,
\item $M$ and $M'$ have the same number of faces,
\item the number of vertices of $M'$ is the number of components of
  $F$.
\end{itemize}

Let us now prove that $\Phi$ is bijective. 
To recover the forested map $(M,F)$ from the contracted map $M'$ and
the associated collection of trees, we  inflate each vertex $v$ of
$M'$ into the corresponding tree $T_v$. If $v$ is  the root vertex of
$M'$, the root corner of $T_v$ must coincide with the root corner of
$M'$. Otherwise, the root half-edge of $T_v$ is put on the smallest of the
half-edges incident to $v$ in $M'$. This proves the injectivity of
$\Phi$. Since this reverse construction can be applied to any map
$M'$ with a corresponding collection of trees, $\Phi$ is also
surjective.
\end{proof}

\begin{proof}[Proof of Theorem~{\rm\ref{thm:equations}}]
In a recent paper,  Bouttier and Guitter~\cite{bg-continuedfractions}
have expressed the series 
$\bar M$ via a system of equations,
established bijectively\footnote{Strictly
  speaking, they do not take the vertex or face number into account,
  but both are
  prescribed by the distribution of vertex degrees.}. Their
expression    is actually fairly
complicated~\cite[Eq.~(1.4)]{bg-continuedfractions}, but the series
$z\bar M'_z$, which counts maps with a marked face, has a much simpler
expression~\cite[Eq.~(2.6)]{bg-continuedfractions}: 
\beq\label{relM} 
\bar M'_z = \sum_{i \geq 0} \sum_{j \geq 0} h_{2i+j} {2i + j \choose i,i,j} R^i S^j, 
\eeq
where $h_0=0$ and, by~\cite[Eq.~(2.5)]{bg-continuedfractions},
\beq
R = z + u\sum_{i\geq 1} \sum_{j \geq 0} g_{2i+j} { 2i + j - 1 \choose i
  - 1,i,j} R^i S^j, \ \ \ \ 
S =u \sum_{i\geq 0} \sum_{j \geq 0} g_{2i+j+1} { 2i + j  \choose i,i,j}
R^i S^j. 
\label{relS} 
\eeq 
Theorem~\ref{thm:equations} follows by specialization, using Lemma~\ref{lem:cont}. 

It remains to check that~(\ref{defR}--\ref{defS}) defines a
unique pair of series $R$ and $S$ in $z$ with constant terms $0$. This is
readily proved by observing that~\eqref{defR} determines $R$ up to
order $n$ if we know $R$ and $S$ up to order $n-1$; and
that~\eqref{defS} determines $S$ up to order $n$ if we know $R$ up to
order $n$ and $S$ up to order $n-1$.
\end{proof}

\noindent{\bf Remark.}
The expression of $\bar M$ given in~\cite[Eq.~(1.4)]{bg-continuedfractions} leads to an
explicit expression of $F(z,u)$ in terms of $R$ and $S$. However, this
expression involves a triple sum (a double sum when $p$ is even, see
for instance~\eqref{Fzu-expl}).
This is why we prefer handling the  expression of $F'$. We discuss
this further in the final section.

\subsection{Quasi-$p$-valent  maps ($p$ odd)}

A map is said to be \emph{quasi-$p$-valent} if all its vertices have
degree $p$, apart from one vertex  which is a leaf.
Such maps exist only when $p$ is odd. They are
naturally rooted at their leaf:  the root
corner is  the unique corner incident to the leaf and the root edge is
the unique edge incident to the leaf. Let  $G(z,u)$ denote the
generating function of quasi-$p$-valent forested maps counted by faces
($z$) and  non-root components ($u$) (see Figure~\ref{fig:quasic}).

\begin{figure}[h!]
\begin{center}
\includegraphics{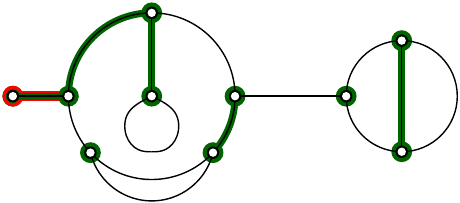}
\end{center}
\caption{A quasi-cubic 
forested map with 6
  faces and 4 non-root components.} 
\label{fig:quasic}
\end{figure}


\begin{prop}\label{prop:G}
The generating function of quasi-$p$-valent forested maps  is
\beq\label{G-expr}
 G(z,u) = ( 1 +  \bu) \left(zS - u \sum_{i \geq 2}\sum_{j \geq 0} t_{2i+j-1} { 2 i + j - 2 \choose i-2,i,j } R^i S^j \right), 
\eeq
where $\bu=1/u$, the series $R$ and $S$ are  defined
by~\rm{(\ref{defR}-\ref{defS})} and the numbers $t_k$ by \eqref{deftk}.  Also,
$$
G'_z(z,u)=(1+\bu)S.
$$
\end{prop}

As in the previous subsection, the key of this result is to relate
$G(z,u)$ to a well-understood \gf\ of maps --- here, the
\gf\ $\Gamma_1\equiv \Gamma_1(z,u;g_1, g_2, \ldots )$ that counts planar
maps rooted at  leaf,  with a weight $z$ per face and $ug_k$ per $k$-valent \textit{non-root} vertex.

\begin{lem} The following analogue of Lemma~{\rm\ref{lem:cont}} holds for
  quasi-$p$-valent forested maps:
$$
 G(z,u) = ( 1 +  \bu ) \, \Gamma_1 (z,u;t_1, t_2, \ldots)
$$
with $\bu=1/u$.
\end{lem}
\begin{proof}
  The bijection used in the proof of Lemma \ref{lem:cont}  shows that the
series  $\Gamma_1(z,u;t_1, t_2, \ldots)$ counts quasi-$p$-valent
forested maps such that the 
root edge \emm is not, in the forest. (With the notation used in that proof,
the root vertex of $M'$, of degree  1, remains a trivial tree during
the inflation step). To each such forested map, we can add the root edge
to the forest. The resulting  forested map has  one less component, hence the factor $\bu=1/u$.
\end{proof}

\begin{proof}[Proof of Proposition~{\rm\ref{prop:G}}]
The series $\Gamma_1$ has also been expressed by  Bouttier \emm et
al., in terms of the series $R$ and $S$ of~\eqref{relS}
(see~\cite[Eq.~(2.6)]{bdg2002}):
\beq\label{G1}
\Gamma_1= zS- u \sum_{i\ge2}\sum_{j\ge 0} g_{2i+j-1} {{2i+j-2} \choose
  {i-2,i,j}}R^i S^j. 
\eeq
This gives the first part of Proposition~\ref{prop:G}. For the second
part, we observe that $\Gamma_1$ is by definition the coefficient of
$h_1$ in the series $\bar M(z,u;g_1, \ldots ; h_1, \ldots)$ defined
above Lemma~\ref{lem:cont}. Hence it follows from~\eqref{relM} that $\Gamma_1'=S$
(this can also be derived combinatorially from~\cite{bdg2002}). 
\end{proof}

\subsection{When the root edge is outside  the forest}
  \label{subsec:outside}
We now focus on 
 forested maps such that  the root edge is
outside the forest. Let  $H(z,u)$ denote the associated generating
function.

\begin{prop}\label{prop:H}
The generating function 
 of $p$-valent forested maps where the
root edge is outside the forest is 
\begin{multline}
 H(z,u) =\bu zR +\bu zS^2 -\bu z^2
\label{H-expr}\\
-2S  \sum_{i \geq 2}\sum_{j \geq 0} t_{2i+j-1}
 { 2 i + j - 2 \choose i-2,i,j } R^i S^j -
 \sum_{i \geq 3}\sum_{j \geq 0} t_{2i+j-2} { 2 i + j - 3 \choose i-3,i,j } R^i S^j,
\end{multline}
where $\bu=1/u$, the series $R$ and $S$ are defined
by~{\rm{(\ref{defR}-\ref{defS})}} and the 
numbers $t_k$ by~\eqref{deftk}. 

  When $p$ is even, then $S=0$  and the first double sum disappears.
In this case, we also have a very simple expression of $H'_z(z,u)$:
\beq\label{Hprime-R}
H'_z(z,u)= 2\bu (R-z).
\eeq
\end{prop}

Again, the key of this result is to relate
$H(z,u)$ to a well-understood \gf\ of maps --- here, the
\gf\ $M\equiv M(z,u;g_1, g_2, \ldots )$ that counts rooted planar
maps with a weight $z$ per face and $ug_k$ per  vertex of degree $k$.

\begin{lem} The following analogue of Lemma~{\rm\ref{lem:cont}} holds:
$$
 H(z,u) = \bu \, M(z,u;t_1, t_2, \ldots).
$$
\end{lem}
\begin{proof}
Let us consider again  the bijection used in the proof of
Lemma~\ref{lem:cont}:   the fact that the root edge of $M$ is not in
the forest $F$
means that, in the corner-rooted tree associated with the root vertex
of $M'$, the root half-edge is a leaf. It is then
equivalent to root this tree at this leaf.
\end{proof}
\begin{proof}[Proof of Proposition~{\rm\ref{prop:H}}] 
The first part of the proposition follows from the known characterization of $M$
(see~\cite[Eq.~(2.1)]{bdg2002}):
$$
M= \frac{\Gamma_1^2+\Gamma_2}{z} -z^2,
$$
where $\Gamma_1$ is given by~\eqref{G1} and 
$$
\Gamma_2=z^2R-uz\sum_{i\ge 3}\sum_{j\ge 0}{{2i+j-3}\choose {i-3,i,j}}
R^i S^j
- u^2 \left( \sum_{i\ge2}\sum_{j\ge 0} g_{2i+j-1} {{2i+j-2} \choose
  {i-2,i,j}}R^i S^j\right)^2,
$$
with $R$ and $S$ satisfying~\eqref{relS}. This gives the first part of
the proposition.  

Observe that $M(z,u;g_1, g_2, \ldots)=u  \bar M(z,u;g_1,g_2, \ldots;
g_1,g_2, \ldots)$ 
where $\bar M$ is defined just above Lemma~\ref{lem:cont}. When $p$ is
even,  the maps obtained by contracting forests have even degrees  ($g_{2k+1}=0$
for all $k$), the series $S$ given by~\eqref{relS} vanishes,
and the combination of~\eqref{relM} and~\eqref{relS} gives $\bar
M'_z(z,u;g_1,g_2, \ldots; g_1,g_2, 
\ldots)=2\bu (R-z)$.  Thus $H'_z= 
\bu M'_z= \bar M'_z=2\bu (R-z)$, as stated in~\eqref{Hprime-R}.
\end{proof}

\section{Differential equations}
\label{sec:de}
The equations established in the previous section imply that series
counting regular forested maps are D-algebraic. We compute explicitly
a few differential equations.

\subsection{The general case}
\label{sec:da-general}
\begin{theo}\label{Dalg}
The generating function $F(z,u)$ of $p$-valent forested maps is
D-algebraic (with respect to $z$). The same holds for
the series $G(z,u)$ and $H(z,u)$ of Propositions~\ref{prop:G} and~\ref{prop:H}.
\end{theo}

\begin{proof} We start from the expression~\eqref{frs} of $F'(z,u)$ (as we always differentiate with respect
  to $z$, we simply denote $F'(z,u)$ for $F'_z(z,u)$).
We first observe that the doubly hypergeometric series  $\theta$, $\Phi_1$,
$\Phi_2$  are D-finite (this follows from the closure properties of
D-finite power series~\cite{lipshitz-df}).

Then, by differentiating  (\ref{defR}) and \eqref{defS} with respect to $z$, we obtain
rational expressions of $R'$ and $S'$  in terms of $u$ and the partial
derivatives $\partial \Phi_\ell/\partial x$ and  $\partial
\Phi_\ell/\partial y$, evaluated at $(R,S)$, for $\ell=1,2$. (Indeed,  differentiating
\eqref{defR} and \eqref{defS} gives  a linear system in $R'$ and
$S'$. Its determinant  is a power series in $z$ with coefficients in
$\qs[u]$. It is non-zero, since it equals $1$ at  $u=0$.)

 Let $\GK$ be the field $\Q(u)$. Using~\eqref{frs} and the previous
 point, it is now easy to prove by induction that for all $k\geq 1$,
 there exists a rational expression of $F^{(k)}(z,u)$ in
 terms of 
$$
\left\{\frac{\partial^{i+j}\Phi_\ell}{\partial x^i \partial
   y^j}(R,S),\frac{\partial^{i+j}\theta}{\partial x^i \partial
   y^j}(R,S)\right\}_{i \geq 0,j \geq 0, \ell \in \{1,2\}}
$$
with coefficients in  $\GK$. But since  $\theta$, $\Phi_1$ and
$\Phi_2$  are D-finite, the above set of series spans a 
vector space of finite dimension $d$ over $\GK(R,S)$. Therefore there exist $d$
elements $\varphi_1,\dots,\varphi_d$ in this space, and rational
functions $A_k \in \GK(x,y,x_1,\dots,x_d)$, such that $F^{(k)} (z,u)
= A_k(R,S,\varphi_1,\dots,\varphi_d)$ for all $k \geq 1$.  

Since the transcendance degree \cite[p.~254]{lang}
of 
$\GK(R,S,\varphi_1,\dots,\varphi_d)$ over $\GK$ is (at most) $d+2$,
the $d+3$ 
series  $F^{(k)} (z,u)$, for $1\le k\le
d+3$, are algebraically dependent, so that $F'$ (and thus $F$) is
D-algebraic.

The proof is similar for the series $G(z,u)$ and $H(z,u)$, with
$\theta$ replaced by the adequate D-finite series derived
from~\eqref{G-expr} and~\eqref{H-expr}. Moreover, since these two
expressions involve $z$ explicitly, the field $\qs(u)$ used in the
above argument must be replaced by $\qs(z,u)$.
 \end{proof}

\subsection{The 4-valent case}
\label{sec:de4}
We specialize the above argument to the case $p=4$. As explained in
the second remark following Theorem~\ref{thm:equations}, the series
$S$ vanishes and $F'(z,u)$ is 
given by the system~\eqref{system-simple}, with
\beq\label{theta-phi-4V}
 \theta(x) = 4 \sum_{i \geq 2} \frac{(3i-3)!}{(i-2)!i!^2} x^i
 \quad \hbox{and} \quad \Phi(x) = \sum_{i \geq 2}
 \frac{(3i-3)!}{(i-1)!^2i!} x  ^i. 
\eeq
The series
$\theta(x)$, $\Phi(x)$ and their derivatives lie in a 3-dimensional vector
space over $\qs(x)$ spanned (for instance) by $1$, $\Phi(x)$ and
$\Phi'(x)$. This follows  from the following equations, which are
easily checked:
\beq\label{phi-second}
x (27x-1 ) \Phi'' ( x )+6\Phi( x )  +6x = 0,
\eeq
\beq \label{thetrat} 
 3\theta(x)= 2(27x-1 ) \Phi'(x) -42\Phi(x) +  12x.
\eeq
By the argument described above, we can now express first $R'$, and
then $F'$ and all its
derivatives as rational functions of $u$, $R$, $\Phi(R)$ and 
$\Phi'(R)$. But since $R=z+u\Phi(R)$, this means a rational function
of $u$, $z$, $R$ and $\Phi'(R)$. We compute the explicit  expressions of
$F'$, $F''$ and $F'''$, eliminate  $R$ and $\Phi'(R)$ from these
three equations, and this gives a differential equation of order 2
and degree 7 satisfied by $F'$, the details of which are not particularly
illuminating:
$$
  \substack{9\,{F'}^{2}{F''}^{5}{u}^{6}+36\,{F'}^{2}{F''}^{3
}F'''\,{u}^{5}z+144\,{F'}^{2}{F''}^{4}{u}^{5}-12\,
 \left( 21\,z-1 \right) F'\,{F''}^{5}{u}^{5}+432\,{F'
}^{2}{F''}^{2}F'''\,{u}^{4}z-48\, \left( 24\,z-1 \right)
F'\,{F''}^{3}F'''\,{u}^{4}z
\\ +864\,{F'}^{2}{F''
}^{3}{u}^{4}-96\, \left( 27\,z-2 \right) F'\,{F''}^{4}{u}^{
4}+4\, \left( 27\,z-1 \right)  \left( 15\,z-1 \right) {F''}^{5}{u
}^{4}+1728\,{F'}^{2}F''\,F'''\,{u}^{3}z-288\, \left( 21
\,z-2 \right) F'\,{F''}^{2}F'''\,{u}^{3}z
\\
+10368\,F'\,{F'''}^{2}{u}^{2}{z}^{3}+16\, \left( 27\,z-1 \right)  \left( 
21\,z-1 \right) {F''}^{3}F'''\,{u}^{3}z+2304\,{F'}^{2}{
F''}^{2}{u}^{3}-288\, \left( 31\,z-4 \right) F'\,{F''}^{3}{u}^{3}
\\
-64\, \left( 6\,uz-162\,{z}^{2}+33\,z-1 \right) {F''}^
{4}{u}^{3}+2304\,{F'}^{2}F'''\,{u}^{2}z-2304\, \left( 6\,z-1
 \right) F' \,F''\,F'''\,{u}^{2}z \\
-192\, \left( 8\,uz-54\,{z}^{2}+29\,z-1 \right) {F''}^{2}F'''\,{u}^{2}z-768\,
 \left( 2\,u+189\,z-7 \right) {F'''}^{2}u{z}^{3}+2304\,{F'}^
{2}F''\,{u}^{2}-3072\, \left( 3\,z-1 \right) F'\,{F''}
^{2}{u}^{2}
\\-192\, \left( 24\,uz-27\,{z}^{2}+55\,z-2 \right) {F''}
^{3}{u}^{2}
-1536\, \left( 21\,z-2 \right) F'\,F'''\,uz-768\,
 \left( 12\,uz+81\,{z}^{2}+24\,z-1 \right) F''\,F'''\,uz+1536
\, \left( 9\,z+2 \right) F'\,F''\,u\\
-512\, \left( 39\,uz+81
\,{z}^{2}+51\,z-2 \right) {F''}^{2}u+36864\,F'\,z-1024\,
 \left( 12\,uz-162\,{z}^{2}+33\,z-1 \right) F'''\,z-1024\, \left( 
36\,uz+27\,z-1 \right) F''-24576\,z
=0.}
$$
As discussed in  Section~\ref{sec:final}, we conjecture that $F$ does not
satisfy a differential equation of order~2.

We have applied the same method to the series $H$ of
Proposition~\ref{prop:H}: 
$$
H(z,u)=\bu z R-\bu z^2-\Lambda(R)
$$
where
$$
\Lambda(x)= \sum_{i\ge 3} \frac{(3i-6)!}{(i-3)!(i-2)!i!}x^i
$$
satisfies
$$
30\Lambda(x)=x(27x-1) \Phi'(x)+(1-24x)\Phi(x)+3x^2.
$$
 This gives for $H$ an equation of order 2 and degree 3:
\begin{multline*}
3\, ( u+1 ) {u}^{2}{H'}^{2}H'' +  12\,{u}^{2}zH' H''  
+6\, ( u-8 ) u{H'}^{2}
+240\,H  \\
+4\, ( 6\,uz-54\,z+1 ) H'
  +4\, (3\,u{z}^{2}+30\,uH  +27\,{z}^{2}-z )H''+24\,{z}^{2}=0 .
\end{multline*}
One reason explaining this more modest size is the simplicity of the
expression~\eqref{Hprime-R} of  $H'$. 

\subsection{The cubic case}
\label{sec:cubic-de}
We start from the expression of $F'$ given in
Theorem~\ref{thm:equations}. We now have to deal with series 
$\theta$, $\Phi_1$ and $\Phi_2$ in two variables:
$$
 \theta(x,y) = 3 \sum_{i \geq 0} \sum_{\substack{j \geq 0 \\ 2i+j \geq 3 }} {\frac { \left( 4\,i+2\,j-4 \right) !}{ \left( 2i+j-3 \right) !\,
  i!^{2}j!}}
 x^iy^j, 
$$
\beq\label{phi1}
 \Phi_1(x,y) = \sum_{i \geq 1}\sum_{\substack{j \geq 0 \\ 2i+j \geq 3 }} \frac{(4i+2j-4)!}{(2i+j-2)!\,(i-1)!\,i!\,j!}x^iy^j,
\eeq
\beq\label{phi2}
 \Phi_2(x,y) = \sum_{i \geq 0}\sum_{\substack{j \geq 0 \\ 2i+j \geq 2 }} \frac{(4i+2j-2)!} {(2i+j-1)!i!^2j!}x^iy^j.
\eeq
Let us first observe that 
$$
\theta(x,y)=-2\Phi_1(x,y)+(1-y)\Phi_2(x,y)-2x-y^2.
$$
Consequently, Theorem~\ref{thm:equations} gives:
\beq \label{F-cubic}
F'=2z\bu +\bu S -(1+\bu)(2R+S^2).
\eeq
Then, the summations over the variable $j$ that occur in $\Phi_1$ and
$\Phi_2$ can be performed explicitly,
which gives to the cubic case   a one-variable flavour. Indeed,
\begin{eqnarray}
  \Phi_1(x,y) &=& \left( 1-4y \right) ^{3/2}\,{\Psi_1} \left(
t\right) -x,\label{pp1}
\\
\Phi_2(x,y) &=& \sqrt {1-4y}\,{\Psi_2} \left(t \right) 
+ \frac 1 4\left({1-\sqrt{1-4y}}\right)^2, \label{pp2}
\end{eqnarray}
where $t=x/(1-4y)^2$ and
$$
\Psi_1(z) = \sum _{i\ge 1}{\frac { \left( 4\,i-4 \right) !}{
 \left( 2\,i-2 \right) !\,i!\, \left( i-1 \right) !}}\,{z}^{i}, \quad
\Psi_2(z) = \sum_{i\ge 1}{\frac { \left( 4\,i-2 \right) !}{
 \left( 2\,i-1 \right) !\, i! ^{2}}}\,{z}^{i}.
$$
Our system thus reads:
\begin{eqnarray}
 \label{phi1enpsi}
R &=&z+u \left( 1-4\,S \right) ^{3/2}\,{\Psi_1} \left(
T\right) -uR,
\\
 \label{phi2enpsi}
S&=&u \sqrt {1-4S}\,{\Psi_2} \left(
T
 \right) 
+ \frac u 4\left({1-\sqrt{1-4S}}\right)^2,
\end{eqnarray}
with
$T=R/(1-4S)^2$.

The series $\Psi_1(z)$, $\Psi_2(z)$ 
and their derivatives lie in a 3-dimensional vector space over
$\qs(z)$ spanned (for instance) by 1,  $\Psi_1(z)$ and
$\Psi_2(z)$. This follows from the following identities, which are
easily checked:
\beq\label{ed-cubic}
(1-64z) \Psi_1'(z)+ 48\Psi_1(z)+ 2\Psi_2(z)=1 ,
 \quad \quad 
z (1-64  z) \Psi_2'(z)+ 6\Psi_1(z)+ 16  z\Psi_2(z)=8  z.
\eeq
By the argument of Section~\ref{sec:da-general}, we can now express
$R'$ and $S'$ as rational functions of $u$, $R$, $S$, $\Psi_1(T)$ and 
$\Psi_2(T)$.
But $\Psi_1(T)$ and $\Psi_2(T)$ can be expressed rationally in
terms of $z$, $u$, $R$ and $\sqrt{1-4S}$
using~\eqref{phi1enpsi} and~\eqref{phi2enpsi}. Hence
we obtain rational expressions in $u$, $z$, $R$ and $\sqrt{1-4S}$. In
fact no square root  occurs: 
\begin{eqnarray}
  R'&=& \frac{R(48z-1+16(u+1)R+2(3+u)S-8(u+1)S^2)}{D} ,\nonumber
\\
S'&=&\frac{2(3z+(u-3)R-12zS+4(u+1)RS)}D,\nonumber
\end{eqnarray}
with
$$
D=36z^2+(24z-1+24uz)R+4(u+1)RS-4(u+1)^2RS^2+4(u+1)^2R^2.
 $$
Combining  these two equations with~\eqref{F-cubic}, one can now
express $F'$, $F''$ and $F'''$ in terms of
$u$, $z$, $R$ and $S$, and then eliminate $R$ and $S$ to obtain a
differential equation of order 2 satisfied by $F'$ (of degree 17).
 For the \gf\ $G(z,u)$ of
quasi-cubic forested maps (Proposition~\ref{prop:G}), we
replace~\eqref{F-cubic} by
$$
10G= (1+\bu) \left(z-R+6zS+2(u+1)RS\right),
$$
and obtain a
differential equation of order 2 and degree 5. It becomes a bit simpler
when we rewrite $G=(W+z\bu)/2$:
\begin{multline*}
 0= \left( 3\,{u}^{4}z {{W}'} ^{4}-{u}^{3} ( 5\,W
      u-uz+z )  {{W}'} ^{3}+4\, ( u+1 ) (
      5\,W u-uz+z ) ^{2} \right) {W}''\\
 -48\,{u}^{2}z ( u+1)  {{W}'} ^{3}+8\,u ( u+1 )  ( 5\,W u-uz+z )  {{W}'} ^{2}+4\, ( u^2-1) ( 5\,W u-uz+z ){{W}'}.
\end{multline*}
 Introducing the series $W$ is natural in the solution of the
 Potts model presented in~\cite{bernardi-mbm-de}, where the above
 equation was first obtained.

\section{Combinatorics of forested trees}
\label{sec:u+1}

Equation~\eqref{F-Tutte}, and the positivity of the Tutte
coefficients, show that
the series $F(z,u)$ that counts $p$-valent forested maps has non-negative
coefficients when expanded in $(1+u)$. We say that it is \emm
$(u+1)$-positive,. Section~\ref{sec:tutte} presents several combinatorial
descriptions of $F(z,\mu-1)$ (see~\eqref{F-act},~\eqref{F-dual},~\eqref{F-sandpile},~\eqref{F-Potts}).
 This will lead us to study the asymptotic
behaviour of the coefficient of $z^n$ in $F(z,u)$ not only for $u\ge 0$, but 
for $u\ge  -1$. 
%
In this study, we will need to know
that several other series
are also $(u+1)$-positive. We  prove this thanks to a combinatorial argument
that applies to certain classes of \emm forested trees,.

\subsection{Positivity in $(1+u)$}
\label{sec:++}
Let $T$ be a tree having at least one edge,
and $\cF$ a set of spanning forests of $T$. We
define  a property of  $\cF$ that guarantees that the  
 generating function $A_\cF(u)$ that counts forests of  $\cF$  by the
 number of components 
 is $(u+1)$-positive (after division by $u$). 

Let $F\in \cF$, and let $e$ be an edge of
$T$. By \emm flipping, $e$ in the forest $F$, we mean adding $e$ to $F$ if it is not in
$F$, and  removing it from $F$ otherwise. This gives a new forest $F'$
of $T$. We
say that $e$ is \emm flippable, for $F$ if  $F'$
 still belongs to $\cF$.
We say that $\mathcal F $ is
\emph{stable} if for each $F \in \mathcal F$,
\begin{enumerate}
\item[(i)] every edge of $T$ not belonging to $F$ is flippable,
\item[(ii)] flipping a flippable
edge gives a new forest with the same set  of flippable edges.
\end{enumerate}
We say that a set $\cE$ of forested trees is \emm stable, if,
for each tree $T$, the set  of forests $F$ such that $(T,F)\in \cE$ is stable.
We consider below a \gf\ $E(z,u)$ of $\cE$, where each forested tree $(T,F)$ is
weighted by $z^n u^k$, where $n$ is the size of $T$ (number of edges,
of leaves...) and $k$  the number of components of $F$, minus $1$.
\begin{lem} \label{pose}
With the above notation, assume $\mathcal F $ is stable. Then all elements of $\cF$ have the
same number, denoted by $f$, of flippable edges. The \gf\ of forests of $\cF$,
counted by components, is
$
A_\cF(u)= u(1+u)^{f}.
$
Consequently, if $\cE$ is a stable set of forested trees, then $E(z,u)$ is $(u+1)$-positive.
\end{lem}

\begin{proof} Condition (i) implies that the forest $F_{\max}$
  consisting of all edges of $T$ belongs to~$\cF$. Moreover, we can
  obtain $F_{\max}$ from any forest $F$ of $\cF$ by adding iteratively
  flippable edges. By Condition (ii), this implies that any forest $F$ of
  $\cF$ has the same set of flippable edges as $F_{\max}$. It also implies
  that, to construct a forest $F$ of $\cF$, it suffices to choose, for each
  flippable edge of $F_{\max}$, whether it belongs to $F$ or not. Since
  $F_{\max}$ has a unique component, and since deleting  an edge from a forest
  increases by 1 the 
  number of components, the expression of $A_\cF(u)$ follows.
\end{proof}

\subsection{Enriched blossoming trees}
\label{sec:enriched}
We now apply the above principle to establish $(u+1)$-positivity
properties for the series  $R$ and $S$ given
by~\rm{(\ref{defR}--\ref{defS})}, and for the series $\tilde S\equiv \tilde S(z,u)$
defined by 
\beq \label{defStilde} 
\tilde S(0,u)=0, \quad \tilde S = u\,\Phi_2(z,\tilde S),
\eeq
where $\Phi_2$ is given by~\eqref{phi}.

We consider 
 plane trees rooted at a half-edge, which we draw hanging from their root as in Figure~\ref{enrichi}.
 A vertex of degree $d$ is thus seen as the parent of $d-1$  children.
A \emm subtree, 
consists of a vertex $v$ and all its descendants. It is naturally rooted at the
half-edge incident to $v$ and located just above it.
A \emm blossoming tree,
is a leaf-rooted plane tree with two
 kinds of childless vertices: 
\emm{leaves},, represented by white arrows, and
\emm{buds},, represented by black arrows.
The edges that carry leaves and buds 
 are considered as half-edges. (This means that leaves and buds are not
 actual nodes of the tree, so that a spanning forest of
 a blossoming tree does not contain any of its half-edges.)
The root half-edge does not carry any leaf or bud.
Each leaf is assigned  a \emm charge, $+1$ while each bud is
assigned a charge $-1$.  The \emm 
charge, of a blossoming tree is the difference
between the number of leaves and buds that it contains. This
definition is extended to subtrees.

\begin{defi}\label{def:enriched}
Let $p\ge 3$.  A $p$-valent blossoming tree $T$ equipped with a spanning
forest $F$ is an  \emm enriched R- (resp. S-) tree, if
\begin{enumerate}
\item[(i)] its  charge is $1$ (resp. $0$),
\item[(ii)] any subtree rooted at an edge \emm not in, $F$
has charge $0$ or $1$. 
\end{enumerate}
We also consider as an enriched R-tree a single root half-edge
carrying at a leaf (Figure~\ref{enrichi}, left).

The pair $(T,F)$ is an \emm enriched \~S-tree, if
each component of $F$  is incident to as many leaves as buds.
In this case it is also an enriched S-tree.
\end{defi}
An enriched R-tree is shown in Figure~\ref{enrichi}.
The readers who are familiar with the R- and S-trees of~\cite{bdg2002}
will recognize that 
our enriched R- and S-trees  are obtained from them by inflating each vertex of  degree $k$ into a (leaf rooted) $p$-valent tree with $k$
leaves. The following proposition should not come as a surprise for
them.
\begin{prop}\label{prop:RSSt}
Let $p\ge 3$. The series $R$, $S$ and  $\tilde S$ defined
by~\eqref{defR},~\eqref{defS} and~\eqref{defStilde} 
 count respectively enriched R-, S- and \~S-trees, by the number of leaves
 ($z$) and the number  of components in the forest ($u$). 
\end{prop}
\begin{proof}
 The equations follow from a recursive decomposition of enriched trees. For instance,
 an enriched R-tree 
is either reduced to a single leaf, with no forest at all (contribution: $z$),
or contains a root node. This node belongs to a component of the
forest. This component is incident to  several
 half-edges (not belonging to the forest), one of them being the
 root half-edge. Each of the other incident half-edges can be a leaf, a bud, or
 the root of a non-trivial subtree. In this case, the  definition of
 enriched R-trees implies that this subtree is itself 
 an enriched R-tree (of charge~1)  or an enriched S-tree (of charge
 0). Since a single leave is considered as an R-tree, we can say that
 every half-edge incident to the root component of the forest carries
 either a bud, or an R- or S-tree. If there are $i$ attached
  R-trees, we must have $i-1$ buds for the tree charge to be
 1, and an arbitrary number $j$ of 
 S-trees. The root component of the forest is then a leaf-rooted tree
 with $k=2i+j$ leaves. This gives~\eqref{defR}, where the multinomial
 coefficient occurring in $\Phi_1$ describes the order in which the $i$ R-trees, the $i-1$
 buds and the $j$ S-trees are organized.

The proof of~\eqref{defS} is similar, but now as many buds as R-trees
must be attached to the root component of the forest to make the charge 0.

Finally, an \~S-tree is an S-tree in which all attached R-trees are
actually leaves. This explains that~\eqref{defStilde} is obtained
from~\eqref{defS} by replacing each occurrence of $R$ by $z$.
\end{proof}

\begin{figure}[htb]
\begin{center}
\begin{tabular}{c}
\includegraphics[scale=1.3]{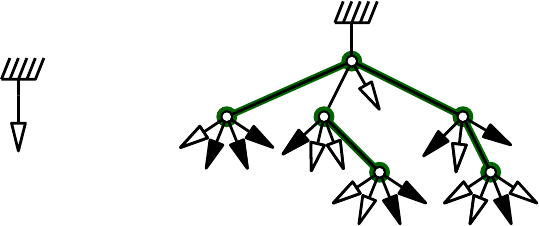}
\end{tabular}
\end{center}
\vskip -4mm
\caption{\emph{Left:}  The smallest   enriched R-tree. \emph{Right:} An
  enriched 5-valent R-tree having 10 leaves (white; charge 
  $+1$) and 9 buds (black; charge $-1$). }
\label{enrichi}
\end{figure}

We now come back to $(u+1)$-positivity.
\begin{prop}
The set of $p$-valent enriched R- (resp.~S-, \~S-) trees  having at
least one edge 
is stable, in the sense of Section~{\rm\ref{sec:++}}.
%
\end{prop}
\begin{proof}
 For enriched R- and S-trees, an edge is flippable if and only if the
attached subtree has charge 0  or 1, and this condition is independent of the forest.

For enriched \~S-trees, an edge is flippable if and only if the
attached subtree is incident to as many leaves as buds, and this
condition is again independent of the forest.
\end{proof}

By combining this proposition with Lemma~\ref{pose} and
Proposition~\ref{prop:RSSt},  we obtain:
\begin{cor}
\label{positiv}
The series $\bu({R-z})$, $ \bu S$  {and} $\bu{\tilde S}$ are
$(u+1)$-positive.  When $u=\mu-1$, they count respectively
(non-empty)  enriched R-, S- and \~S-trees, by the number of leaves
 ($z$) and the number  of flippable edges ($\mu$). 
\end{cor}

\noindent When $p=3$ for instance,  we have, writing $\mu=u+1$,
\begin{eqnarray*}
  \bu(R-z)&=& 2(2\mu+1)z^2+ 4(10\mu^3+12\mu^2+9\mu+4)z^3 + O(z^4),
\\
\bu S&=&
2z+6(2\mu^2+2\mu+1)z^2+ 8 (16\mu^4+28\mu^3+30\mu^2+22\mu+9)z^3+ O(z^4),
\\
\bu \tilde S &=&2z+2(2\mu^2+8\mu+5)z^2+8(2\mu^4+12\mu^3+33\mu^2+40\mu+18)z^3
+ O(z^4).
\end{eqnarray*}
We will also need the following variant of these results.

\begin{lem}
\label{positiv-bis}
Define $\Phi_2$ by~\eqref{phi} and  ${\tilde S}$
by~\eqref{defStilde}. The series $\pd {\Phi_2} y (z, \tilde S)$ is
$(u+1)$-positive.   
\end{lem}
\begin{proof}
  Let us extend the definition of \~S-enriched trees to $p$-valent blossoming trees that, in addition to leaves and
 buds, contain also one dangling half-edge, having no charge
 (Figure~\ref{fig:dang}). Using the arguments of Proposition~\ref{prop:RSSt}, one can prove that the series 
  $u\pd {\Phi_2} y (z, \tilde S)$ counts such \~S-enriched trees for which
 the half-edge is incident to the root component (as before, $z$
 counts leaves and $u$ components).

\begin{figure}[h!]
\begin{center}
\includegraphics[scale=1.2]{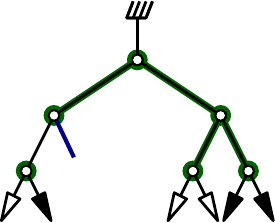}
\end{center}
\caption{A cubic enriched \~S-tree with a dangling half-edge incident to the
  root component.}
\label{fig:dang}
\end{figure}

The set of such trees is again stable:
indeed, an edge is flippable if
it is flippable in the \~S-sense, and is not on the path from the root
to the dangling half-edge.
\end{proof}

\section{Implicit functions: some general results}
\label{sec:general}
The singular behaviour of a series $Y(z)$ defined by an implicit
equation $H(z,Y(z)) = 0$ is well-understood when the singularities of
$Y$ occur at a point $z$ such that  $H$ is  analytic at $(z,Y(z))$, but $H'_y(z,
Y(z))=0$.  A typical situation is  the so-called  \emm smooth
implicit schema, of~\cite[Sec.~VII.4]{flajolet-sedgewick}, which leads
 to square root singularities in $Y$. 

However, in our asymptotic analysis of 4-valent and cubic forested
maps, we will have to deal with implicit functions $Y$ that become
singular at a point $z$  such that  $H$ ceases to be  analytic at
$(z,Y(z))$. Our series $Y$ 
 have non-negative real
coefficients, which implies that their radius is also a dominant
singularity, and leads us to pay a special attention to the behaviour
of $Y$ along the positive real axis.

In this section, we thus examine how far a real series $Y$ defined by an implicit equation
can be extended along the positive real axis. We first establish a
general result for equations of the form $H(z,Y(z)) = 0$
(Proposition~\ref{lemmeasympt2}), which will 
apply for instance to the series $\tilde S$ defined by~\eqref{defStilde}. We
then specialize this proposition  to an \emm inversion
equation, of the form $\Omega(Y(z))=z$ (Corollary~\ref{lemmeasympt}). This
corollary will apply in particular
to the series $R$ defined, in the 4-valent case, by $R=z+u\Phi(R)$ (see~\eqref{system-simple}).
\begin{prop} \label{lemmeasympt2}
Let  $H(x,y)$ be a real bivariate power series, analytic in a
neighbourhood of $(0,0)$, satisfying $H(0,0) = 0$ and $ H'_y (0,0) >
0$. 
Let $Y\equiv Y(z)$ be the unique  power series satisfying  $Y(0)=0$
and $H(z,Y(z)) = 0$.
Then $Y$ has a non-zero  radius of convergence. Moreover,  there
exists $\rho > 0$  such that:
\begin{enumerate}
	\item[(a)] $Y$ has an analytic continuation, still denoted by
          $Y$, in a neighbourhood of  $[0,\rho]$, 
which is real valued,
	\item[(b)]  
$H$ has an analytic continuation, still denoted by $H$, in a neighbourhood
          of $\{(z,Y(z)), z \in [0,\rho]\}$      ,
	\item[(c)] $H(z,Y(z)) = 0$ for $z \in [0,\rho]$,
	\item[(d)]  $H'_ y (z,Y(z)) > 0$ for $z \in [0,\rho]$.
\end{enumerate}
Moreover, if $\rt$ is the supremum (in $\R\cup \{+\infty \}$) of these
values $\rho$, at least one of the  following properties holds:
\begin{itemize}
\item[(i)] $\rt = +\infty$,
\item[(ii)] $\liminf_{z \rightarrow \rt^-} H'_y  (z,Y(z)) = 0$,
\item[(iii)] for each $y \in [\liminf_{z \rightarrow \rt^-}
  Y(z),\limsup_{z \rightarrow \rt^-} Y(z)]$, $H$ is singular at $(\rt,y)$,
\item[(iv)] $\limsup_{z \rightarrow \rt^-} |Y(z)| = + \infty$.
\end{itemize}
\end{prop}
We begin with a simple lemma.
\begin{lem}\label{lem:real}
Let $a<0<b$ and let $Y$ be a function analytic in a neighbourhood of $[a,b]$, whose Taylor expansion
  at $0$ has real coefficients. Then $Y$ takes real values on $[a,b]$. 
\end{lem}
\begin{proof} 
The functions $z\mapsto Y(z)$ and  $z\mapsto \overline{ Y(\bar z)}$
are analytic and coincide in the
neighbourhood of 0 where $Y(z)$ is given by its Taylor expansion. Hence
they coincide everywhere, and $Y(z)$ is real when $z$ is real.
\end{proof}

\begin{proof}[Proof of Proposition~\ref{lemmeasympt2}]
The uniqueness of $Y$ comes from the fact that its coefficients can be
computed by induction using the equation $H(z,Y(z))=0$ and the initial
condition $Y(0)=0$ (the assumption $H'_y(0,0)\not =0$ is crucial here). Note that
these coefficients are real, so that Lemma~\ref{lem:real} will
apply. But let us first prove that $Y$ has a positive radius of convergence.
Since $H'_y  (0,0) > 0$, the analytic implicit function theorem at
$z=0$ implies the existence of a locally analytic solution $\hat Y$
to the implicit equation $H(z,\hat Y(z)) = 0$ satisfying $\hat Y(0)=0$.
The expansion of $\hat Y$ around $0$ must satisfy this equation as well (in
the world of formal power series), and thus coincides with $Y$. Hence
$Y$ has a positive radius.

Now consider the set 
$$
I = \left\{\rho >0 \ \left|\ \rho\textrm{ satisfies conditions
  (a), (b), (c), (d)} \right.\right\}.
$$ 
This is clearly an open interval of the form $(0, \rt)$, and it is
non-empty since  (a), (b), (c) and (d) hold in the neighbourhood of 0. 
Assume that none of the properties (i), (ii), (iii)
and (iv) hold at $\rt$. In particular, $\rt$ is finite. We will reach a
contradiction by proving that $\rt \in I$.

Since (iv) does not hold, $Y$ is  bounded on $[0,\rt)$. By
continuity, the set of accumulation points of $\{Y(z), z \in [0,
  \rt)\}$ is an interval, which coincides with $[\liminf_{z\rightarrow \rt^-}
  Y(z),\limsup_{z\rightarrow \rt^-} Y(z)]$.  For each $y$
  in this interval,  the point  $(\rt,y)$ is in the closure of the
  set $\{(z,Y(z)), z \in [0,\rt)\}$     where $H$ is known to be
    analytic. Since (iii) does not   hold, there exists an element  $\tilde y $
  in this interval such that $H$ is analytic at
  $(\rt,\tilde y )$. 
In particular, it is continuous at this point, and (c)
implies that  $H(\rt, \tilde y )=0$.
Finally, since (d) holds, but  (ii) does not,  $H'_ y (\rt,\tilde y ) > 0$.

 These three properties allow us to apply the analytic implicit function
 theorem: there exists an analytic function $\tilde Y$ defined in a
 neighbourhood of $\rt$ such that $H(z,\tilde Y(z)) = 0$ and $\tilde
 Y(\rt)=\tilde y $.  We want to prove that $\tilde Y$ is an analytic
 continuation  of $Y$ at $\rt$, so that, in particular, the interval
 $[\liminf_{z\rightarrow \rt^-}   Y(z),\limsup_{z\rightarrow \rt^-} Y(z)]$ is reduced to the point
 $\tilde y$.

Since $H'_y(\rt,\tilde y )>0$, there exists $\delta>0$ and a complex
neighbourhood $V$ of $(\rt,\tilde y )$ such that for $(x,y)$ and  $(x,y')$ in $V$,
$$
|H(x,y)-H(x,y')|\ge \delta |y-y'|.
$$
We can also assume that $\tilde Y(x)$ is well-defined for
$(x,y)\in V$. 

Since $(\rt,\tilde y  )$ is an accumulation point of $\{(z, Y(z)), z \in (0, \rt)\}$, and $Y$
is continuous, there exists an interval $[z_0,z_1]\subset (0, \rt)$ such that
$(z,Y(z))\in V$ for $z\in [z_0,z_1]$. Then for $z$ in this interval,
$$
0=|H(z,Y(z))-H(z,\tilde Y(z))|\ge \delta |Y(z)-\tilde Y(z)|,
$$
which shows that the analytic functions $Y$ and $\tilde Y$ coincide on
$[z_0,z_1]$. So they coincide where they are both defined, and $\tilde Y$
is an analytic continuation of $Y$ at $\rt$. This tells us that (a)
holds at $\rt$. Now (b) also holds by the choice of $\tilde y$, (c)
holds by construction of $\tilde Y$, and (d) holds as well, as argued
above.
Thus $\rt$ belongs to $I$, which cannot be true since it is  the supremum of the open
interval $I$. Hence one of the properties  (i), (ii), (iii)
and (iv) must hold.
\end{proof}

\begin{cor} \label{lemmeasympt}
Let  $\Omega(y)$ be a real power series 
such that $\Omega(0) = 0$ and $\Omega'(0) >0$. 
Let $\omega\in (0, +\infty ]$ be the first singularity of $\Omega$
on the positive real axis, if it exists, and $+\infty$ otherwise.
Let $Y\equiv Y(z)$ be the unique  power series satisfying $Y(0)=0$
and $\Omega(Y(z)) = z$. Then $Y$ has a non-zero radius of convergence. 
Moreover, there exists $\rho \in (0, \infty]$ such that:
\begin{enumerate}
	\item $Y$ has an analytic continuation, still
          denoted by $Y$, in a neighbourhood of $[0,\rho)$, 
which is real valued,
	\item $Y$ is increasing on $[0,\rho)$,
\item 
$Y(z) \in [0, \omega)$ for  $z \in [0,\rho)$,
	\item $\Omega(Y(z)) = z$ for $z \in [0,\rho)$,

	\item   $\lim_{z\rightarrow \rho ^- }Y(z)= \tau$ and $\lim_{y\rightarrow \tau^-}\Omega(y)= \rho$,
\end{enumerate}
where 
$$
\tau = \min \{ y \in [0,\omega)\ |\ \Omega'(y) = 0\}
$$ 
if this set is non-empty, and $\tau=\omega$ otherwise. 
\end{cor}
\noindent This result is stated as an existence result for $\rho$,
but~(5) actually \emm determines, the value of  $\rho$. 
\begin{proof} 
We specialize Proposition~\ref{lemmeasympt2} to $H(x,y)= \Omega(y)
-x$. Clearly, $H$ is analytic around $(0,0)$, $H(0,0)=0$ and
$H'_y(0,0)=\Omega'(0)>0$. 
We take for  $\rho$ the value $\rt$ of Proposition~\ref{lemmeasympt2}. Then (1) 
 follows from (a).
Conditions (b) and (c) tell us that $\Omega$ has an analytic continuation on
$\{Y(z), z \in [0, \rho)\}$, such that $\Omega(Y(z))=z$ for $z\in [0, \rho)$.
By differentiating this identity, we obtain $Y'(z)
\Omega'(Y(z))=1$, so that (2) now follows from (d). Thus the existence of
an analytic continuation of $\Omega$ on
$\{Y(z), z \in [0, \rho)\}$ now translates into  (3).
%
The monotonicity of $Y$ also  allows us to define $Y(\rho):=\lim_{z\rightarrow \rho ^- }Y(z)$,
which is not necessarly finite.

Let us now derive (5) from the second series of properties of
Proposition~\ref{lemmeasympt2}. We have already seen (this is (3))
that $Y(\rho)\le \omega$. By Condition (d) of
Proposition~\ref{lemmeasympt2}, and by definition of $\tau$, the value
$Y(\rho)$ is also
less than or equal to $\tau$. Assume $Y(\rho) < \tau$. Then $\Omega$ is
analytic at $Y(\rho)$, and by continuity of $\Omega$ and $Y$, $\rho=\Omega(Y(\rho))<+\infty$,
so that (i) cannot hold. By definition of $\tau$, we cannot have
 (ii). It is easy to see that Conditions (iii) and (iv)  do not
hold either. So we have reached a contradiction, and $Y(\rho) =
\tau$. Returning to (4)  gives  
$\rho = \Omega(Y(\rho)) = \Omega(\tau)$.
\end{proof}

\section{Inversion of functions with a $z\ln z$ singularity}
\label{sec:inversion}
The inversion of a locally injective analytic function is a well-understood
topic: if $\Psi$ is analytic in the disk $C_s$ of radius $s$ centered at $0$ and $
\Psi(z)\sim z$ as $z\rightarrow 0$, then there exist $\rho \in (0, s)$
and $\rho'>0$, and a function $\Upsilon$ analytic on $C_\rho '$, 
taking its values in  $C_\rho $, 
such that 
$$
\forall (y,z) \in C_\rho '\times C_\rho , \quad \quad
\Psi(z)=y \Longleftrightarrow z= \Upsilon(y).
$$ 
The aim of this section is to see to what extent this can be
generalized to a function $\Psi(z)$ having a singularity in $z \ln z$
around $0$. Of course we cannot consider disks anymore, and our local
domains will be of the following form:
$$
D_{\rho , \alpha}:= \{z =re^{i\theta}: r \in (0,\rho )\ \hbox{ and } \ 
|\theta|<\alpha\}.
$$
\begin{theo}[{\bf Log-Inversion}]
 \label{lninversion}
Let $\Psi$ be analytic on $D_{s,\pi}$ for some $s > 0$. 
Assume that as $z$ tends to $0$ in this domain,
$$
\Psi(z) \sim -c z \ln z
$$
with $c>0$. 
Then for each $\alpha \in (0,\pi)$, there exist $\rho  \in( 0,s)$ and $\rho '
> 0$, and  a function $\Upsilon$ analytic in $ D_{\rho ',\alpha}$, taking its
 values in $D_{\rho , \pi}$, that satisfies
$$
\forall (y,z) \in D_{\rho ', \alpha} \times D_{\rho , \pi}, \quad \quad
\Psi(z)=y \Longleftrightarrow z= \Upsilon(y).
$$ 
%
%
Moreover, as $y\rightarrow 0$ in  $ D_{\rho ',\alpha}$,
$$
\Upsilon(y) \sim - \frac y{c \ln y}.
$$
\end{theo}
The proof is rather long. The most difficult part is to prove the
existence of a unique preimage of $y$ under $\Psi$ in $
 D_{\rho , \pi}$, for each $y\in D_{\rho',\alpha}$
 (Lemma~\ref{psinj}). This preimage is of course
 $\Upsilon(y)$. Proving the analyticity of $\Upsilon$ is then  a
 simple application of the analytic implicit function theorem.   In
 order to prove Lemma~\ref{psinj}, we first study the injectivity and
 surjectivity of the function $H: z\mapsto -z \ln z$ around
 0 (Section~\ref{sec:H}), before transferring them to the function~$\Psi$
 (Section~\ref{sec:H-Psi}).

\subsection{The function $z \mapsto -z\ln z$}
\label{sec:H}
Consider the following function
$$
\begin{array}{lcccllll}
  H : &\C\setminus \R^-& \rightarrow &\C
\\
&z &\mapsto &-z\ln z,
\end{array}$$
where $\ln$ denotes the principal value of logarithm: if
$z=re^{i\theta}$ with $r>0$ and $\theta \in (-\pi, \pi)$, then $\ln z= \ln
r+i\theta$. We also define $\Arg z:=\theta$.
Let us begin with a few elementary properties of $H$.
\begin{lem} \label{argH}
The function $H$ satisfies
\beq\label{H-conj}
H(\bz) =\overline {H(z)}.
\eeq
For $z=re^{i\theta}$ with $r>0$ and $\theta \in (-\pi, \pi)$,
\beq\label{module-H}
|H(z)|=
r \sqrt{\ln ^2 r + \theta^2}.
\eeq
The arguments of $z$ and $-\ln z$ have opposite
signs. If in addition $r<1$, then $\Arg(-\ln z)\in (-\pi/2, \pi/2)$. Hence
\beq\label{H-a}
  \Arg H(z) = \Arg z + \Arg(-\ln z) 
=
\theta + \arctan\left(\frac{ \theta}{\ln{r}}\right).
\eeq
If in addition  $r \leq 1/ \sqrt e$, then
\beq\label{4e}
|\Arg H(z)| \leq \theta.
\eeq
In particular, $H(z)\not \in \R^-$.
\end{lem}
\begin{proof}
The first two identities are straightforward.
The first part of~\eqref{H-a} follows from the fact that the
arguments of $z$ and $-\ln z$ have opposite signs. The second part
follows from $\Arg (-\ln z)\in (-\pi/2, \pi/2)$. Let us now
prove~\eqref{4e}.  Assume $\theta\ge 0$. Then $\Arg(-\ln z)\le 0$ and
the first part of~\eqref{H-a} gives $\Arg H(z)\le \theta$.
Moreover, $\arctan x\ge x $ if $x\le 0$, and thus by the second part of~\eqref{H-a},
$$
\Arg H(z)  
\geq  
\theta + \frac{
\theta}{\ln{r}}
\geq \left(1-\frac{1} {\ln{(1/\sqrt e)}} \right) \theta
= - \theta
.$$
The case where $\theta \leq 0$ now follows using~\eqref{H-conj}.
\end{proof}
Observe that $H$ is not injective on $\C$: for instance, $H(i) =
{\pi}/{2} =H(-i)$. However, $H$ is  injective
in a (slit) neighbourhood of $0$.
\begin{prop} \label{Hinj}
The function $H : z \mapsto -z\ln z$ is injective on 
$D_{e^{-1},\pi}$.
\end{prop}
\begin{proof}
Assume there exist $z_1$ and $z_2$ in $D_{e^{-1},\pi}$ such that
$H(z_1) = H(z_2)$. By   Lemma~\ref{argH}, the value $H(z_1)$ is not
real and negative,  and thus  $\ln H(z_1) = \ln H(z_2)$.

This lemma also implies  that for $z \in D_{e^{-1},\pi}$, we have 
\beq\label{lnln}
\ln H(z) = \ln z + \ln(-\ln z).
 \eeq
 Hence
\beq \label{lll}
|\ln z_1- \ln z_2| = |\ln(-\ln z_1)-\ln(-\ln z_2)|.
\eeq
Let $\kappa= -\max (\ln |z_1|, \ln|z_2|)>1$. Then $-\ln z_1$ and
$-\ln z_2$ lie in $\{z\:|\:\Rea (z)\ge \kappa\}$.
This set is convex, so the  (vectorial) mean value inequality gives
$$
|\ln(-\ln z_1)-\ln(-\ln z_2)| \leq |\ln z_1- \ln z_2|\,\sup_{z \in
  [-\ln z_1,-\ln z_2]} |\ln '(z)| \le \frac 1 \kappa |\ln z_1-
\ln z_2|.
$$
Combining  this with~\eqref{lll} gives $|\ln z_1- \ln z_2|=0$,  so
that $z_1=z_2$.
\end{proof}

We now address the surjectivity of the map $H$.

\begin{prop} \label{inHD}
For $0<\alpha<\pi$ and $\rho$ small enough (depending on $\alpha$),  we have 
$$
 D_{-\,\rho\ln \rho,\alpha} \subset H(D_{\rho,\pi}).
$$
\end{prop}

\begin{proof} 
We are going to prove that the inclusion holds for every 
  $\rho\in(0,1/e)$ satisfying
\beq\label{cond-rho}
\arctan{\left(\frac{\pi}{|\ln{\rho}|}\right)}\le \pi-\alpha.
\eeq
Let us fix a complex number $se^{i \gamma}$ with $0<s<-\rho \ln \rho$
and $|\gamma|<\alpha$. We want to prove the existence of $r<\rho$ and $\theta
\in (-\pi, \pi)$ such that $H(re^{i\theta})=s e^{i \gamma}$. We
proceed in two steps.

\smallskip
\noindent  \textbf{(1) There exists a continuous function $\theta :
  (0,\rho) \rightarrow (-\pi,\pi)$ such that
  $\forall r \in(0,\rho)$,}
  $$
\Arg H(r\,e^{i\theta(r)})=\gamma.
$$

\noindent \emph{Proof\/}. Fix $r\:\in\:(0,\rho)$. For  $\theta \in
(-\pi,\pi)$,  Lemma~\ref{argH} gives
$$
f(r,\theta):=\Arg H(r e^{i\theta})=\theta + \arctan \left( \frac \theta
{\ln r}\right).
$$
Differentiating with respect to $\theta$ gives 
$$
f'_\theta(r,\theta) = 1 + \frac{1}{\left(1+\frac {\theta^2}
  {\ln^2 r} \right)\ln r } \geq 1 + \frac{1}{\ln r} > 0.
$$ 
Hence $f(r, \theta)$ is a continuous increasing function of $\theta$, sending 
$(-\pi,\pi)$ 
onto $(-\pi - \arctan (\pi /\ln r), \pi + \arctan (\pi
  /\ln r))$. Since $r<\rho$ and $\rho$ satisfies~\eqref{cond-rho}, this
interval contains $(-\alpha,  \alpha)$, and thus the value
$\gamma$. This proves the existence, and uniqueness (since $f$ increases), of $\theta(r)$.

Now in the neighbourhood of $(r, \theta(r))$, we can apply the
implicit function theorem to the equation $f(r,\theta)=\gamma$, and
this shows that $\theta$ is continuous on $(0, \rho)$.

\smallskip
\noindent \textbf{(2) 
There exists $r \in (0,\rho)$ such that $|H(r\,e^{i\theta (r)})|=s$. }

\noindent \emph{Proof\/}. The function 
$$
r\mapsto |H(r\,e^{i\theta(r)})|= r\sqrt{\ln^2 r+\theta(r)^2}
$$
is continuous on $(0, \rho)$. It tends to $0$ as $r$ tends to $0$, and
to a value at least equal to $-\rho\ln \rho$ as $r$ tends to $\rho$.
Since $\alpha <-\rho \ln \rho$,  the intermediate value theorem implies that there exists
$r\in(0,\rho)$ such that $|H(re^{i\theta(r)})|=s$. 

This completes the proof of the proposition.
\end{proof}

\subsection{The Log-Inversion Theorem}
\label{sec:H-Psi}

By combining Propositions~\ref{Hinj} and~\ref{inHD}, we see that for 
$\alpha \in (0, \pi)$ and $\rho$ small enough, every point of
$D_{-\rho\ln \rho, \alpha}$ has a unique preimage under $H$ in 
$D_{\rho, \pi}$. We now want to adapt
  this result to functions $\Psi$ that behave like $H$ in the
  neighbourhood of the origin.

\begin{lem} \label{psinj}
Let $\Psi$ be analytic 
on $D_{s,\pi}$ for some $s > 0$. Assume that as $z$ tends to 0 in
this domain,
$$
\Psi(z) \sim H(z) = -z \ln z.
$$
For all $\alpha \in (0,\pi)$, there exist $\rho \in( 0,s)$ 
and $\rho' > 0$ such that every point of $D_{\rho',\alpha}$ has a unique preimage under $\Psi$ in $D_{\rho,\pi}$.
\end{lem}

\begin{proof} By assumption,
$\Psi(z) -H(z)= o(z\ln z)=o(-|z| \ln |z|)$ 
as $z$ tends to 0. 
Let  $\rho \in( 0, s)$ be small enough for every $z 
\in D_{\rho,\pi}$ to satisfy
\begin{eqnarray}
|\Psi(z)-H(z)| &<& - \min\left(\frac 1 2,\sin\left(\frac{\pi-\alpha}
4\right)\right) |z| \ln |z|,
\label{eq-delta}
\\
1 + \frac 1 {\ln |z|} &>& \frac 1 2 + \frac \alpha {\pi + \alpha},
\label{c1}
%
\\
|z \ln z| &\leq &- 2 |z| \ln |z|,\label{c5}
\\
- \ln \frac{|z|}8 &\le& -2  \ln |z|,
\label{c6}
\end{eqnarray}
and assume moreover than $\rho$ is also small enough for the following
property to hold:
\beq
\label{inclusion}
 D_{- \frac \rho 8 \ln \frac \rho 8,\alpha} \subset H(D_{\frac \rho 8,\pi}).
\eeq
This inclusion is made possible by
Proposition~\ref{inHD}. Several of the above listed conditions can be
described  by an explicit upper bound on $\rho$ (for instance,
\eqref{c6} just means that $\rho \le e^{-8})$, but we will use them in
the above form and  find convenient to write them so.

Now fix $y_0 \in D_{\rho',\alpha}$ with $\rho' = - \frac \rho 8 \ln
\frac \rho 8$. We want to prove that $y_0$ has a unique preimage under
$\Psi$ in $D_{\rho, \pi}$. By~\eqref{inclusion} and  Proposition
  \ref{Hinj}, it has a unique preimage under $H$, denoted by $z_0$, in  $D_{\rho,
    \pi}$ (in fact, $|z_0|< \rho/8$). We thus want to prove that the
  functions $\Psi-y_0$ and 
    $H-y_0$ have the same number of roots in   $D_{\rho,    \pi}$, and we will
    do so using   Rouch\'e's theorem.

For $\varepsilon \in \left(0, {|z_0|}/ 8\right)$, let
$\Gamma\equiv  \Gamma^{(\varepsilon)}$ be the contour shown in
Figure~\ref{fig:contour}. The interior of  $\Gamma$ converges to $D_{\rho,
    \pi}$ as $\varepsilon \rightarrow 0$. Hence for $\varepsilon$
small enough, it contains the point $z_0$, and we just need to prove
that  $\Psi-y_0$ and    $H-y_0$ have the same number of roots inside
$\Gamma^{(\varepsilon)}$ for every small enough $\varepsilon$.
\begin{figure}[h!]
\includegraphics[scale=0.8]{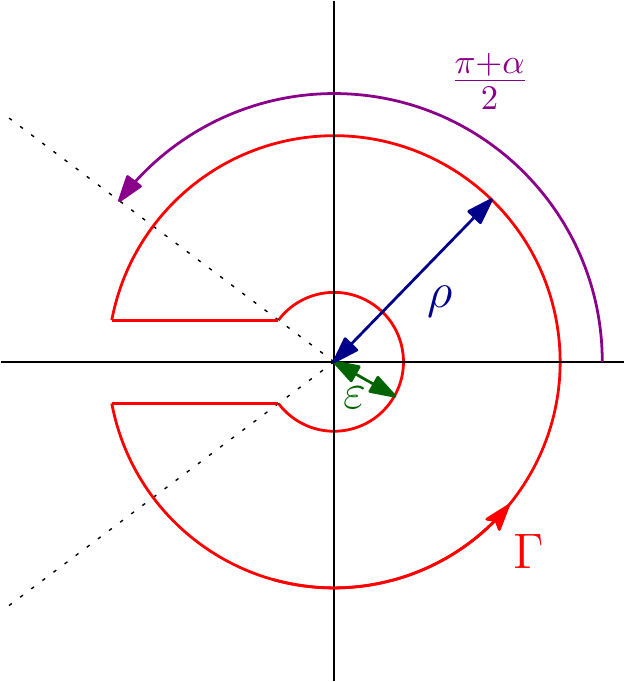} 
\caption{The contour $\Gamma^{(\varepsilon)}$.}
\label{fig:contour}
\end{figure}

By Rouch\'e's theorem, it suffices to show  that $|\Psi-H| < |H - y_0|$
on $\Gamma$. Let us  decompose $\Gamma$ into three (non-disjoint) parts :
$$
\Gamma_1 = \Gamma \cap \{ z: |z| = \rho \},
\quad \Gamma_2 = \Gamma \cap \{z:  |z| = \varepsilon \}
\quad \hbox{and} \quad \Gamma_3 = \Gamma \cap \left\{ z:|\Arg z| > \frac{\pi + \alpha} 2 \right\}.$$

We will use in the study of $\Gamma_1$ and $\Gamma_2$ the following
elementary result.
\begin{lem} \label{property1}
If $\rho \geq |z| \geq 8\,|z'|$ with $z, z' \in
  \C\setminus(-\infty, 0]$, then 
$$
|z\ln z - z' \ln z'| \geq - \frac 1 2 |z| \ln |z|.
$$
\end{lem}
\begin{proof}
We have the following lower bounds:
\begin{eqnarray*}
  |z\ln z - z' \ln z'| &\geq& |z \ln z| - |z' \ln z'|
\\
& \geq& -|z|\ln |z| + 2 |z'| \ln |z'|  \hskip 18mm \hbox{by  \eqref{module-H} 
and \eqref{c5}},
\\
&\geq&  -|z|\ln |z| + \frac {|z|} 4  \ln \frac{|z|} 8 \hskip 20mm
\hbox{because } |z'|\le |z|/8,  
\\
&\geq &- \frac 1 2 |z| \ln |z| \hskip 35mm
\hbox{by \eqref{c6}}.
\end{eqnarray*}
\end{proof}

Since $|z_0|<\rho/8$, we can apply this lemma to $z\in \Gamma_1$
and $z'=z_0$. This gives, using \eqref{eq-delta}:
$$
|H(z) - y_0| \geq - \frac 1 2 |z| \ln |z| > |\Psi(z)-H(z)| .
$$
Since $\varepsilon < |z_0|/8$, applying Lemma~\ref{property1} to $z_0$ and
$z \in \Gamma_2$ gives:
$$
|y_0-H(z)| \geq - \frac 1 2 |z_0| \ln |z_0|  \geq - \frac 1 2 |z| \ln
|z| > |\Psi(z)-H(z)|.
$$
We are left with the contour $\Gamma_3$. If $z\in \Gamma_3$, we claim that
\beq\label{arg-3}
|\Arg H(z)| \geq \alpha + \frac {\pi - \alpha} 4.
\eeq
By~\eqref{H-conj}, it suffices to prove this when $\Arg z\ge 0$. In this case,
\begin{eqnarray*} 
\Arg H(z) &=&  \Arg z +
  \arctan\left(\frac{\Arg z} {\ln |z|}\right) \hskip 18mm \hbox{by \eqref{H-a}},
\\
 &\geq &\left(1 +
  \frac 1 {\ln {|z|}}\right)  \Arg z  \hskip 28mm \hbox{since } \arctan x\ge x,
\\ & >& \left(\frac 1
  2 + \frac \alpha {\pi + \alpha}\right)  \Arg z  \hskip 26mm \hbox{by \eqref{c1}},
\\ &
  \geq &\alpha + \frac {\pi - \alpha} 4   \hskip 40mm \hbox{since }
  \Arg z \ge \frac {\pi +\alpha} 2.
\end{eqnarray*}
Hence~\eqref{arg-3} holds on $\Gamma_3$. But since $ |\Arg H(z_0)|=
|\Arg y_0|< \alpha$, we have 
\beq\label{arg-diff}
|\Arg H(z)-\Arg H(z_0)| > \frac {\pi - \alpha} 4.
\eeq

We still need one more result to conclude.
\begin{lem} For $\beta >0$ and  complex numbers $a$ and $b$ in
  $\C\setminus (-\infty,0]$,
$$
|\Arg a-\Arg b|\geq \beta \ \ \Longrightarrow \ \ |a - b| \geq
|a|\sin\beta.
$$
\end{lem}
\begin{proof} (1) If $|\Arg a-\Arg b| \leq \frac \pi 2$, then
$$
|a-b|
= \left||a|e^{i (\Arg a-\Arg b)} - |b| \right| \geq  \left|\Ima \left(|a|e^{i( \Arg a-\Arg b)}\right)\right| \geq |a|\sin\beta.$$
(2) If $|\Arg a-\Arg b| \geq \frac \pi 2$, then
\begin{multline*} 
|a-b|
 = \left||a| - |b|e^{i( \Arg b-\Arg a)} \right|  \geq
 \left|\Rea \left(   |a| - |b|e^{i( \Arg b-\Arg a)}
 \right)\right| \\ 
 = |a| - |b| \cos\left(\Arg b-\Arg a\right)  \geq |a| \geq |a|\sin\beta.
\end{multline*}
\end{proof}
By applying this lemma to~\eqref{arg-diff} with $\beta= (\pi-\alpha)/4$, we obtain
\begin{eqnarray*}
 |H(z) - y_0| &\geq & |H(z)| \sin\left(\frac{\pi-\alpha} 4\right)
\\
& \geq &-  \sin\left(\frac{\pi-\alpha} 4\right) |z| \ln |z|  \hskip
20 mm \hbox{by \eqref{module-H}}\\
&>& |\Psi(z)-H(z)| \hskip 34mm \hbox{by \eqref{eq-delta}}.
\end{eqnarray*}
We have finally proved that $|\Psi(z)-H(z)|<|H(z)-y_0|$ everywhere on
the contour $\Gamma\equiv \Gamma^{(\varepsilon)}$, and we can now conclude that
$\Psi-y_0$ has, like $H-y_0$, a unique root in $D_{\rho, \pi}$.
\end{proof}

We are finally ready to prove the log-inversion theorem
(Theorem~\ref{lninversion}). 

\begin{proof}[Proof of Theorem~\ref{lninversion}]
Upon writing $\Psi=c \Psi_1$ and $\Upsilon(y)=\Upsilon_1(y/c)$, we can
assume without loss of generality that $c=1$. We then choose $\rho$
and $\rho'$ as in Lemma~\ref{psinj}. For $y_0 \in D_{\rho ' ,\alpha}$, we define $\Upsilon(y_0 )$ as the unique point
$z_0$ of $D_{\rho, \pi}$  such that $\Psi(z_0 ) = y_0$.
We now apply the analytic implicit function theorem to the equation $\Psi(\Upsilon(y)) = y$, in the
neighbourhood of $(y_0 , z_0 )$. The function $\Psi$ is analytic at $z_0$ and locally injective by Lemma~\ref{psinj}.
Therefore $\Psi' (z_0 ) \not = 0$, and there exists an analytic
function $\bar \Upsilon$ defined in the neighbourhood of $y_0$ such
that $\bar \Upsilon(y_0)= z_0$ and $\Psi (\Upsilon(y)) = y$ in this neighbourhood.

This forces $\Upsilon(y)$ and $\bar \Upsilon(y)$ to coincide in a neighbourhood of $y_0$, and implies that $\Upsilon$ is analytic
at $y_0$ --- and hence in the domain $D_{\rho', \alpha}$.

Let us conclude with the singular behaviour of $\Upsilon$ near
$0$. The equation $\Psi(\Upsilon(y))=y$, combined with $\Psi(z)\sim -z
\ln z$, implies that $\Upsilon(y) \rightarrow 0$ as $y\rightarrow 0$. Thus
$$
y \sim -\Upsilon(y)\ln(\Upsilon(y))
$$
 as  $y\rightarrow 0$. Upon
taking logarithms, and using~\eqref{lnln}, this gives
$$
\ln y \sim \ln(\Upsilon(y)) +\ln (-\ln(\Upsilon(y))) \sim
\ln(\Upsilon(y)) .
$$
Combining the last two equations
finally gives $\Upsilon(y) \sim
-y/\ln y$.
\end{proof}

\section{Asymptotics for 4-valent forested  maps}
\label{sec:asympt-4}

 Let $F(z,u) = \sum_{n} f_n(u) z^n$ be the \gf\ of 4-valent forested
 maps, given by Theorem~\ref{thm:equations}. That is, $f_n(u)$ counts
 forested 4-valent maps with $n$ faces by the number of non-root
 components. As recalled in Section~\ref{sec:tutte}, the polynomial
 $f_n(\mu-1)$ has several interesting combinatorial descriptions in
 terms of maps equipped with an additional structure, and we will
 study the asymptotic behaviour of $f_n(u)$ for  any $u\ge -1$. 

Recall that $F(z,u)$ is characterized by~\eqref{system-simple} where
$\theta$ and $\Phi$ are given by~\eqref{theta-phi-4V}.
As discussed after Theorem~\ref{thm:equations}, $F(z,0)$ is explicit
and given by~\eqref{Fz0}:
$$
  F(z,0)=\int \theta(z) dz= 
4 \sum_{i \geq 2} \frac{(3i-3)!}{(i-2)!i!(i+1)!} z^{i+1}, 
$$
which makes the case $u=0$ of the following theorem a simple
application of Stirling's formula.

\begin{theo}\label{tetra} 
Let $p=4$, and take $u\ge -1$.  The radius of convergence of
    $F(z,u)$ is
\beq\label{rho}
\rho_u =\tau -u\Phi(\tau) 
\eeq
where $\Phi$ is given by~\eqref{theta-phi-4V} and
$$\left\{
 \begin{array}{lll}
\tau=1/27 & \hbox{if } u\le 0,
\\
 1-u\Phi'(\tau)=0 &\hbox{if } u>0.
\end{array}\right.
$$
The later condition determines  a unique $\tau \equiv \tau_u$
 in $(0,1/27)$. 

In particular, $\rho_u$  is an  affine function of $u$  on  $[-1,0]$:
\beq\label{rho-affine}
\rho_u= \frac 1 {27} -u\Phi\left(\frac 1 {27}\right)= \frac {1+u} {27}
- u \frac{\sqrt 3}{12\pi} .
\eeq
The function $\rho_u$ is decreasing, real-analytic everywhere except
at $0$, where it is still infinitely differentiable: as $u\rightarrow
0^+$,
\beq\label{rho-0}
\rho_u=
\frac 1 {27} -u\Phi\left(\frac 1 {27}\right)+O\left(\exp\left(-
\frac{2\pi}{\sqrt 3 u}\right)\right).
\eeq

Let    $f_n(u)$ be the coefficient of $z^n$ in $F(z,u)$. There exists a
    positive constant $c_u$
%
such that
$$
f_n(u) \sim \left\{\begin{array}{lll}
\displaystyle  c_u\, {\rho_u^{-n}}n^{-3}(\ln n)^{-2}& \hbox{ if } u\in[-1,0),\\
\displaystyle c_u\, {\rho_u^{-n}}{n^{-3} }  & \hbox{ if }  u = 0,\\
\displaystyle  c_u \,{\rho_u^{-n}}{n^{- 5 /2}}  & \hbox{ if }  u > 0.
\end{array}\right.
$$
The constant $c_u$ is given explicitly in Propositions~\ref{prop:4v-pos} (for
$u>0$) and~\ref{prop:4v-neg} (for $u<0$), and $c_0=2/(9\sqrt 3 \pi)$.  
\end{theo}
The exponent $-5/2$ found for
$u>0$ is standard for planar maps (see for instance Tables 1 and 2 in~\cite{banderier-maps}). The behaviour for $u<0$ is
much more surprising, and, to our knowledge, it is the first time that
it is observed in the world of maps. 
A plot of $\rho_u$ is shown in Figure~\ref{fig:rhou}. Note that
$\rho_{-1}=\sqrt 3/(12 \pi)$, a transcendental radius for the series
counting 4-valent maps equipped with an internally inactive spanning
tree.

\begin{figure}[h!]
\begin{center}
\includegraphics[scale=0.3]{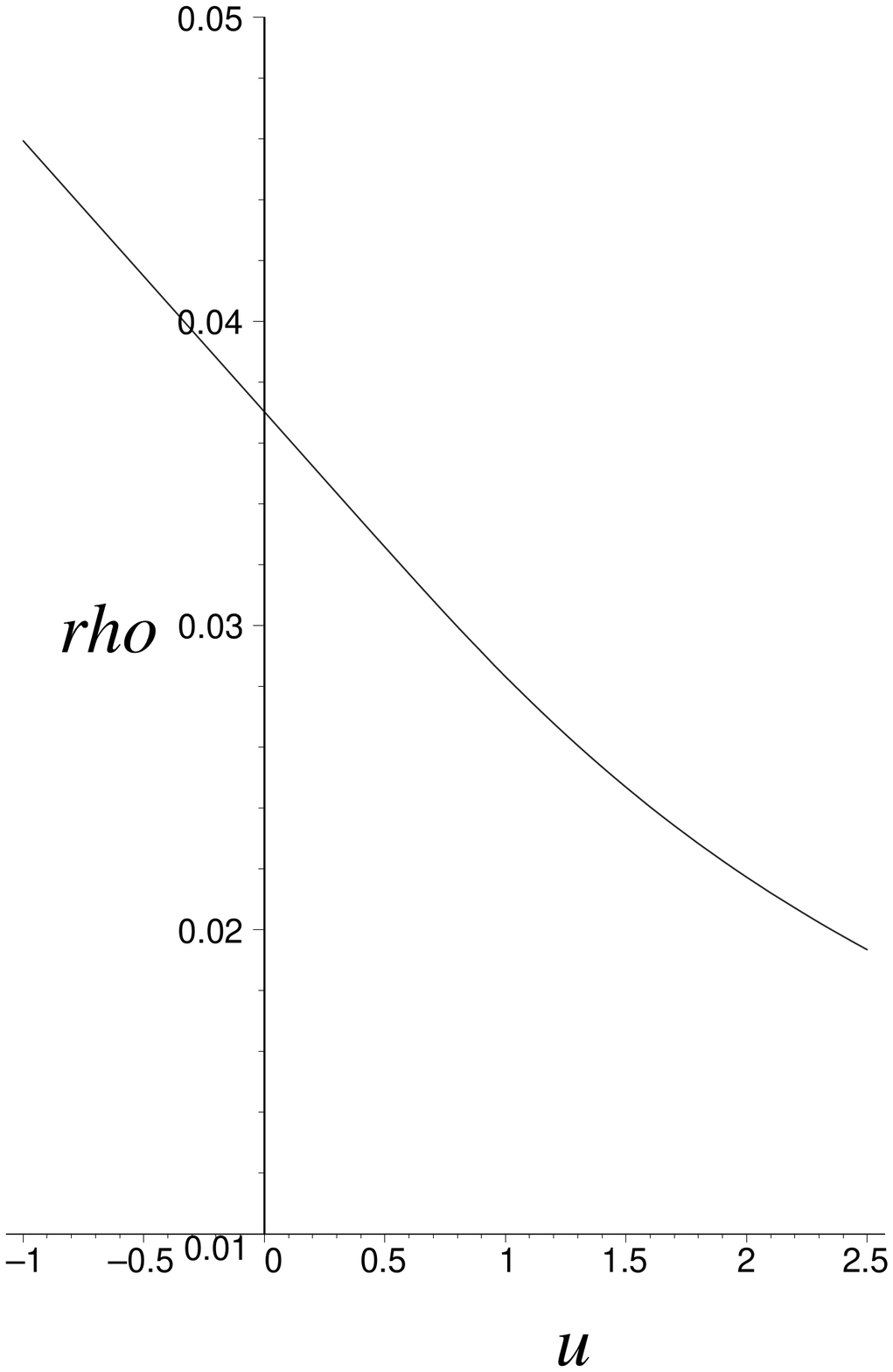}
\end{center}
\vskip -10mm
\caption{The radius $\rho_u$ of $F(z,u)$, as a function of $u\ge -1$.}
\label{fig:rhou}
\end{figure}

The proof of  the theorem uses the \emm singularity analysis, of~\cite[Ch.~VI]{flajolet-sedgewick}. We
thus need to locate the \emm dominant, singularities of the series $F'$
(that is, those of minimal modulus), and to find how $F'$ behave in their
vicinity. In order to do this, we begin with the series $R$, defined
 by $R=z+u\Phi(R)$, and then  move to $F'=\theta(R)$.  We
will find that both series have the same radius $\rho_u$. Moreover,
since $F'$ and $\bu(R-z)$ have non-negative coefficients in $z$,
this radius is a singularity of each (by Pringsheim's theorem). We
will prove that neither $F'$ nor $R$ have other dominant
singularities, and obtain  estimates of these functions near $\rho_u$
(the same estimate, up to a multiplicative factor).

Now the location of $\rho_u$, and its nature as a singularity, depend on
whether $u>0$ or $u<0$ (Figure~\ref{fig:R4valent}).
For $u>0$, the series $R$ will be shown to satisfy the \emm smooth
implicit schema, of~\cite[Sec.~VII.4]{flajolet-sedgewick}. In brief,
the dominant singularity $\rho_u$ of
$R$ comes from the failure of the  assumption $u\Phi'(R(z))\not =1$ in the implicit
function theorem. The value $R(\rho_u)$ lies in the analyticity
domain  of $\Phi$ and $\theta$, and the singularities of these series
play no role. Both $R$ and $F'$ will 
be proved to have a square root dominant singularity. If $u<0$ however, the
series $R$ reaches at $\rho_u$ the dominant singularity of $\Phi$ and
$\theta$, and the singular behaviours of $R$ and $F'$ at $\rho_u$
depend on the singular behaviours  of $\Phi$ and $\theta$. In particular, we find that, around
$\rho\equiv \rho_u$, the function $F''(z,u)$ behaves like
$1/\ln(1-z/\rho)$, up to a multiplicative constant. 
 Since this cannot be the singular behaviour of a D-finite series~\cite[p.~520
  and~582]{flajolet-sedgewick}, we have the following corollary.
\begin{cor} 
\label{cor-nonDF}
For $u\in[-1,0)$, the generating function $F(z,u)$ of $4$-valent forested
  maps is not D-finite. The same holds when $u$ is an
  indeterminate. 
\end{cor}
Recall that $F(z,u)$ is, however, differentially algebraic (Theorem~\ref{Dalg}).

\begin{figure}[h!]
\begin{center}
 \includegraphics[scale=0.3]{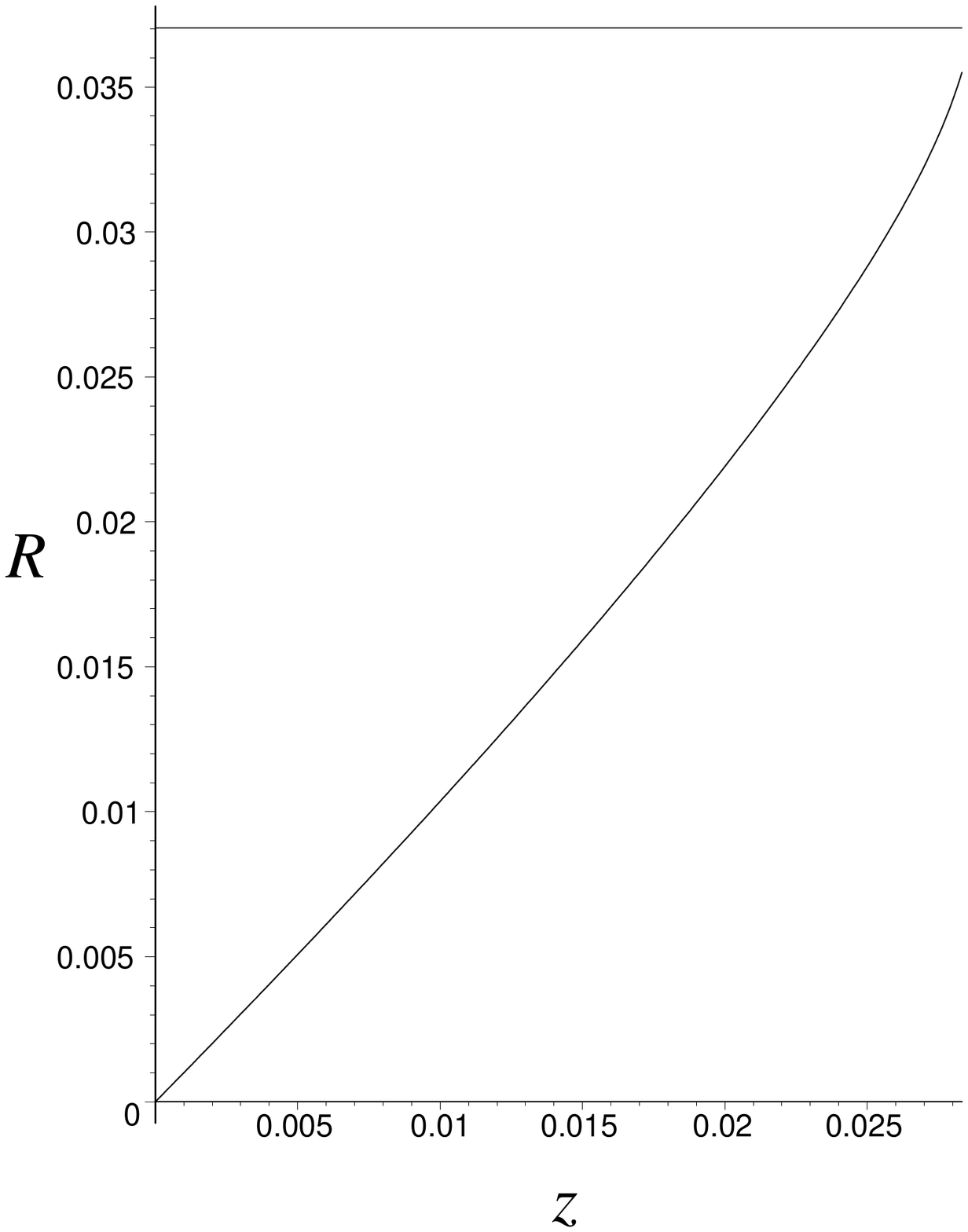}
\hskip 5mm
 \includegraphics[scale=0.3]{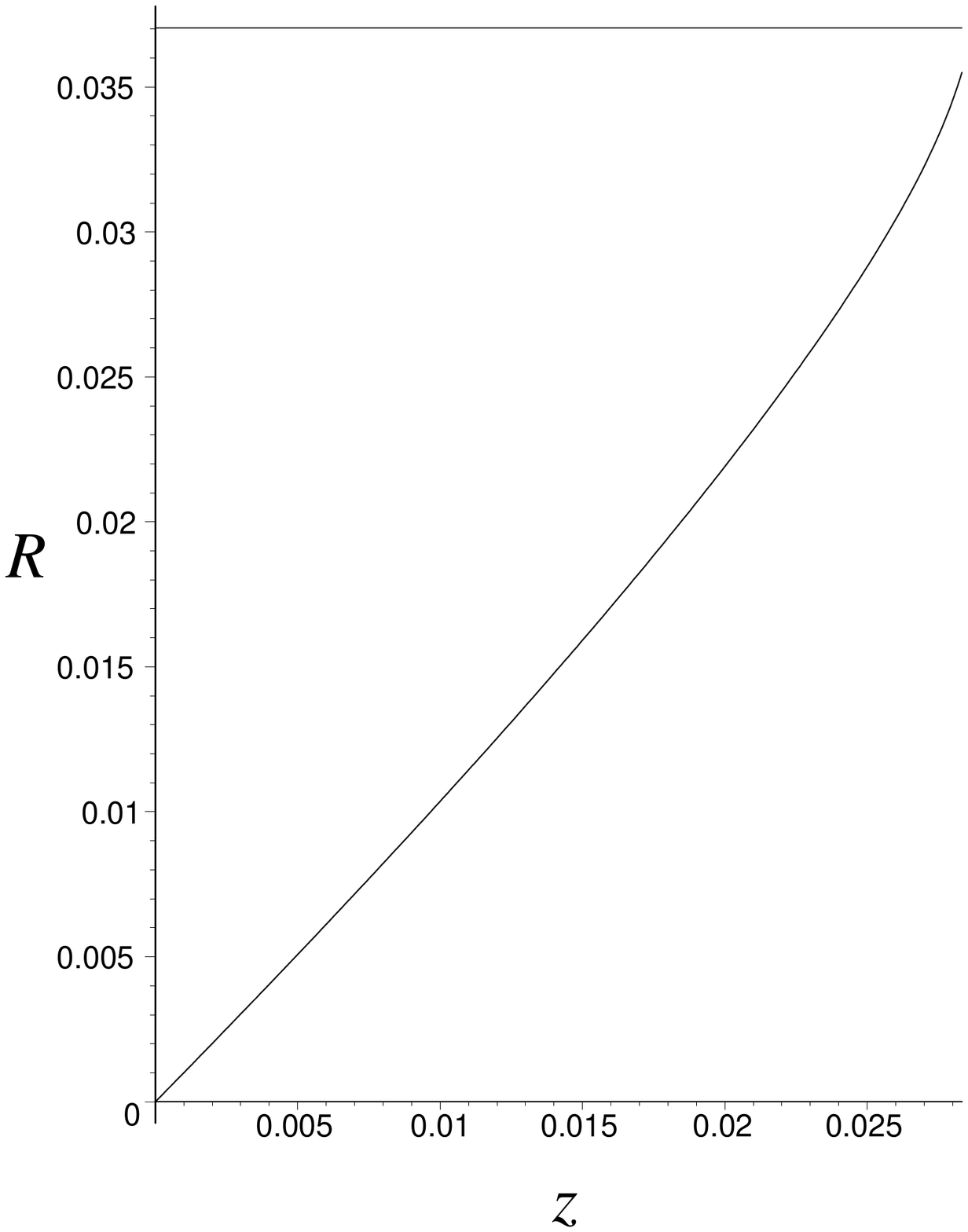}
\end{center}
\vskip -10mm
\caption{Plot of $R(z,u)$, for $z\in[0, \rho_u]$. \emph{Left:} when
  $u=1$, and more generally $u>0$, $R$ does not reach the dominant singularity of $\Phi$ (which
  is $1/27\simeq 0.037$). \emph{Right:} When $u=-1/2$, and more generally when $u
  \in [-1, 0]$, we have $R(\rho_u)=1/27$.} 
\label{fig:R4valent}
\end{figure}

\subsection{{The series $\Phi$ and $\theta$}}
\label{ssec:prelim}
Recall the definition~\eqref{theta-phi-4V} of these series. The $i$th
coefficient of $\theta$  is asymptotic to  $27^i/i^2$, up to a
multiplicative constant, and the same holds for $\Phi$. Hence both
series have  radius of convergence $ 1/ {27}$, converge at this point,
but their derivatives diverge.

This is as much information as we need to obtain the asymptotic
behaviour of $f_n(u)$ when $u>0$. When $u<0$, we will
need to know singular expansions of $\Phi$ and $\theta$ near $1/27$.
Let us first  observe that 
\beq\label{F-hg}
 \Phi(x) = x \left(  _2F_1\left(\frac 1 3,\frac 2 3;2;27 x\right) - 1 \right) 
\eeq
  where $_2F_1(a,b;c;x)$ denotes the standard hypergeometric function
  with parameters $a$, $b$ and $c$:  
$$
  _2F_1(a,b;c;x) = \sum_{n\geq0} \frac{(a)_n(b)_n} {(c)_n} \frac{x^n}{n!}, 
$$
with $(a)_n$ the rising factorial $a(a+1)\cdots(a+n-1)$. The series
$_2F_1\left(\frac 1 3,\frac 2 3;2;27 x\right)$ 
can be analytically continued in $\cs\setminus [1/27, +\infty)$, and
its behaviour as $z$ approaches $1/27$ in this domain is given
by~\cite[Eq.~(15.3.11)]{AS}. Translated it terms of $\Phi$, this
gives, as $\varepsilon \rightarrow 0$, 
\beq\label{sing-phi}
 \Phi\left(\frac{1 }{27}- \varepsilon\right) =
 \frac{\sqrt{3}}{12\pi} - \frac{1}{27} +
 \frac{\sqrt{3}}{2\pi}\,\varepsilon\ln{\varepsilon} + 
\left( 1 - \frac{\sqrt 3}{2 \pi}\right) \, \varepsilon + O(\varepsilon^2
 \ln{\varepsilon}).  
\eeq   
One also has:
\beq\label{phi-prime}
   \Phi'\left(\frac{1 }{27}- \varepsilon\right)=
 - \frac{\sqrt{3}}{2\pi}\,\ln{\varepsilon} 
-1+ O(\varepsilon \ln{\varepsilon}).
\eeq
The series  $\theta$  is related to $\Phi$ by~\eqref{thetrat}. It has the
same analyticity domain as $\Phi$, with local expansion at $1/27$:
\beq\label{sing-theta}
 \theta\left(\frac{1 }{27}- \varepsilon\right) = 
\frac 2 3 - \frac{7 \sqrt{3}}{6 \pi} 
+ \frac{2\sqrt{3}}{ \pi}\varepsilon \ln{\varepsilon} 
+ \frac{7\sqrt{3}}{ \pi}\varepsilon + O(\varepsilon^2 \ln{\varepsilon}). 
\eeq
Also,
\beq
\label{theta-prime}
\theta'\left(\frac{1 }{27}- \varepsilon\right)=
 -\frac{2\sqrt{3}}{ \pi}\ln{\varepsilon} 
- \frac{9\sqrt{3}}{ \pi}  + O(\varepsilon \ln{\varepsilon}). 
\eeq

\subsection{When $u>0$}
\label{ssec:pos&tetr}
As in~\cite[Def.~VI.1, p.~389]{flajolet-sedgewick}, we call
\emm $\Delta$-domain of radius $\rho$, any domain of the form
$$
\{z :|z|<r, z\not = \rho \hbox{ and } |\Arg (z-\rho)|> \phi\}
$$
for some $r>\rho$ and $\phi \in (0, \pi/2)$.
\begin{prop}\label{prop:4v-pos}
  Assume $u>0$. Then the series $R(z,u)$ is aperiodic and satisfies the
  smooth implicit   schema of~\cite[Def.~VII.4,
    p.~467]{flajolet-sedgewick}. Its radius is given
  by~\eqref{rho}, and satisfies~\eqref{rho-0}. The series $R$ is analytic in a 
$\Delta$-domain of radius $\rho\equiv \rho_u$, with a square root
singularity at $\rho$:
\beq\label{R-sing-pos}
R(z,u)= \tau -\gamma \sqrt{1-z/\rho } +O(1-z/\rho),
\eeq
where $\tau$ is defined as in Theorem~\ref{tetra}, and
$\gamma=\sqrt{\frac{2 \rho}{u\Phi''(\tau)}}$ with $\Phi$  given by~\eqref{theta-phi-4V}.

 The series $F'(z,u)$ is also analytic in a
$\Delta$-domain of radius $\rho$, with a square root
singularity at $\rho$:
\beq\label{Fprime-sing-pos}
F'(z,u)= \theta(\tau) -\gamma \theta'(\tau)\sqrt{1-z/\rho } +O(1-z/\rho),
\eeq
where $\gamma$ is given above  and $\theta$  is
defined by~\eqref{theta-phi-4V}. Consequently, the 
$n$th coefficient of $F$ satisfies, as $n\rightarrow \infty$,
$$
f_n(u) \sim  \theta'(\tau)  \sqrt{\frac{\rho^3}{2 \pi u \Phi''(\tau)}}
\rho^{-n} n^{-5/2}.
$$
\end{prop}
\noindent This proposition establishes the case  $u>0$ of Theorem~\ref{tetra}.
\begin{proof}
  The results that deal with $R$ are  a straightforward application of Definition~VII.4 and
  Theorem~VII.3 of~\cite[p.~467-468]{flajolet-sedgewick}. Using the
  notation of this book, $G(z,w)=z+u\Phi(w)$ is analytic for $(z,w)\in
  \cs\times \{|w|<1/27\}$. The so-called \emm characteristic system, holds
  at $(\rho,\tau)$ where $\tau$ is the unique element of $(0,1/27)$
  such that $G_w(\rho,\tau)=u\Phi'(\tau)=1$, and
  $\rho:=\tau-u\Phi(\tau)$. The existence and uniqueness of $\tau$ is guaranteed by the fact
  that $\Phi'(w)$ increases (strictly) from 0 to $+\infty$ as $w$ goes from 0
  to $1/27$.
The aperiodicity of $R$ is obvious from the first terms of its
expansion:
$
R=z+3z^2u+6u(3u+5)z^3+O(z^4).
$

\smallskip
We now move to $F'=\theta(R)$.
  Since $R(\rho,u)=\tau<1/27$, and $R$ has non-negative coefficients,
  there exists a $\Delta$-domain of radius $\rho$ in which  $R$ is
  analytic and strictly bounded (in modulus) by $1/27$. Since $\theta$
  has radius $1/27$, the series $F'=\theta(R)$  is also analytic in
  this domain, and its singular behaviour around $\rho$ follows from
a  Taylor expansion. One then applies the Transfer Theorem~VI.4
from~\cite[p.~393]{flajolet-sedgewick} to obtain the behaviour of the
$n$th coefficient of $F'$, which is
$(n+1)f_{n+1}(u)$. The estimate of $f_n(u)$ follows.

It remains to find an estimate of $\rho_u$ as $u\rightarrow
0^+$. Recall that $u\Phi'(\tau)=1$. Thus $\tau\equiv \tau_u$
approaches $1/27$ as $u\rightarrow 0$, and~\eqref{phi-prime} gives
$$
\ln (1/27-\tau)= -\frac{2\pi(1+\bu)}{\sqrt 3}+o(1)
$$
with $\bu=1/u$, so that
\beq\label{tau-0}
\tau -\frac 1 {27} \sim - \exp\left( -\frac{2\pi(1+\bu)}{\sqrt
  3}\right).
\eeq
Since $\rho=\tau -u \Phi(\tau)$, this gives~\eqref{rho-0} in view of
the expansion~\eqref{sing-phi} of $\Phi$.
\end{proof}

\subsection{When  $u<0$}
\begin{prop}\label{prop:4v-neg}
   Let $u \in [-1,0)$. The series $R$ and $F'$ have radius $\rho\equiv
     \rho_u$ given by~\eqref{rho-affine}. They are analytic in a
     $\Delta$-domain of radius $\rho$, and the following estimates
     hold in this domain, as $z\rightarrow \rho$:
     \begin{eqnarray}
      R(z) - \frac 1 {27} &\sim& -
\frac{2\pi\rho }{\sqrt 3 u}\, \frac{1-z/\rho}{\ln (1-z/\rho)},
  \label{Rsing-neg}
 \\
F''(z)+4\bu &\sim&   
\frac{72 \sqrt 3 \pi \bu^2 \rho}{\ln (1-z/\rho)}.
\label{F-sing-neg}
    \end{eqnarray}
Consequently, the $n$th coefficient of $F$ satisfies, as
$n\rightarrow \infty$,
$$
f_n(u)\sim  72 \sqrt 3 \pi \bu^2 
\frac{\rho^{-n+3}}{n^3 \ln ^2 n}.
$$
\end{prop}
\noindent Since~\eqref{F-sing-neg}  cannot be the singular behaviour of a
D-finite series~\cite[p.~520  and~582]{flajolet-sedgewick}, this
proves Corollary~\ref{cor-nonDF}. This proposition also establishes the
case $u<0$ of Theorem~\ref{tetra}.

\begin{proof}
We begin as before with the series $R$. The equation $R=z+u\Phi(R)$ reads
$\Omega(R)=z$ with $\Omega(y)=y-u\Phi(y)$. Clearly $\Omega(0)=0$ and
$\Omega'(0)=1 >0$, so that we can  apply
Corollary~\ref{lemmeasympt}, in which the role of $Y$ is  played by
$R$.  Let $\omega$, $\tau$ and $\rho$ be defined as in this corollary. 
It follows from Section~\ref{ssec:prelim} that $\omega=1/27$. Since $u<0$, $\Omega'(y)=1-u  
\Phi'(y)$ does not vanish on $[0,   1/27)$. Hence $\tau=1/27$ as well.
   By Property (5) of Corollary~\ref{lemmeasympt},
$$
\rho = \Omega\left(\frac 1{27}\right)=\frac 1 {27} - u \Phi\left(\frac 1 {27}\right),
$$
which, combined with~\eqref{sing-phi}, gives~\eqref{rho-affine}.

 Corollary~\ref{lemmeasympt} tells us that  $R$ has an analytic continuation along
$[0, \rho)$.
Moreover, $R(z)$ increases from $0$ to
  $1/27$  on $[0, \rho)$, and  the equation 
\beq\label{eqfuncR}
R=z+u \Phi(R)
\eeq
holds in the whole interval   $[0, \rho)$.

By Corollary~\ref{positiv}, the series $\bu(R-z)$ has non-negative
coefficients. As $R$ itself, it is analytic along $[0,\rho)$. By
  Pringsheim's theorem, its radius is at least $\rho$, and this holds
  for $R$ as well. We will now study the behaviour of $R$  in the neighbourhood
  of $\rho$, 
and prove  that it is singular at this point, so that
$\rho$ is indeed the radius of $R$.

For $z \in \C\setminus\R^-$,  let us define 
$$
\Psi(z) := \rho + \frac {z-1} {27} + u\,\Phi\left(\frac{1 -
  z}{27}\right).
$$ 
As explained above, $ 1-27\,R(\rho-y)$ increases from $0$
to $1$ as $y$ goes from $0$ to $\rho$, and the functional equation~\eqref{eqfuncR}
satisfied by $R$ reads, for $y  \in [0, \rho)$,
 \beq\label{eq-R-minus}
\Psi(1-27\,R(\rho-y)) = y.
\eeq
By~\eqref{sing-phi}, we have $\Psi(z) \sim  - c z \ln z$ where 
$$c=-\frac{\sqrt 3 u}{54 \pi}>0.$$ 
Let us apply the log-inversion theorem (Theorem~\ref{lninversion}) to $\Psi$,
with $\alpha=3\pi/4$ (we now denote $r$ and $r'$ the numbers $\rho$
and $\rho'$ of Theorem~\ref{lninversion}): There
exists $r>0$ and $r'>0$, and a function $\Upsilon$ analytic on $D_{r',
  \alpha}=\left\{|z| < r' \textrm{ and
} |\Arg z | <  3  \pi/4 \right\}$, such that $\Psi (\Upsilon(y)) =
y$. Furthermore,  $\Upsilon(y)$ is the only preimage of $y$ under $\Psi$
that can be found in $D_{r, \pi}=\left\{|z| < r \textrm{ and
} |\Arg z| <    \pi \right\} $. Comparing with~\eqref{eq-R-minus}
shows that  for $y$ small enough and positive, one has 
 $\Upsilon(y) = 1-27\,R(\rho-y)$. Returning
to the original variables, this means that, for $z$ real and close to $\rho^-$,
$$
R(z)=\frac 1 {27} \left( 1-\Upsilon(\rho-z)\right),
$$
 so that $R$ can be
analytically continued on $\left\{|z-\rho| < r \textrm{ and }
|\Arg(z-\rho)\ | > \pi/4\right\}$. Moreover, the final statement of
Theorem~\ref{lninversion} gives~\eqref{Rsing-neg}.
This shows that $R$ is singular at $\rho$, which is thus the
radius of $R$.

In order to prove  that $R$ is analytic in a $\Delta$-domain of
radius $\rho$, we now have to prove that it has no
singularity other than $\rho$ on its circle of convergence.
So let $\mu \neq \rho$ have modulus $ \rho$. Since $\Rc:=\bu(R-z)$ has positive coefficients and
$|\Rc(\rho)| < + \infty$, the series $\Rc$ converges at $\mu$, and so
does $R$.
Recall that $\Phi$ is analytic in $\C\setminus[1/27,
  +\infty)$. Hence~\eqref{eqfuncR}, which holds in a neighbourhood of
  $0$, will hold in the closed disk of center $\rho$ if we can prove 
the following lemma.

\begin{lem} For $|z|\le \rho$ and $z\not = \rho$, we have $R(z) \not
  \in [1/27, +\infty)$.
\end{lem}

\begin{proof} We have already seen that the property holds (since $R$
  is increasing) on the interval $[0, \rho)$. On the interval $[-\rho,
      0]$,  the function $R$ is real (Lemma~\ref{lem:real}) and
    continuous. Hence, if $R$ exits $(-\infty, 1/27)$ on this
    interval,  there exists $t \in [-\rho, 0]$ such that
    $R(t)=1/27$. Let $t$ be maximal for this property. Then $R(z)\in
    \C\setminus [1/27,  +\infty)$ on a complex  neighbourhood of
      $(t,0]$, and~\eqref{eqfuncR}  holds there. By differentiating
    it, we obtain
\beq\label{R-prime}
R'(z)= \frac 1 {1-u\Phi'(R(z))} \not = 0.
\eeq
In particular, $R'(0)=1$. But since $R(t)=1/27> R(0)=0$, the
function $R'(z)$ must vanish in $(t,0)$, which is impossible in view
of its expression above.

Assume now that $z$ is
    not real, and let us prove that $R(z)$ is not real either. First,
 \beq\label{tata}
|\Ima R(z)| 
= |\Ima (z+u\Rc(z))|
\geq |\Ima  z| +u\,|\Ima \Rc(z)|.
\eeq
Then:
\begin{multline}
\label{toto}
  |\Ima \Rc(z)|
=|\Ima \left(\Rc(z)-\Rc(\Rea z)\right)|
 \leq |\Rc(z) - \Rc(\Rea z)| 
\\
< |z - \Rea z|\max_{y \in [\Rea z,z]}|\Rc'(y)| \leq
|\Ima z|\max_{|y| \leq \rho}|\Rc'(y)|.
\end{multline}
The strict inequality comes from the fact that  $\Rc'$ is not constant
over $[\Rea z,z]$. But $\Rc'$ is a power series with positive
coefficients, and thus for $|y| \leq \rho$, 
\beq\label{Rc-prime}
|\Rc'(y)| \leq \Rc'(\rho) = \bu \left(R'(\rho) - 1\right) 
= \bu \left(\lim_{t \rightarrow \rho} \frac 1 {1 - u \Phi'(R(t))}-
1\right) = -\bu,
\eeq
because $\Phi'(z)$ tends to $+\infty$ as $z\rightarrow 1/27$.
Returning to~\eqref{toto} gives $ |\Ima \Rc(z)|
< -\bu |\Ima z|$, and this inequality, combined
with~\eqref{tata}, gives
$|\Ima R(z)| > 0$.
\end{proof}
So we now know that~\eqref{eqfuncR} holds everywhere in the disk of
radius $\rho$, with $R$ only reaching the critical value $1/27$ at
$\rho$. By differentiation, \eqref{R-prime} holds as well. Let us return to our point $\mu\not =\rho$,
of modulus $\rho$. We now want to apply the analytic implicit function
theorem to~\eqref{eqfuncR} at the point $(\mu, R(\mu))$. We know that
$\Phi$ is analytic around $R(\mu)$. Could it be that $u\Phi'(R(\mu))
=1$? By~\eqref{R-prime}, this would imply that $|R'(z)|$, and thus $|\Rc'(z)|$, is not
bounded as $z$ approaches $\mu$ in the disk. However, $\Rc'$ has
non-negative coefficients and $\Rc'(\rho)$ has been shown to converge
(see~\eqref{Rc-prime}). Thus $\Rc'(z)$ remains bounded in 
the disk of radius $\rho$, and in particular  $u\Phi'(R(\mu))
\not=1$.  The analytic implicit function theorem then implies that
$R$ is analytic at $\mu$.

In conclusion, we have proved that there exists a $\Delta$-domain of
radius $\rho$ where $R$ is analytic and avoids the half-line
$[1/27,+\infty)$.

\medskip
Let us now turn our attention to   $F'=\theta(R)$. Since
$\theta$ is analytic in  $\C\setminus[1/ {27},+\infty)$, the
  series  $F'$ is analytic in the same $\Delta$-domain as $R$.
The estimate~\eqref{Rsing-neg} of $R$, combined with  the
expansion~\eqref{sing-theta} of $\theta$, does not give immediately the singular behaviour of $F'$. Another route would be
possible, but it is more direct  to work with $F''$ instead. Indeed,
\beq\label{Fzz}
F''(z)= R'(z) \theta'(R(z)) = \frac{\theta'(R(z))}{1-u\Phi'(R(z))}.
\eeq
By~\eqref{phi-prime} and~\eqref{theta-prime},
\begin{eqnarray*}
  \frac{\theta'(1/27-\varepsilon)}{1-u\Phi'(1/27-\varepsilon)}&=& -4\bu -
{2}\bu\left(  {9} -\frac{4\pi(1+\bu)}{\sqrt 3}\right) \frac
1{\ln \varepsilon}+ O (1/{\ln^2 \varepsilon})\\
&=& -4\bu +  \frac{72 \sqrt 3 \pi \bu^2 \rho}
{\ln \varepsilon}+ O (1/{\ln^2 \varepsilon}), 
\end{eqnarray*}
in view of~\eqref{rho-affine}.  
This, combined with~\eqref{Fzz} and the
estimate~\eqref{Rsing-neg} of $R(z)$, gives~\eqref{F-sing-neg}.
One finally applies the Transfer Theorem~VI.4
from~\cite[p.~393]{flajolet-sedgewick} to obtain the behaviour of the
$n$th coefficient of $F''$, which is
$(n+2)(n+1)f_{n+2}(u)$. The estimate of $f_n(u)$ follows.
\end{proof}

\section{Large random  maps equipped with a forest or a tree}
\label{sec:random}
We still focus in this section on 4-valent maps, equipped either with
a spanning forest or with a spanning tree. In each case, we define a
Boltzmann probability distribution on maps of size $n$, which involves
a parameter $u$ and takes
into account  the number of components of the spanning forest, or the
number of internally active edges of the spanning tree (equivalently,
the level of a recurrent sandpile configuration, as explained in
Section~\ref{sec:tutte}). We observe on several random variables the
effect of the phase transition found at $u=0$ in the previous section.

\subsection{Forested maps: Number and size of components}
Fix $n \in \N$ and $u \in [0,+\infty)$. Consider the following
probability  distribution on 4-valent
forested maps $(M,F)$ having $n$ faces :
$$
\mathbb P_c(M,F) =  \frac {u^{c(F)-1}}{f_n(u)},
$$
 where $c(F)$ is the number of  components of $F$, and $f_n(u)$ counts
 4-valent forested maps by the number of non-root components. Under this
 distribution, let  $C_n$ be the number of  components of $F$, and $S_n$ 
 the size (number of vertices)  of the root component.
When $u=0$, only tree-rooted maps have a positive
probability, $C_n=1$ and $S_n=n-2$, the total number of vertices
in the map. Let us examine how this changes when $u>0$.
\begin{prop}
\label{random-forest}
Assume $u>0$. Under the distribution $\PP_c$, we have, as $n\rightarrow \infty$:
$$
\E_c(C_n) \sim \frac {u\Phi(\tau)}{\tau -u\Phi(\tau)}\, n,
$$
where 
$\Phi$ is given by~\eqref{theta-phi-4V} and $\tau\equiv \tau_u$  is the unique solution in
$\left(0,1/{27}\right)$ of $u\Phi'(y)=1$. 

The size $S_n$ of the root component
admits  a discrete limit law: for $k\ge 1$,
\beq
\label{PSn}
\lim_{n \rightarrow + \infty} \mathbb P_c(S_n=k) = 
\frac{4\,(3\,k)!\,}{(k-1)!\,k!\,(k+1)!}\frac{ \tau^k}{\theta'(\tau)}
\eeq
with $\theta$ defined by~\eqref{theta-phi-4V}.
\end{prop}

\begin{proof} 
We have 
\beq\label{ECn}
\E_c(C_n-1)= \sum_{(M,F)} \left(c(F)-1\right) \frac {u^{c(F)-1}}{f_n(u)}=u\frac{f'_n(u) }{f_n(u)}=
u\frac{[z^{n-1}]F''_{zu}(z,u) }
{[z^{n-1}]F'_{z}(z,u) }.
\eeq
It follows from the definition~\eqref{system-simple} of $R$ and $F$ that
\beq\label{Fzu}
F''_{zu}(z,u)= \frac{\Phi(R) \theta'(R)}{1-u\Phi'(R)}.
\eeq
We now  use singularity analysis. The functions $\Phi$ and $\theta$
are analytic at $\tau=R(\rho,u)$, the number $\tau$ satisfies
$1=u\Phi'(\tau)$, and a singular estimate  of $R-\tau$ is given
by~\eqref{R-sing-pos}. This gives, as $z\rightarrow \rho$,
$$
F''_{zu}(z,u)\sim \frac{\Phi(\tau) \theta'(\tau)}{u\Phi''(\tau) \gamma
  \sqrt{1-z/\rho}}
$$
where $\gamma$ is as in  Proposition~\ref{prop:4v-pos}. An estimate of
$F'_z(z,u)$  is given by~\eqref{Fprime-sing-pos}. Our estimate
of $\E_c(C_n)$ then follows from  a transfer theorem, and the fact
that $\rho=\tau-u\Phi(\tau)$.

\medskip
To study $S_n$, we add to our \gf\ $F(z,u)$ a weight $x$ for each
vertex belonging to the root component. Lemma~\ref{lem:cont} becomes
$$
 F (z,u,x)=  \bar M
(z,u;\,0,0,0,t_4,0, t_6, \ldots; 0, 0, 0, x\,t^c_4,0, x^2\,t^c_6, \ldots).
$$
(Recall that $t_2 = t_{2k+1} =t^c_2 = t^c_{2k+1} = 0$ for every $k \geq 0$ when
$p=4$.) Thanks to~\eqref{relM}, the first equation of~\eqref{system-simple} becomes
$$
 xF'_z(z,u,x) = \theta(x\,R),
$$
where $R=R(z,u)$ is as before.
We can express $\PP_c(S_n=k)$ is terms of $F'_z$:
$$
\PP_c(S_n=k)= \frac{[z^{n-1}x^k] F'_z(z,u,x)}{[z^{n-1}] F'_z(z,u,1)}=
\frac{[z^{n-1}x^{k+1}]  \theta(x\,R)}{[z^{n-1}] \theta(R)}.
$$
 We can now
 apply Proposition~IX.1 from~\cite[p.~629]{flajolet-sedgewick}. Proposition~\ref{prop:4v-pos}
 guarantees that its  hypotheses are indeed satisfied, and this
 gives~\eqref{PSn}, using the expression~\eqref{theta-phi-4V} of $\theta$.
\end{proof}

\subsection{Tree-rooted maps: Number of internally active edges}

Fix $n \in \N$ and $u \in [-1,+\infty)$. Consider the following
probability  distribution on 4-valent
tree-rooted maps $(M,T)$ having $n$ faces :
$$
\mathbb P_i(M,T) =  \frac {(u+1)^{i(M,T)}}{f_n(u)},
$$
where $i(M,T)$ is the number of internally active edges in
$(M,T)$. Eq.~\eqref{F-act} shows that this is indeed a probability
distribution. Under this
 distribution, let  $I_n$ denote the number of  internally active
 edges. As shown by~\eqref{F-sandpile}, $I_n$ can also be
 described as the level $\ell(C)$ of a  recurrent sandpile configuration $C$
 of an $n$-vertex quadrangulation $M$, drawn according to the  distribution
$$
\mathbb P_s(M,C) =  \frac {(u+1)^{\ell(C)}}{f_n(u)}.
$$
\begin{prop}
 The expected number of internally active edges undergoes a (very
 smooth) phase
 transition at $u=0$: as $n\rightarrow \infty$,
\beq\label{EIn}
\mathbb E_i(I_n) \sim \kappa_u \, n,
\eeq
with
$$
 \kappa_u = \frac {(1+u)\Phi(\tau)}{\tau -u\Phi(\tau)}
$$
where  $\Phi$ is given
by~\eqref{theta-phi-4V} and $\tau\equiv \tau_u$ is defined in Proposition~\ref{tetra}.
The function $\kappa_u$ is real-analytic everywhere except at $0$, where it is
still infinitely differentiable: as $u\rightarrow 0^+$,
$$
\kappa_u = \frac {(1+u)\Phi(1/27)}{1/27 -u\Phi(1/27)}+ O \left(\exp\left(-
\frac{2\pi}{\sqrt 3 u}\right)
\right).
$$
\end{prop}

\begin{proof}
We have
\beq\label{EIn-expr}
\mathbb E_i(I_n) = \sum_{(M,T)}i(M,T)  \frac
{(u+1)^{i(M,T)}}{f_n(u)} 
= (u+1) \frac{f'_n(u)}{f_n(u)} 
=  (u+1) \frac{[z^{n-1}] F''_{zu}(z,u)}{[z^{n-1}] F'_{z}(z,u)}.
\eeq
Comparing with~\eqref{ECn}, we see that for $u>0$, we have
$
\E_i(I_n)=(1+\bu) \E_c(C_n)
$. Thus~\eqref{EIn} follows from Proposition~\ref{random-forest} when $u>0$. 
The expansion of $\kappa_u$ near $0^+$ follows from the
estimate~\eqref{tau-0} of $\tau$ and the expansion~\eqref{sing-phi} of~$\Phi$.

Let us now take $u \in [-1,0)$. The series $F''_{zu}$ is still given
  by~\eqref{Fzu}, which can also be written $\Phi(R) F''_{zz}$
  (by~\eqref{Fzz}), or $\bu(R-z)F''_{zz}$. In view of the estimates~\eqref{Rsing-neg}
  and~\eqref{F-sing-neg} of $R$ and $F''_{zz}$, we find
$$
[z^{n-1}]F''_{zu}(z,u) \sim \bu(1/27-\rho) [z^{n-1}] F''_{zz}(z,u).
$$
Returning to~\eqref{EIn-expr} gives~\eqref{EIn} by singularity analysis, since
$\rho=1/27-u\Phi(1/27)$.

 When $u=0$, we have $R=z$. Hence~\eqref{Fzu} reads
$F''_{zu}(z,0)=\Phi(z) \theta'(z)$, while $F'_z(z,0)= \theta(z)$. As
 above, \eqref{EIn} follows from~\eqref{EIn-expr} by singularity
 analysis, 
using~\eqref{sing-phi},~\eqref{sing-theta} and~\eqref{theta-prime}. 
\end{proof}

\section{Asymptotics for cubic forested  maps}
\label{sec:asympt-3}

We study in this section the singular behaviour of the series $F(z,u)$
that counts cubic forested maps by the number of components, and the
asymptotic behaviour of its $n$th coefficient $f_n(u)$. As
expected, we observe a ``universality'' phenomenon:  our results are
qualitatively the same as for 4-valent maps (Theorem~\ref{tetra}).
However, the cubic case is more difficult since we now have to
deal with a pair of equations:
$$
R=z+u \Phi_1(R,S), \quad S=u\Phi_2(R,S),
$$
where $\Phi_1$ and $\Phi_2$ are given by~\eqref{phi1}
and~\eqref{phi2}. Our results are  less complete than in the
4-valent case: when $u<0$, we only determine the singular
behaviour of $F'(z,u)$ as $z $ approaches the radius of $F'$ on the real
axis. We do not know if $F'$ has  dominant singularities other than
its radius. Consequently, we have not obtained the asymptotic
behaviour of $f_n(u)$  when $u<0$.

\begin{theo}\label{cubique} 
Let $p=3$, and  take $u\ge -1$. The radius
  of convergence of $F(z,u)$  reads
 $$
    \rho_u=\tau -u \Phi_1(\tau, \sigma)
$$
where the pair $(\tau, \sigma)$ satisfies
$$
\sigma= u \Phi_2(\tau, \sigma)
$$ 
and
$$\left\{
\begin{array}{lll}
 \displaystyle64  \tau =  (1-4\sigma)^2 & \hbox{if } u\le 0,
\\
\\
\displaystyle \left(1-u\Phi_1^{x}(\tau, \sigma)\right)
\left(1-u\Phi_2^{y}(\tau, \sigma)\right)
= u^2 \Phi_1^{y}(\tau,\sigma)\Phi_2^{x}(\tau, \sigma) & \hbox{if } u>0.
\end{array}\right.
$$
 The series $\Phi_1$ and $\Phi_2$
are given by~\eqref{phi1} and~\eqref{phi2}, and $\Phi_i^{x}$
(resp. $\Phi_i^{y}$) denotes the derivative of $\Phi_i$ with respect
to its first (resp. second) variable. 
\\
In particular, $\rho_u$ is an algebraic function of $u$ on $[-1,0]$:
\beq\label{rho-3}
\rho_u= \frac{3(1-u^2)^2\pi^4+96u^2\pi^2(1-u^2)+512u^4+ 16u\sqrt2 \left( \pi^2(1-u^2)+8u^2\right)^{3/2}}{192\pi^4(1+u)^3}.
\eeq
  Let $f_n(u)$  be the coefficient in $z^n$ in $F(z,u)$. There exists 
   a positive constant $c_u$ such that
$$
f_n(u) \sim \left\{\begin{array}{lll}
\displaystyle c_u\, {\rho_u^{-n}}{n^{-3} }  & \hbox{ if }  u = 0,\\
\displaystyle  c_u \,{\rho_u^{-n}}{n^{- 5 /2}}  & \hbox{ if }  u > 0.
\end{array}\right.
$$
For $u \in [-1,0]$, the series $F'(z)\equiv F'(z,u)$  has the
following singular expansion  as $z\rightarrow \rho_u^-$:
\beq\label{exp-3}
F'(z) = F'(\rho_u) + \alpha (\rho_u-z) + 
\beta\,
\frac{  \rho_u-z}{\ln(\rho_u-z)}\left(1+o(1)\right),
\eeq
where 
$$
\beta =\frac{4u-3\sqrt2 \sqrt{\pi^2(1-u^2)+8u^2}}{2u^2} <0.
$$
\end{theo}

\noindent{\bf Remarks} \\
1. As in the 4-valent case, the singular behaviour of $F'$ obtained when
$u<0$ is incompatible with D-finiteness~\cite[p.~520
  and~582]{flajolet-sedgewick}.

\begin{cor} For $u\in [-1,0)$, the \gf\ $F(z,u)$ of cubic forested
    maps is not D-finite. The same holds when $u$ is an indeterminate.
\end{cor}

\noindent 2. The series   $F(z,0)$ has a simple   explicit
expression given by~\eqref{Fz0}:
$$
F(z,0)=3 \sum_{\ell \ge 1} \frac{(4\ell)!}{(2\ell-1)! (\ell+1)!
  (\ell+2)! }z^{\ell+2}.
$$
The above theorem  follows in this case from Stirling's formula. One has
$\sigma=0$ and $\rho_0=\tau=1/64$.  We will thus focus below on the cases $u>0$
and $u<0$.

\noindent 3. At $u=-1$, one finds $\rho_{-1}= \pi^2/384$, a beautiful
transcendental radius of convergence for the series counting cubic
maps equipped with an internally inactive spanning tree.

\subsection{The series $\Phi_1$, $\Phi_2$, $\Psi_1$ and $\Psi_2$}
\label{sec:pp}

We have performed  in Section~\ref{sec:cubic-de} a useful reduction by showing
that the bivariate series  $\Phi_1(x,y)$ and $\Phi_2(x,y)$ can  be
expressed in terms of the  univariate hypergeometric series $\Psi_1$ and
$\Psi_2$ (see~(\ref{pp1}--\ref{pp2})).
 The $i$th
coefficient of $\Psi_1$  is asymptotic to  $64^i/i^2$, up to a
multiplicative constant, and the same holds for $\Psi_2$. Hence both
series have  radius of convergence $ 1/ {64}$, converge at this point,
but their derivatives diverge. In fact,
$$
\Psi_1(z)/z=  \, _2F_1(1/4, 3/4;2;64z),
$$
so that $\Psi_1$ can be analytically defined on $\C \setminus[1/64, +\infty)$. The
  same holds for $\Psi_2(z)$ in view of~\eqref{ed-cubic}. It follows
  from~\cite[Eq.~(15.3.11)]{AS} that, as $\vareps\rightarrow 0$ in
  $\C\setminus \R^-$,
\beq\label{exp-psi1}
\Psi_1\left(\frac{1} {64}-\varepsilon\right) = 
{\frac {\sqrt {2}}{24\, \pi }}
+  {\frac {\sqrt {2}  }{2\, \pi }}\, \varepsilon \,\ln   \varepsilon
-  {\frac {\sqrt {2}  }{2\, \pi }} \varepsilon
+ O \left( {\varepsilon}^{2} \ln \varepsilon \right).
\eeq
By~\eqref{ed-cubic}, we also have
\beq
\label{exp-psi2}
\Psi_2\left(\frac{1} {64}-\varepsilon\right) = 
\frac 1 2 -{\frac {\sqrt {2}}{\pi }}
+ {\frac {4\sqrt {2} }{ \pi }}\, \varepsilon\,\ln 
  \varepsilon 
+ {\frac {12\,\sqrt {2} \, }{\pi }}  \varepsilon
+ O \left( {\varepsilon}^{2} \ln \varepsilon \right).
\eeq
Let us now return to $\Phi_1$ and $\Phi_2$.  The series $\sqrt{1-4y}$ has radius
$1/4$, the series $\Psi_1$ and $\Psi_2$ have radius $1/64$, and thus
$\Phi_1(x,y)$ and $\Phi_2(x,y)$ converge absolutely for $|y|<1/4 
$ and $64|x|< (1-4|y|)^2$ (Figure~\ref{fig:domaine}, left).  The
expressions~\eqref{pp1} and~\eqref{pp2} show that $\Phi_1$ and
$\Phi_2$ have an analytic continuation for $y \in \C\setminus [1/4,
  +\infty)$ and $x/(1-4y)^2 \in \C\setminus [1/64, +\infty)$
    (Figure~\ref{fig:domaine}, right). As $\Psi_1'(t)$ and
    $\Psi_2'(t)$ tend to $+\infty$ when $t\rightarrow 1/64$, there is
    no way to extended analytically $\Phi_1$ or $\Phi_2$ at a point of
    the \emm critical parabola, $\{64x =(1-4y)^2\}$. 

\begin{figure}[h!]
\begin{center}
 \begin{tabular}{c}\includegraphics[scale=0.3]{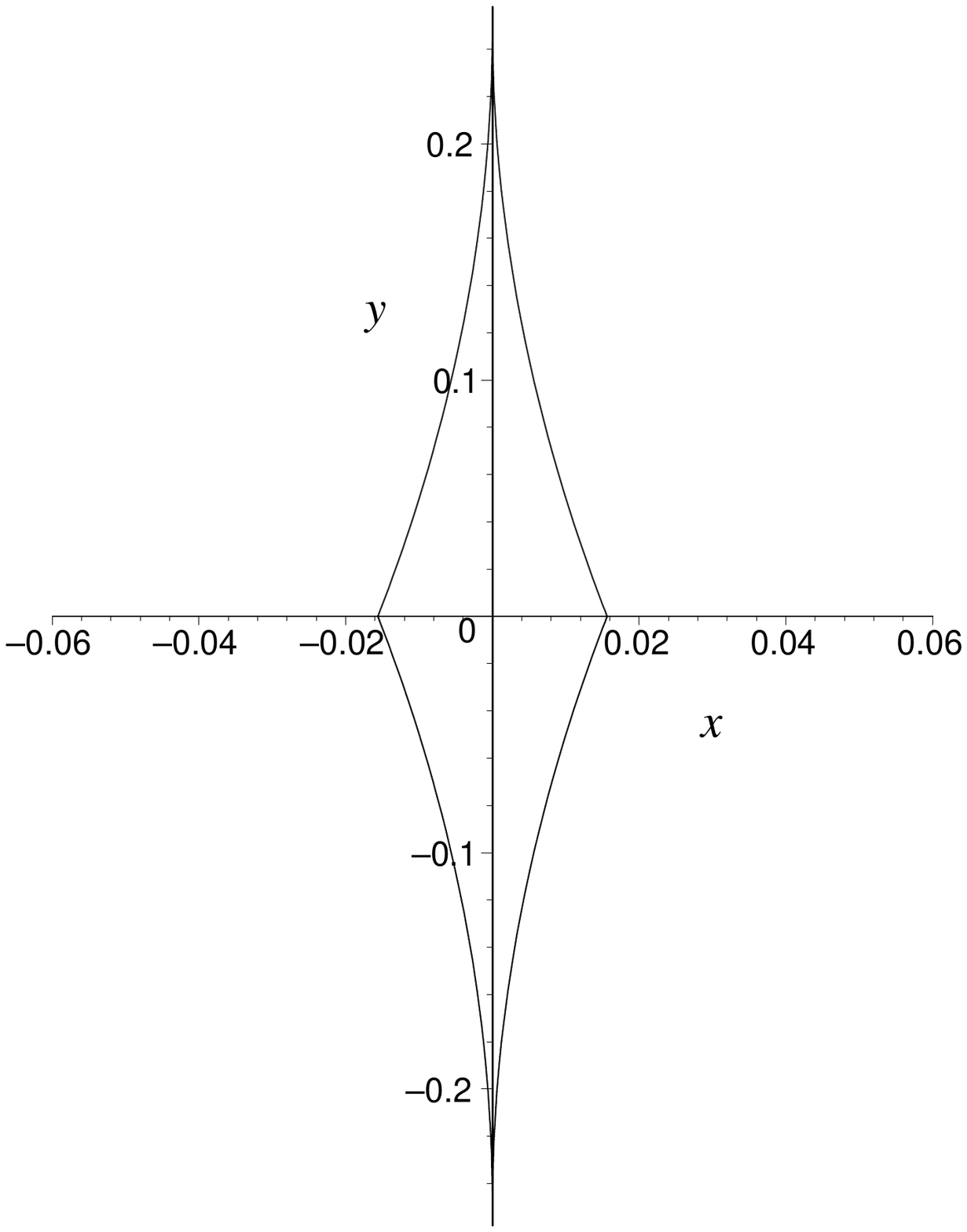}\end{tabular}
\hskip 5mm
 \begin{tabular}{c}\includegraphics[scale=0.3]{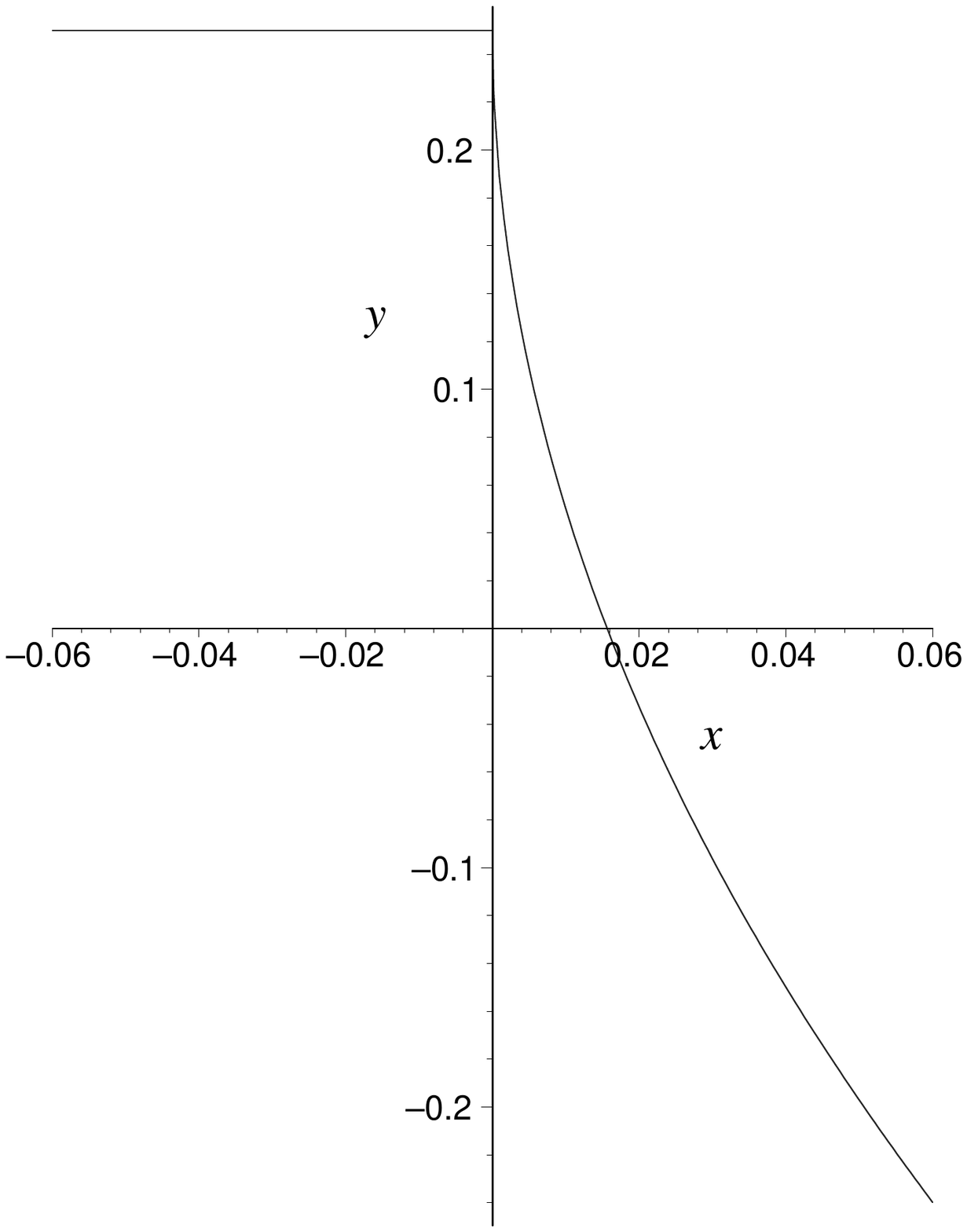}\end{tabular}
\end{center}
\vskip -10mm
\caption{\emph{Left:} The  domain of absolute convergence of the series $\Phi_1$ and
  $\Phi_2$, in the real plane. \emph{Right:} A domain where an
  analytic continuation exists. No analytic continuation exists at
  a point of the parabola.}
\label{fig:domaine}
\end{figure}

\subsection{When  $u>0$}
\begin{prop}
  Assume $u>0$. The series $R$, $S$ and $F'$ have the same radius of
  convergence, denoted $\rho_u$, which satisfies the conditions stated in
  Theorem~{\rm\ref{cubique}}.  The three series are analytic in a
  $\Delta$-domain of radius $\rho_u$, with a square root singularity at
  $\rho_u$. In particular,
$$
f_n(u) \sim c_u \rho_u^{-n} n^{-5/2}
$$
for some positive constant $c_u$.
\end{prop}
\begin{proof}

Recall that these three series are defined by the
system~\eqref{F-cubic},~\eqref{phi1enpsi},~\eqref{phi2enpsi}.  The analysis of systems of
functional equations can be   a tricky 
exercise, even in the \emm positive case,\footnote{By this, we mean a system
  given by equations of the form $R_i= F_i(R_1, \ldots, R_m)$ where
  the series $F_i$ have non-negative coefficients.} and with 2
equations  only. In particular,  the connection 
between the location of the radius and 
the solution(s) of the so-called \emm characteristic system, is subtle
(see~\cite{drmota-systems,burris}). In our case 
however,  the equation that defines $S$ does not
involve the variable $z$ explicitely, and this allows us to
proceed safely in two steps. As in Section~\ref{sec:enriched}, we
first define $\tilde S\equiv \tilde S(z,u)$ as the unique
power series in $z$  satisfying $\tilde S (0,u)=0$ and 
\begin{eqnarray}
  \tilde S &=& u\,\Phi_2(z,\tilde S) \label{S-t-phi}\\
& =& u \sqrt{1-4\tilde S}\ \Psi_2\left(
\frac z {(1-4\tilde S)^2}\right)+ \frac u 4 \left( 1-\sqrt{1-4\tilde
  S}\right)^2.
\label{S-t}
\end{eqnarray}
We will first study $\tilde S$, and then  move to $R$, which is now defined
by the following equation:
\begin{eqnarray}
R&=&z+u \Phi_1(R, \tilde S(R))
\label{R-t}\\
&=&z+u (1-4\tilde S(R))^{3/2} \Psi_1 \left( \frac{R}{(1-4\tilde
  S(R))^2}\right) -uR,
\label{R-t-Psi}
\end{eqnarray}
where we have denoted for short $\tilde S(z)=\tilde S(z,u)$. Of
course, $S=\tilde S(R)$. 

 So let us begin with $\tilde S$. One can prove  that~\eqref{S-t-phi} fits in the
smooth implicit function schema of~\cite[Def.~VII.4]{flajolet-sedgewick}, but we can actually content
ourselves with an application of Proposition~\ref{lemmeasympt2}, where $\tilde S$
plays the role of $Y$. The series
$H(x,y)= y-u\Phi_2(x,y)$ satisfies the assumptions of this
proposition. Define $\rt$ as in the proposition. Since $\tilde S$ has non-negative  
coefficients, the points $(z, \tilde S(z)) $ form, as $z$ goes from $0$ to
$\rt$,  an increasing curve starting from $(0,0)$ in the plane $\R^ 2$. Condition
(b), together with the properties of $\Phi_2$ described in Section~\ref{sec:pp},
implies that this curve cannot go beyond  the parabola $64x=(1-4y) ^
2$.
This rules out the possibilities (i) and (iv). Now $H'_ 
y(x,y)=1-u\Phi'_2(x,y)$ approaches $-\infty$ as $(x,y)$ approach the
parabola, and thus Condition (d) rules out the possibility (iii). The
curve $(z, \tilde S(z))$ thus ends (at $z=\rt$) before reaching the
parabola.  Moreover (ii) holds: $H'_y(\rt, \tilde S(\rt))=0$, or
equivalently, 
\beq\label{char-Stilde-positif}
1= u\pd{ \Phi_2} y (\rt,\tilde S(\rt)).
\eeq
(The $\liminf$ of (ii)
becomes here a true
limit because of the positivity of the coefficients of $\Phi_2$ and
$\tilde S$.) By (a), the radius of $\tilde S$ is at least $\rt$.
Finally, it follows from~\eqref{S-t-phi} that for $z \in [0, \rt)$,
 \beq\label{Sprime}
\tilde S'(z) = u \, \frac {\pd {\Phi_2} x (z,\tilde S(z))} {1- u\, \pd
  {\Phi_2} y (z,\tilde S(z))}.
\eeq
By~\eqref{char-Stilde-positif}, this  derivative tends to $ + \infty$
as $z\rightarrow \rt$. Hence $\tilde S$ has radius $\tilde
\rho$. Figure~\ref{fig:Stilde-pos} (left) illustrates 
the behaviour of $\tilde S$ on $[0, \tilde \rho]$.

\begin{figure}[h!]
\begin{center}
\begin{tabular}{c}\includegraphics[scale=0.3]{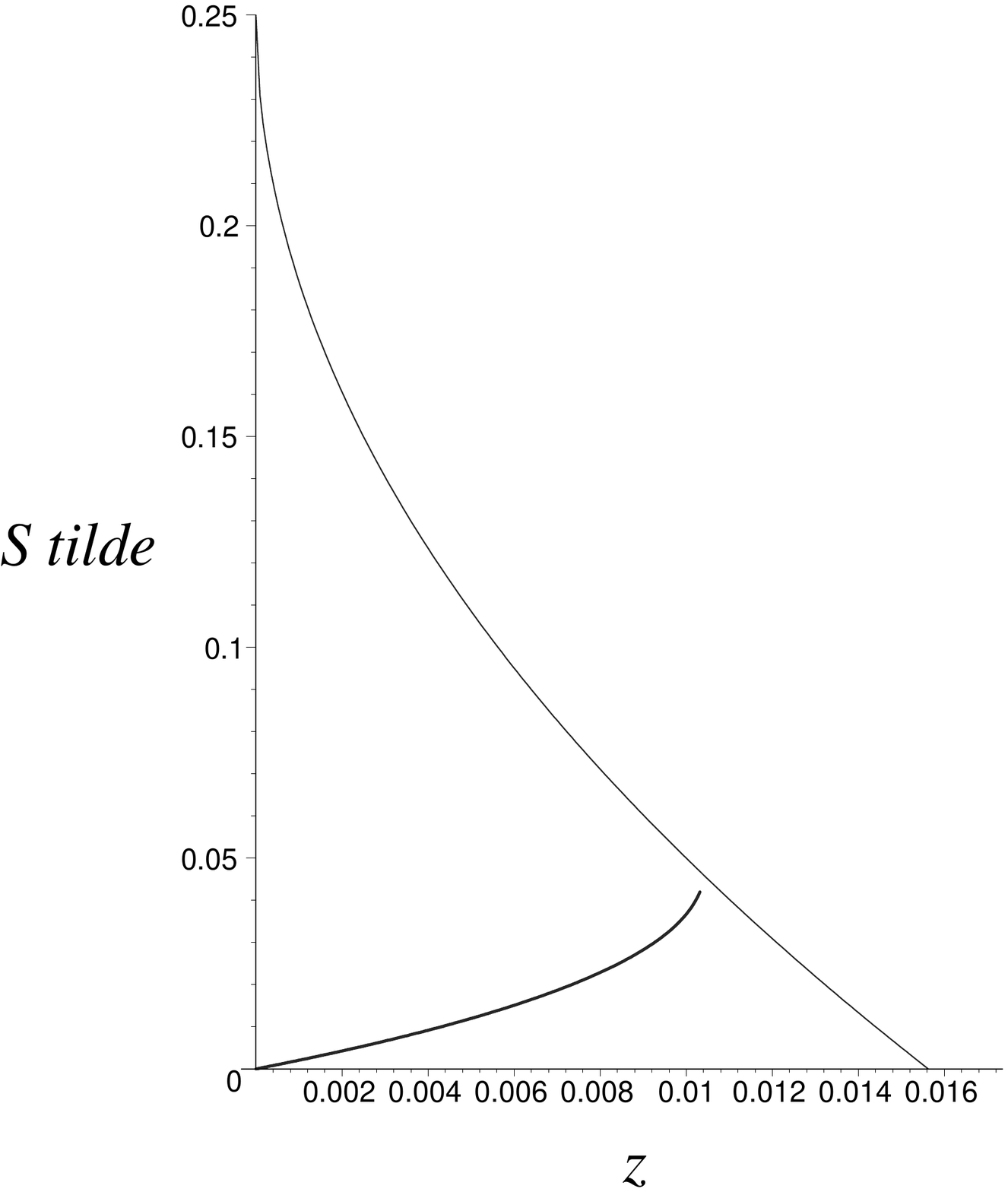}\end{tabular}
 \begin{tabular}{c}\includegraphics[scale=0.3]{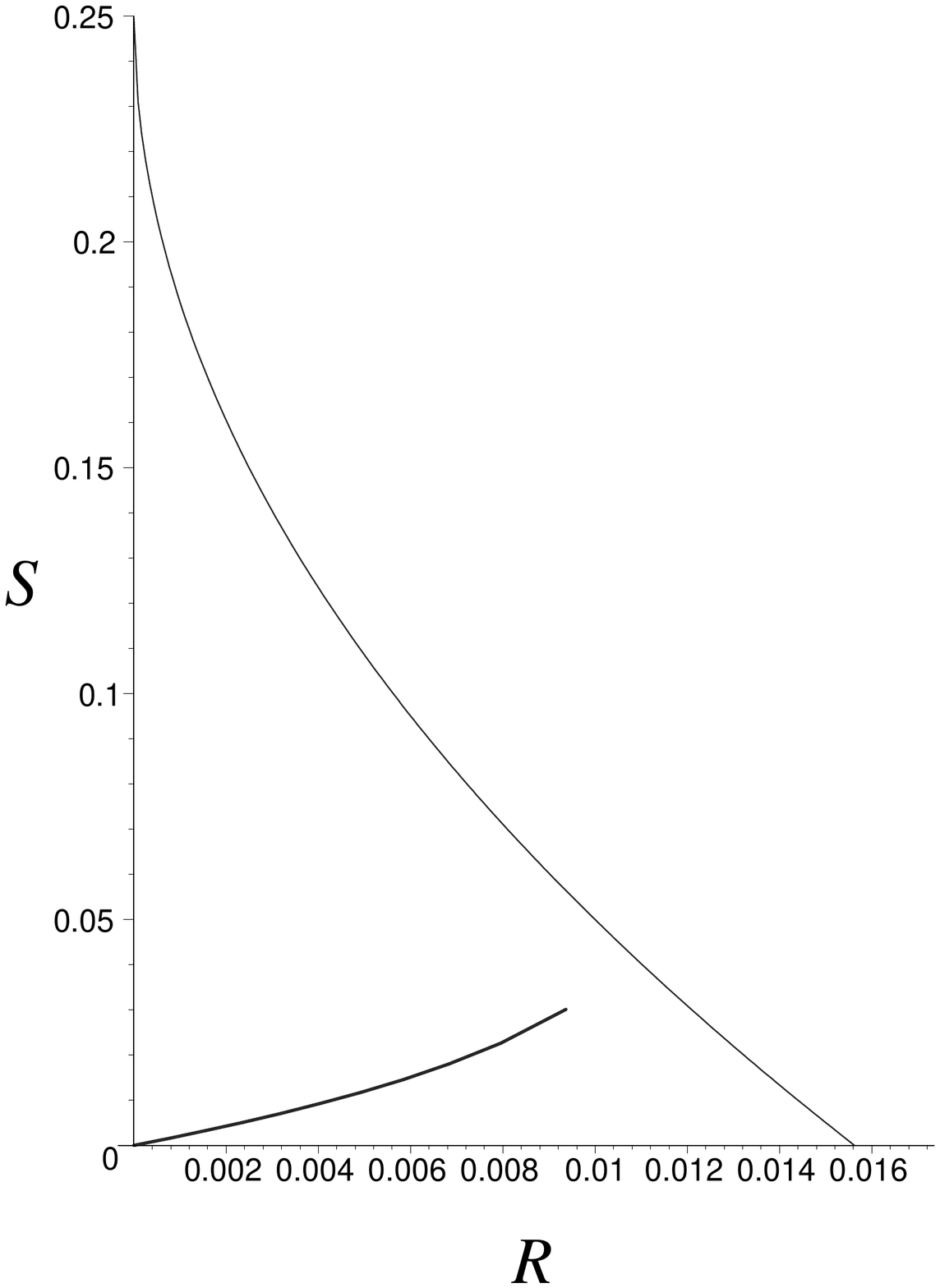}\end{tabular}
\end{center}
\vskip -10mm
\caption{\emph{Left:} Plot of $\tilde S(z)$ for $u=1$ and $z\in[0,
    \rt]$. The points $(z, \tilde S(z))$  remain below  the parabola
    $64z=(1-4\tilde S)^2$.  
The plot was obtained using
  the expansion of $\tilde S(z)$ at order 80 (this is why the divergence of $\tilde
  S'$ at $\rt$ is not very clear on the picture), and the
  estimate $\rt\simeq 0.01032$. \emph{Right:} Plot of $(R(z), S(z))$ for $u=1$
and $z\in[0,   \rho]$, with $\rho\simeq 0.0098$. This curve follows the plot of $\tilde S$, but
stops at the point $(R(\rho),S(\rho))$, for which
$R(\rho)<\rt$.}
\label{fig:Stilde-pos}
\end{figure}

\medskip
Let us now consider the equation~\eqref{R-t} that defines $R$, and prove that it fits in
the smooth implicit function schema of~\cite[Def.~VII.4,
  p.~467-468]{flajolet-sedgewick}. With the notation of this
definition,  $G(z,w)= z+u
\Phi_1(w, \tilde S(w))$. The properties of $\tilde S$ established
above show that $G$ is analytic in $\C\times \{w: |w|<\rt\}$. 
 The characteristic equation $1=G'_w(\rho,\tau)$ does not involve $\rho$ and reads
 \beq\label{cubic-char}
 1= u\, \left(\pd {\Phi_1} x  (\tau ,\tilde S(\tau )) + \,\tilde S'(\tau )\, \pd
 {\Phi_1} y  (\tau ,\tilde S(\tau ))\right).
\eeq
The right-hand side of this equation increases from 0 to $+\infty$ as
$\tau$ goes from $0$ to $r$ (because, as observed above, $\tilde
S'(\rt)=+\infty$). Hence~\eqref{cubic-char} determines a unique
value of $\tau$ in $(0, \rt)$. The equation 
$\tau = G(\rho, \tau)$ gives the value of $\rho$:
\beq\label{rho-val}
\rho=\tau -u \Phi_1(\tau,\tilde S(\tau)).
\eeq
Let $\sigma=\tilde S(\tau)$. The combination
of~\eqref{rho-val}, \eqref{S-t-phi}, \eqref{cubic-char}
and~\eqref{Sprime} proves the properties of $\rho, \tau$ and $\sigma$ stated in
Theorem~\ref{cubique}.

 The rest of the argument is analogous to the end of the proof of
Proposition~\ref{prop:4v-pos}. First, $R$ is irreducible as shown by the first
terms of its expansion at $0$:
$$
R=z+2u(2u+3)z^2+4u(42u^2+63u+10u^3+35)z^3+ O(z^4).
$$
 By Theorem~VII.3 of~\cite[p.~468]{flajolet-sedgewick}, it has
radius $\rho$, and is analytic in a $\Delta$-domain of radius $\rho$. It
takes the value $\tau$ at $\rho$, with  a square root
singularity there. By composition with the series $\tilde S$,
which has  radius $\rt >\tau$, the same properties hold for
$S=\tilde S(R)$, and finally for the series $F'$ given by~\eqref{frs}
(since $\theta(x,y)$ is analytic  in $\R^2$ for $64x<(1-4y)^2$). 

The behaviour of $R$ and $S$ is illustrated in
Figure~\ref{fig:Stilde-pos} (right). 
\end{proof}

\subsection{When $u<0$}

\begin{prop}
  Let $u\in [-1,0)$. The series $R$, $S$ and $F'$ have radius
    $\rho\equiv \rho_u$ given by~\eqref{rho-3}. As $z\rightarrow
    \rho_u^-$, these three series admit an expansion of the
    form~\eqref{exp-3}, with $\beta>0$.  
\end{prop}
\begin{proof}  
As a preliminary remark, recall that $F(z,u)$ is $(u+1)$-positive,
with several combinatorial interpretations described in Section~\ref{sec:tutte}. By
Pringsheim's theorem,  the radius of $F$ is also its smallest
real positive singularity. By Corollary~\ref{positiv}, the same holds
for $R$, $S$ and $\tilde S$. This will be used frequently in the
proof, without further reference to Pringsheim's theorem.

As in the case $u>0$, we proceed in two steps, and study first the
series $\tilde S$ defined by~\eqref{S-t-phi}, and then the series $R$ defined
by~\eqref{R-t}.
Let us begin with $\tilde S$, and apply
Proposition~\ref{lemmeasympt2} with $H(x,y)=y-u\Phi_2(x,y)$. Let us
rule out the possibilities (i), (ii) and (iv).

\noindent
\textbf{(i)} Could $\tilde S\equiv \tilde S(z,u)$ have an analytic
continuation on $(0, +\infty)$? That is,  an
infinite radius of convergence?  Corollary~\ref{positiv}  implies that the radius of  $
\tilde S$ is at most the radius of $\tilde S(z,-1)$, which counts (by
the number of leaves) 
enriched \~S-trees with no flippable edge. Since these trees can have
arbitrary large size (Figure~\ref{fig:peigne}), $
\tilde S(z,-1)$ is not a polynomial. Its coefficients are non-negative
integers, and hence its radius is  at most 1. The same thus holds for $\tilde S(z,u)$.

\begin{figure}[h!]
\begin{center}
\includegraphics[scale=1.2]{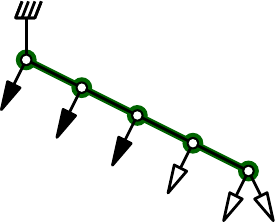}
\end{center}
\caption{A cubic enriched \~S-tree with no flippable edge.}
\label{fig:peigne}
\end{figure}

\noindent
\textbf{(ii)} By Lemma~\ref{positiv-bis}, the series  $\pd {\Phi_2} y (z,\tilde
S(z))$ has non-negative coefficients.  Since its constant term is $0$,
the function $1-u \pd {\Phi_2} y (z,\tilde S(z))$ is increasing on
$[0, \rt)$, with initial value 1: this rules out~(ii).

\noindent
\textbf{(iv)} By Corollary~\ref{positiv}, $\tilde S$ is negative and decreases on 
$[0, \rt)$. Assume that it tends to $ - \infty$. Since $\rt$ is
finite, this implies that
$$ 
\lim_{z\rightarrow \rt^-}\Psi_2 \left( \frac{z} {(1 -4\, \tilde S(z))^2} \right) =
\Psi_2(0)=0.
$$
But then~\eqref{S-t} gives
$$ 
(1+\bu)\,{\tilde S} = -u \sqrt{1-4 \tilde S}/2 + o\left(\sqrt{1-4
  \tilde S} \right)
,$$
which is impossible if $\tilde S \rightarrow -\infty$. 

We conclude that (iii) holds, so that $\Phi_2$ has no analytic continuation at $(\rt,
\tilde S(\rt))$. Given the properties of $\Phi_2$ described in Section~\ref{sec:pp},
this means that 
$$
64 \rt=(1-4\tilde S(\rt))^2.
$$
The radius of $\tilde S$ is at least $\rt$, the value of which we will
determine explicitely later. 
 Figure~\ref{fig:Stilde-neg} shows a plot of $\tilde S$ for
 $u=-1/2$. One can in fact prove that $\rt$ \emm is, the radius of $\tilde
 S$, but we will not use that.

\begin{figure}[h!]
\begin{center}
\includegraphics[scale=0.3]{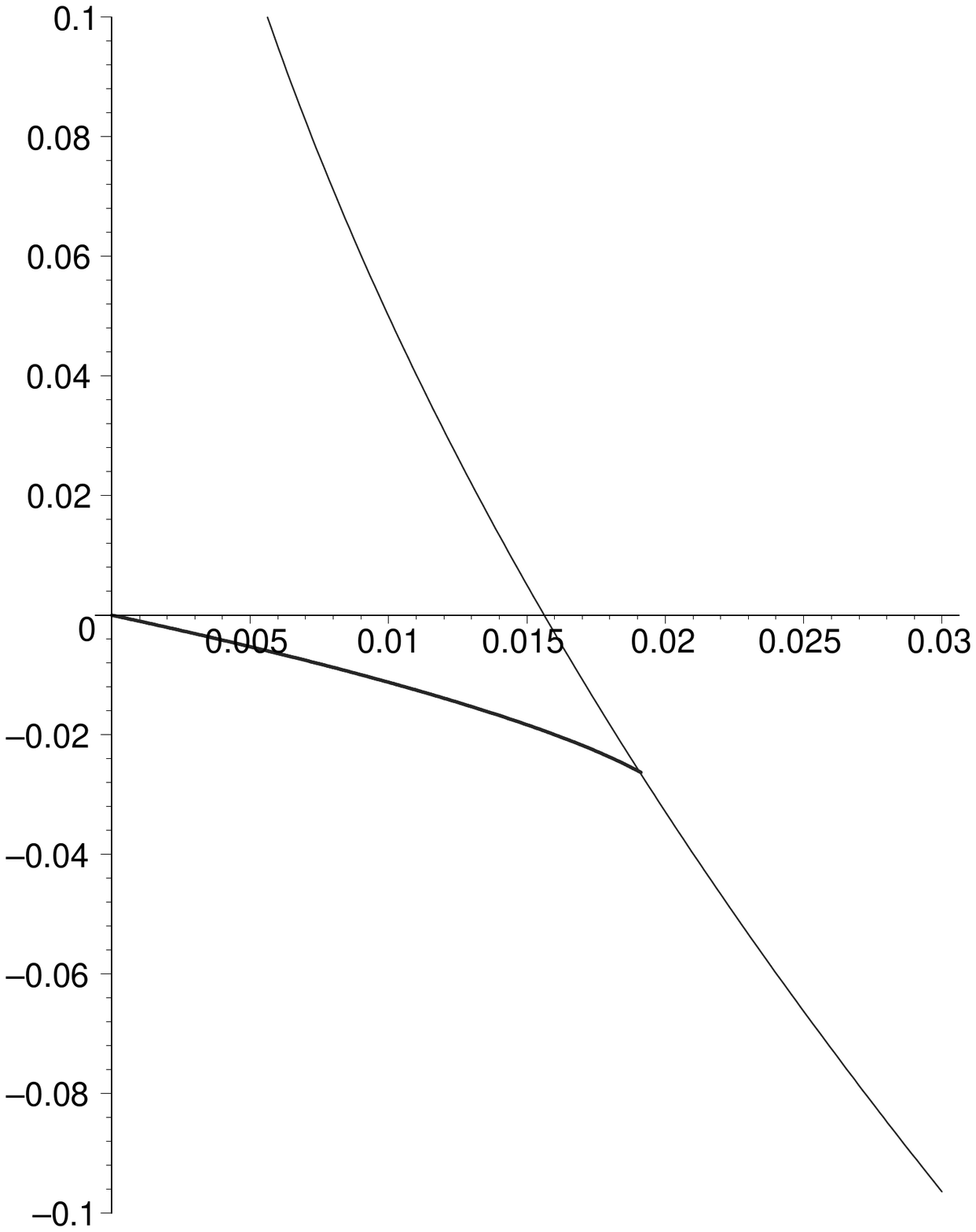}
\end{center}
\vskip -10mm
\caption{A plot of $\tilde S(z)$ for $u=-1/2$ and $z\in[0,
    \rt]$. The curve reaches the parabola
    $64z=(1-4\tilde S)^2$ at $\rt$. The plot was obtained using
  the expansion of $\tilde S(z)$ up to order 25. Plotting the pairs
  $(R(z), S(z))$ for $z\in [0, \rho)$ gives the same curve.}
\label{fig:Stilde-neg}
\end{figure}

\medskip
 Let us now consider the equation~\eqref{R-t} that defines $R$, and apply Corollary~\ref{lemmeasympt} with
 $\Omega(y)=y-u\Phi_1(y, \tilde S(y))$. We have just seen that,
as $y$ goes from $0$ to $\rt$,  the pair $(y, \tilde S(y))$ reaches for the
 first time  the critical parabola at $\rt$. Hence, with the
 notation of Corollary~\ref{lemmeasympt}, the first
 singularity of $\Omega$ on the positive real axis satisfies 
 $\omega\ge \rt$. Let us define $\tau$ and $\rho$ as in Corollary~\ref{lemmeasympt}.

Could it be that $\Omega'(\tau)=0$? By Corollary~\ref{lemmeasympt}, $R(z)$ increases on $(0,
\rho)$ and  $\Omega'(R(z))= 1/R'(z)\ge 0$. So could it be that
$R'(z)$ tends to $+\infty$ as $z$ tends to $\rho$? No: by
Corollary~\ref{positiv}, $\bu(R'(z)-1)$  has non-negative
coefficients, and thus is always larger that its value at $z=0$, which
is $0$. Since $u<0$, this gives $R'(z)\le 1$ on $(0, \rho)$, and we
conclude that $\Omega'(\tau)>0$. Hence $\tau=\omega\ge \rt$.

Since $R(z)$ increases from $0$ to $\omega$ on $[0, \rho]$, there
  exists a unique $\rh$ such that  $R(\rh)=\rt$. 
 Since $\tilde S$ has radius at least $\rt$, the series $S =
  \tilde S( R)$ has also radius at   least $\rh$.
The plot of the  pairs $(R(z), S(z))$ for $z\in [0, \rh]$ coincides
with the plot of    $(z, \tilde S(z))$ for $z\in [0, \rt]$  shown in
Figure~\ref{fig:Stilde-neg}.  

We will now use the  system~(\ref{phi1enpsi}--\ref{phi2enpsi})
defining $R$ and $S$ to   obtain expansions of 
  $R$ and $S$ near $\rh$. These expansions will be found to be singular
  at $\rh$: this implies that $\rh=\rho$ is the radius of $R$ and $S$. 

We adopt the following notation: $z=\rh-x$,  $R(z)=\rt-r$,  $S(z)=S(\rh)-s$  {and}
\beq\label{TRS}
\frac{R(z)}{(1-4S(z))^2}=\frac 1{64}-\vareps.
\eeq
The quantities $x$, $r$, $s$ and $ \vareps$ tend to $0$ as $z$ tends to
$\rh$. 
Let us begin by expanding~\eqref{phi2enpsi} for $z$ close to
$\rh$. Using the expansion~\eqref{exp-psi2} of 
$\Psi_2$ near $1/64$, we obtain
\beq\label{exp1}
a_1+b_1s+c_1\vareps \ln \vareps + d_1\vareps = O(\vareps^2 \ln \vareps ) +
O(s^2)+ O(s\,\vareps \ln \vareps),
\eeq
with
$$
a_1=\frac{1+u}4 \Dc^2 - \frac{u\sqrt 2} \pi \Dc+ \frac{u-1}4,
$$
$$
b_1= - \frac{2u \sqrt 2}{\pi \delta}+1+u, \quad 
c_1= \frac{4\sqrt 2}\pi u \Dc , \quad 
d_1=3c_1,
$$
and $\delta=\sqrt{1-4\tilde S(\rt)}$.
In particular, $a_1$ must vanish, which gives the value of $\Dc$:
\beq\label{delta-c}
\Dc= \sqrt{1-4\tilde  S(\rt)}= \frac{2\sqrt 2 u +
  \sqrt{\pi^2(1-u^2)+8u^2}}{\pi(1+u)}. 
\eeq
(The choice of a minus sign before a square root would give a negative
value, which is impossible for $\Dc=\sqrt{1-4\tilde S(\rt)}$.)
Note that $u$ has  a rational expression in terms of $\Dc$: 
$$
u= -  \frac{\pi(\delta^2-1)}{\pi \delta^2-4 \sqrt 2 \Dc+\pi}.
$$
We will  replace all occurrences of $u$ by this 
expression, to avoid  handling  algebraic coefficients.


Let us now return to the expansion~\eqref{exp1}. Given that
$\Dc>0$, we have $b_1>0$ for $u \in[-1,0)$. Hence $s=O(\vareps \ln  \vareps)$,
and~\eqref{exp1} can be rewritten as
\beq\label{exp2}
b_1s+c_1\vareps \ln \vareps + d_1\vareps = O(\vareps^2 \ln^2 \vareps ).
\eeq
Let us now move to~\eqref{TRS}. Using $\rt= \Dc^4/64$, it gives
\beq\label{exp3}
 b_2 s + d_2\vareps + e_2 r = O( \vareps^2 \ln^2 \vareps )
\eeq
with $ b_2=-8 \Dc^2, d_2= 64 \Dc^4 , e_2=-64$.
Finally, the equation~\eqref{phi2enpsi} that defines $R$ gives
\beq\label{exp4}
a_3+b_3s + c_3 \vareps \ln \vareps + d_3  \vareps + e_3 r +f_3 x
=O(\vareps^2 \ln ^2 \vareps ),
\eeq
where, in particular,
$$ 
a_3=96(\pi \Dc^2 -4\sqrt 2 \Dc+\pi)\rh+\Dc^3(2\sqrt2\Dc^2-3\pi\Dc
+4\sqrt 2).
$$
Since $a_3$ must vanish, we obtain  a rational expression of $\rh$ in
terms of $\Dc$, and then, using~\eqref{delta-c}, an explicit expression which
coincides with~\eqref{rho-3}. We do not give here the
expressions of $b_3, c_3, d_3, e_3$ and $f_3$, which are rational
in $\Dc$. They are easy to compute. Let us just mention that
$f_3\not =0$. 

Now, using~\eqref{exp2},~\eqref{exp3} and~\eqref{exp4} in this order, we
obtain for $s$, $r$ and finally $x$ expansions in $\vareps$ of the form
\begin{eqnarray}
s&= & c_4\,\vareps  \ln \vareps + d_4\,\vareps + O(\vareps^2 \ln ^2
\vareps ),
\label{s-eps}
\\
r&=&  c_5\,\vareps  \ln \vareps + d_5\,\vareps + O(\vareps^2 \ln ^2
\vareps ),
 \label{r-eps}\\
x&=&  c_6\,\vareps  \ln \vareps + d_6\,\vareps+  O(\vareps^2 \ln
^2\vareps ).
\label{x-eps}
\end{eqnarray}
In particular, $c_6\not = 0$ for $u \in [-1,0)$ and the latter equation gives
$
x\sim c_6 \,\vareps  \ln \vareps,
$
so that $\ln x \sim \ln \vareps$ and thus
\beq \label{eps-x}
\vareps = \frac x{c_6\ln x} \left(1+ o(1)\right).
\eeq
To conclude, we use~\eqref{x-eps} to express $\vareps  \ln \vareps $
as a linear combination of $x$ and $\vareps$ (plus $O()$ terms), and~ \eqref{eps-x}
 to express $\vareps $ in terms of $x$. This replaces~\eqref{s-eps}
 and~\eqref{r-eps} by 
\begin{eqnarray*}
s&= & \frac{c_4}{c_6} x + \frac{d_4c_6-c_4d_6}{c_6^2} \frac x{\ln x} (1+o(1)),
\\
r&=&  \frac{c_5}{c_6} x + \frac{d_5c_6-c_5d_6}{c_6^2} \frac x{\ln x}
(1+o(1)) .
\end{eqnarray*}
These equations, written explicitely, read
\begin{eqnarray*}
S(z)&=& \frac {1-\Dc^2}4+
\frac{4\pi}{\Dc\sqrt{\pi^2(1-u^2)+8u^2}}
(\rh -z)
- \frac{2\sqrt 2\pi}{u\Dc}\frac{
  \rh -z}{\ln     (\rh -z)}\left(1+o(1)\right),
\\
R(z)&=& \rt   - \frac{\pi \Dc}{2\sqrt{\pi^2(1-u^2)+8u^2}}(\rh-z) - 
\frac{\sqrt 2 \pi \Dc}{4u} \frac{
  \rh -z}{\ln(\rh-z)}\left(1+o(1)\right),
\end{eqnarray*}
as $z\rightarrow \rh$.
In particular, $R$ and $S$ are singular at $\rh$, so that $\rh=\rho$
is their common radius.

Using~\eqref{F-cubic}, we finally compute an expansion of $F'(z)$ near $\rho$,
which gives~\eqref{exp-3}. The coefficient $\beta$ of
$(\rho-z)/\ln(\rho-z)$ does not vanish on $[-1,0)$, 
  and $F'$ has radius $\rho$ as well.
\end{proof}

\section{Final comments}
\label{sec:final}

\subsection{Universality}
Our asymptotic results remain incomplete when $p=3$, as we have not
been able to obtain the asymptotic behaviour of $f_n(u)$ for negative
values of $u$ (but only the singular behaviour of $F'(z,u)$). We still
expect $f_n(u)$ to behave like $c_u \rho_u^{-n} n^{-3} (\ln n)^{-2}$,
as in the 4-valent case.

We have also examined general even values of $p$. As explained below
Theorem~\ref{thm:equations}, the series $S$ vanishes, so that we only deal with one
equation (in $R$). When $p=6$ for instance, it reads:
$$
R= z+u\Phi(R)= z+u \sum_{\ell \ge 1} \frac{(5\ell)!}{\ell! (4\ell+1)!}
R^{2\ell+1}.
$$
New difficulties arise from the periodicity of $\Phi$ and $R$, but we
still expect the same behaviour for the numbers $f_n(u)$, even though
$R$ and $F$ will have multiple singularities on their circle of
convergence.

We also plan the study of general (non-regular) forested maps.

\subsection{A differential equation involving $F$, rather than $F'$?}
The two differential equations (DEs) obtained for the series $F$ in
Section~\ref{sec:de}, for the 4-valent, and then for the cubic
case, are in fact equations  of order 2 satisfied by $F'$. It is 
natural to ask if $F$ itself satisfies a DE of order 2. Let us examine
in detail the case $p=4$.

Returning to Lemma~\ref{lem:cont}, we first need an expression
of $\bar M$. Since $t_{2i+1}=t^c_{2i+1}=0$ when $p=4$, we can content
ourselves with an expression of $\bar M$ valid  when $g_{2i+1}=0$ for all $i$. Such an
expression is easily obtained from the expression~\eqref{relM} of
$\bar M'_z$. Indeed,  $S=0$ in the even case, and the
equations~\eqref{relM} and~\eqref{relS}, written as
$$
\bar M'_z= \bar \theta(R), \quad R=z+u\bar \Phi(R),
$$
imply at once
$$
\bar M= \bar \Psi(R)
$$
where
\begin{eqnarray*}
\bar \Psi(x)&=& \int \bar \theta(x) \left( 1-u \bar \Phi'(x)\right) dx
\\
&=& \sum_{i\ge 1} h_{2i}  {2i \choose i} \frac{x^{i+1}}{i+1} -u  \sum_{i\ge 1, j \ge 0}
h_{2i} g_{2j+2} ( {2j+1} ){2i \choose i}{2j\choose j} \frac {x^{i+j+1}}{i+j+1}.
\end{eqnarray*}
This should be compared to Eq.~(1.4)in~\cite{bg-continuedfractions},
which reads, in the even case,
$$
 \bar M= \sum_{n\ge 1} h_{2n} {2n \choose n}\frac{ R^{n+1}}{n+1} -u
 \sum_{n\ge 1, q\ge 0, k>q}h_{2n} g_{2k} {2n+2q \choose n+q} {2k-2q-2 \choose k-q-1}\frac{ R^{n+k}}{n+q+1}.
$$
Our (simpler) expression is obtained by summing over $q$. 

Hence for $p=4$,  Lemma~\ref{lem:cont} gives
\beq\label{Fzu-expl}
F(z,u)= 4\sum_{i\ge 2}\frac{(3i-3)!}{(i-2)! i!^2} \frac{R^{i+1}}{i+1}
 -u  \sum_{i\ge 2, j \ge 1}\frac{(3i-3)!}{(i-2)!
   i!^2}\frac{(3j)!}{j!^3} \frac {R^{i+j+1}}{i+j+1}
=\Psi(R),
\eeq
where $\Psi(x)= \Psi_1(x)- u \Psi_2(x)$,
$$
\Psi_1(x) \int \theta(x), \quad   \quad \Psi_2(x) =\int \theta(x)  \Phi'(x) dx,
$$
and  now $R$ is defined by $R=z+u\Phi(R)$, where $\Phi$ is given
by~\eqref{theta-phi-4V}. 

Now assume that $F$ is differentially algebraic of order 2: there
exists a non-zero polynomial  $P$ such that 
$$
P( F, F', F'', z, u)=0.
$$
Equivalently,
$$
P(\Psi(R), \theta(R), R'\theta'(R),z,u)=0.
$$
Using $z=R-u\Phi(R)$, $R'=(1-u\Phi'(R))^{-1}$ and the
equations~\eqref{phi-second} and~\eqref{thetrat} that relate $\theta$,
$\Phi$, and their derivatives, we conclude that 
either $\Psi(x)$ is algebraic over $\Q(x, u, \Phi(x), \Phi'(x))$, or
$\Phi$ and $\Phi'$ are algebraically related over $\Q(x)$. Let us examine these two possibilities. 

\medskip
\noindent 
1. Can  $\Psi(x)$  be algebraic over $\Q(x, u, \Phi(x), \Phi'(x))$? 
Given that
$$
15 \Psi_1(x)= 15 \int \theta(x)dx=  54 x^2-2(1+81x) \Phi(x)
+8x(27x-1) \Phi'(x)
$$
and
$$
3\Psi_2(x)= 3\int \theta(x) \Phi'(x) dx= 12 x \Phi(x)
-2(1-27x)\Phi(x) \Phi'(x) -48\Phi^2(x) + 12 \int \frac{\Phi^2(x)}{x} dx
,
$$
this is equivalent to saying that $\int \Phi(x)^2/x dx$ is algebraic over
$\Q(x,\Phi(x), \Phi'(x))$. Or, using~\eqref{F-hg}, that the hypergeometric
function
$$
f(x)=\  _2F_1\left(\frac 1 3,\frac 2 3;2;x\right) 
$$
is such that $g(x):= \int xf^2(x)dx$ is algebraic over $\Q(x, f(x) ,
f'(x))$ (here, we use the fact that 
$$
\left. 20 \int xf(x) dx= 9 x^2 f(x) +9x^2(1-x) f'(x).\right)
$$
A related question is whether $g$ is a linear combination of $
f^2, ff', (f')^2$. Given that
$$
2f(x)+18(x-1)f'(x)+9(x-1)f''(x)=0,
$$
the vector space spanned over $\Q(x)$ by these 3 series
contains all products $f^{(i)} f^{(j)}$ and is closed by differentiation. This would imply that $g$
satisfies a linear DE of order 4 with coefficients in
$\Q(x)$. Starting from the order 4 DE satisfied by $g'$,
$$
-4g'(x)+8x(x-1)g''(x) +27x(x-1)^2g^{(3)}(x) +9x^2(x-1)^2g^{(4)}(x)=0,
$$
the Maple command {\tt ode\_int\_y} tells us that $g$ satisfies no
linear DE of order 4. Following discussions with Alin Bostan and Bruno
Salvy, this seems to imply actually that $g$ is not algebraic over
$\Q(x,f,f')$. 

\medskip
\noindent 
2. Now could it be that $F'$ satisfies a DE of order 1? This would imply
that $\Phi$ and $\Phi'$ are algebraically linked over $\Q(x)$, or,
equivalently, that $f$ and $f'$ are algebraically linked over
$\Q(x)$. One can prove that this is not the case, by  combining the
fact that $f'(x)$ diverges at 1 like $\ln(1-x)$, while $f(1)=\sqrt 3
/(12 \pi)$  is finite and transcendental.

These considerations lead us to believe that we have found in Section~\ref{sec:de4} the
equation of minimal order satisfied by $F$.

\medskip\noindent{\bf Acknowledgements.}
 We are grateful to Yvan Le Borgne and Andrea Sportiello for interesting
discussions on this problem, and to Alin Bostan and Bruno Salvy for
their help with hypergeometric series.


\bibliographystyle{plain}
\bibliography{coloured}

\begin{thebibliography}{10}

\bibitem{AS}
M.~Abramowitz and I.~A. Stegun.
\newblock {\em Handbook of mathematical functions with formulas, graphs, and
  mathematical tables}, volume~55 of {\em National Bureau of Standards Applied
  Mathematics Series}.
\newblock For sale by the Superintendent of Documents, U.S. Government Printing
  Office, Washington, D.C., 1964.

\bibitem{angel-perco}
O.~Angel.
\newblock Growth and percolation on the uniform infinite planar triangulation.
\newblock {\em Geom. Funct. Anal.}, 13(5):935--974, 2003.

\bibitem{banderier-maps}
C.~Banderier, P.~Flajolet, G.~Schaeffer, and M.~Soria.
\newblock Random maps, coalescing saddles, singularity analysis, and {A}iry
  phenomena.
\newblock {\em Random Structures Algorithms}, 19(3-4):194--246, 2001.

\bibitem{baxter-dichromatic}
R.~J. Baxter.
\newblock Dichromatic polynomials and {P}otts models summed over rooted maps.
\newblock {\em Ann. Comb.}, 5(1):17--36, 2001.

\bibitem{burris}
J.~P. Bell, S.~N. Burris, and K.~A. Yeats.
\newblock Characteristic points of recursive systems.
\newblock {\em Electron. J. Combin.}, 17(1):Research Paper 121, 34, 2010.

\bibitem{bender-surface-I}
E.~A. Bender and E.~R. Canfield.
\newblock The asymptotic number of rooted maps on a surface.
\newblock {\em J. Combin. Theory Ser. A}, 43(2):244--257, 1986.

\bibitem{bernardi-tutte}
O.~Bernardi.
\newblock A characterization of the {T}utte polynomial via combinatorial
  embeddings.
\newblock {\em Ann. Comb.}, 12(2):139--153, 2008.

\bibitem{bernardi-mbm-de}
O.~Bernardi and M.~Bousquet-M\'elou.
\newblock The {P}otts model on planar maps.
\newblock In preparation.

\bibitem{bernardi-mbm}
O.~Bernardi and M.~Bousquet-M{\'e}lou.
\newblock Counting colored planar maps: algebraicity results.
\newblock {\em J. Combin. Theory Ser. B}, 101(5):315--377, 2011.

\bibitem{bernardi-perco}
O.~Bernardi, N.~Curien, and G.~Miermont.
\newblock A nested loop approach to percolation on random triangulations.
\newblock In preparation.

\bibitem{bernardi-fusy}
O.~Bernardi and {\'E}.~Fusy.
\newblock A bijection for triangulations, quadrangulations, pentagulations,
  etc.
\newblock {\em J. Combin. Theory Ser. A}, 119(1):218--244, 2012.

\bibitem{Bollobas:Tutte-poly}
B.~Bollob{\'a}s.
\newblock {\em Modern graph theory}, volume 184 of {\em Graduate Texts in
  Mathematics}.
\newblock Springer-Verlag, New York, 1998.

\bibitem{borot2}
G.~Borot, J.~Bouttier, and E.~Guitter.
\newblock More on the ${O}(n)$ model on random maps via nested loops: loops
  with bending energy.
\newblock {\em J. Phys. A: Math. Theor.}, 45:275206, 2012.
\newblock ArXiv:1202.5521.

\bibitem{borot1}
G.~Borot, J.~Bouttier, and E.~Guitter.
\newblock A recursive approach to the ${O}(n)$ model on random maps via nested
  loops.
\newblock {\em J. Phys. A: Math. Theor.}, 45:045002, 2012.
\newblock ArXiv:1106.0153.

\bibitem{borot3}
G.~Borot, J.~Bouttier, and E.~Guitter.
\newblock Loop models on random maps via nested loops: case of domain symmetry
  breaking and application to the {P}otts model.
\newblock {\em J. Phys. A}, to appear.
\newblock ArXiv:1207.4878.

\bibitem{BK87}
D.~V. Boulatov and V.~A. Kazakov.
\newblock The {I}sing model on a random planar lattice: the structure of the
  phase transition and the exact critical exponents.
\newblock {\em Phys. Lett. B}, 186(3-4):379--384, 1987.

\bibitem{mbm-survey}
M.~Bousquet-M{\'e}lou.
\newblock Counting planar maps, coloured or uncoloured.
\newblock In {\em Surveys in combinatorics 2011}, volume 392 of {\em London
  Math. Soc. Lecture Note Ser.}, pages 1--49. Cambridge Univ. Press, Cambridge,
  2011.

\bibitem{mbm-schaeffer-ising}
M.~Bousquet-M{\'e}lou and G.~Schaeffer.
\newblock The degree distribution of bipartite planar maps: applications to the
  {I}sing model.
\newblock In K.~Eriksson and S.~Linusson, editors, {\em Formal Power Series and
  Algebraic Combinatorics}, pages 312--323, Vadstena, Sweden, 2003.
\newblock ArXiv math.CO/0211070.

\bibitem{BDG-blocked}
J.~Bouttier, P.~Di~Francesco, and E.~Guitter.
\newblock Blocked edges on {E}ulerian maps and mobiles: application to spanning
  trees, hard particles and the {I}sing model.
\newblock {\em J. Phys. A}, 40(27):7411--7440, 2007.

\bibitem{bdg2002}
J.~Bouttier, P.~Di Francesco, and E.~Guitter.
\newblock Census of planar maps: from the one-matrix model solution to a
  combinatorial proof.
\newblock {\em Nuclear Physics B}, 645(3):477 -- 499, 2002.

\bibitem{bg-continuedfractions}
J.~Bouttier and E.~Guitter.
\newblock Planar maps and continued fractions.
\newblock {\em Comm. Math. Phys.}, 309(3):623--662, 2012.

\bibitem{BIPZ}
E.~Br{\'e}zin, C.~Itzykson, G.~Parisi, and J.~B. Zuber.
\newblock Planar diagrams.
\newblock {\em Comm. Math. Phys.}, 59(1):35--51, 1978.

\bibitem{sportiello}
S.~Caracciolo and A.~Sportiello.
\newblock Spanning forests on random planar lattices.
\newblock {\em J. Stat. Phys.}, 135(5-6):1063--1104, 2009.

\bibitem{chapuy-marcus-schaeffer}
G.~Chapuy, M.~Marcus, and G.~Schaeffer.
\newblock A bijection for rooted maps on orientable surfaces.
\newblock {\em SIAM J. Discrete Math.}, 23(3):1587--1611, 2009.

\bibitem{cori-borgne}
R.~Cori and Y.~Le~Borgne.
\newblock The sand-pile model and {T}utte polynomials.
\newblock {\em Adv. in Appl. Math.}, 30(1-2):44--52, 2003.
\newblock Formal power series and algebraic combinatorics (Scottsdale, AZ,
  2001).

\bibitem{daul}
J.-M. Daul.
\newblock $q$-{S}tate {Potts} model on a random planar lattice.
\newblock ArXiv:hep-th/9502014, 1995.

\bibitem{DFGZJ}
P.~Di~Francesco, P.~Ginsparg, and J.~Zinn-Justin.
\newblock {$2$}{D} gravity and random matrices.
\newblock {\em Phys. Rep.}, 254(1-2), 1995.
\newblock 133 pp.

\bibitem{drmota-systems}
M.~Drmota.
\newblock Systems of functional equations.
\newblock {\em Random Structures Algorithms}, 10(1-2):103--124, 1997.

\bibitem{DK88}
B.~Duplantier and I.~Kostov.
\newblock Conformal spectra of polymers on a random surface.
\newblock {\em Phys. Rev. Lett.}, 61(13):1433--1437, 1988.

\bibitem{eynard-bonnet-potts}
B.~Eynard and G.~Bonnet.
\newblock The {P}otts-{$q$} random matrix model: loop equations, critical
  exponents, and rational case.
\newblock {\em Phys. Lett. B}, 463(2-4):273--279, 1999.

\bibitem{flajolet-sedgewick}
P.~Flajolet and R.~Sedgewick.
\newblock {\em Analytic combinatorics}.
\newblock Cambridge University Press, Cambridge, 2009.

\bibitem{Fortuin:Tutte=Potts}
C.~M. Fortuin and P.~W. Kasteleyn.
\newblock On the random cluster model: {I}. {I}ntroduction and relation to
  other models.
\newblock {\em Physica}, 57:536--564, 1972.

\bibitem{guionnet-jones}
A.~Guionnet, V.~F.~R. Jones, D.~Shlyakhtenko, and P.~Zinn-Justin.
\newblock Loop {M}odels, {R}andom {M}atrices and {P}lanar {A}lgebras.
\newblock {\em Comm. Math. Phys.}, 316(1):45--97, 2012.

\bibitem{Ka86}
V.~A. Kazakov.
\newblock Ising model on a dynamical planar random lattice: exact solution.
\newblock {\em Phys. Lett. A}, 119(3):140--144, 1986.

\bibitem{lang}
S.~Lang.
\newblock {\em Algebra}.
\newblock Addison-Wesley Publishing Co., Inc., Reading, Mass., 1965.

\bibitem{le-gall-miermont}
J.-F. Le~Gall and G.~Miermont.
\newblock Scaling limits of random planar maps with large faces.
\newblock {\em Ann. Probab.}, 39(1):1--69, 2011.

\bibitem{lipshitz-df}
L.~Lipshitz.
\newblock D-finite power series.
\newblock {\em J. Algebra}, 122:353--373, 1989.

\bibitem{marckert-miermont}
J.-F. Marckert and G.~Miermont.
\newblock Invariance principles for random bipartite planar maps.
\newblock {\em Ann. Probab.}, 35(5):1642--1705, 2007.

\bibitem{merino}
C.~Merino~L{\'o}pez.
\newblock Chip firing and the {T}utte polynomial.
\newblock {\em Ann. Comb.}, 1(3):253--259, 1997.

\bibitem{mullin-boisees}
R.~C. Mullin.
\newblock On the enumeration of tree-rooted maps.
\newblock {\em Canad. J. Math.}, 19:174--183, 1967.

\bibitem{Sch97}
G.~Schaeffer.
\newblock Bijective census and random generation of {E}ulerian planar maps with
  prescribed vertex degrees.
\newblock {\em Electron. J. Combin.}, 4(1):Research Paper 20, 14 pp.\
  (electronic), 1997.

\bibitem{stanley-vol-2}
R.~P. Stanley.
\newblock {\em Enumerative combinatorics. {V}ol. 2}, volume~62 of {\em
  Cambridge Studies in Advanced Mathematics}.
\newblock Cambridge University Press, Cambridge, 1999.

\bibitem{tutte-dichromate}
W.~T. Tutte.
\newblock A contribution to the theory of chromatic polynomials.
\newblock {\em Canadian J. Math.}, 6:80--91, 1954.

\bibitem{tutte-triangulations}
W.~T. Tutte.
\newblock A census of planar triangulations.
\newblock {\em Canad. J. Math.}, 14:21--38, 1962.

\bibitem{tutte-census-maps}
W.~T. Tutte.
\newblock A census of planar maps.
\newblock {\em Canad. J. Math.}, 15:249--271, 1963.

\bibitem{lambda12}
W.~T. Tutte.
\newblock Chromatic sums for rooted planar triangulations: the cases {$\lambda
  =1$} and {$\lambda =2$}.
\newblock {\em Canad. J. Math.}, 25:426--447, 1973.

\bibitem{tutte-differential}
W.~T. Tutte.
\newblock Map-colourings and differential equations.
\newblock In {\em Progress in graph theory ({W}aterloo, {O}nt., 1982)}, pages
  477--485. Academic Press, Toronto, ON, 1984.

\bibitem{welsh-merino}
D.~J.~A. Welsh and C.~Merino.
\newblock The {P}otts model and the {T}utte polynomial.
\newblock {\em J. Math. Phys.}, 41(3):1127--1152, 2000.

\bibitem{zinn-justin-dilute-potts}
P.~Zinn-Justin.
\newblock The dilute {P}otts model on random surfaces.
\newblock {\em J. Statist. Phys.}, 98(1-2):j10--264, 2000.

\end{thebibliography}

\end{document}